\documentclass[leqno,11pt]{amsart}

\usepackage{amssymb, amsmath,amsthm}
\usepackage{amssymb,amsfonts}
\usepackage{amsmath,latexsym}
\usepackage{pdfsync}
\usepackage{bbm}
\usepackage{color}
\usepackage[dvipsnames]{xcolor}
\usepackage{cancel}
\usepackage{ulem}
\usepackage{graphicx, color, multicol, float}

\newtheorem{theo}{Theorem}[section]
\newtheorem{prop}[theo]{Proposition}
\newtheorem{lemme}[theo]{Lemma}
\newtheorem{defin}[theo]{Definition}

\newtheorem {remark}[theo] {Remark}

\numberwithin{equation}{section}

\usepackage{verbatim}

\newcommand{\tsl}{\textsl}

\newcommand{\vtheta}{\vartheta}

\newcommand{\mc}{\mathcal}

\renewcommand{\Re}{{\rm Re}\,}
\renewcommand{\Im}{{\rm Im}\,}

\newcommand{\divh}{{\rm div}_\h}

\newcommand{\curl}{{\rm curl}\,}

\newcommand{\Id}{{\rm Id}\,}
\newcommand{\supp}{{\rm Supp}\,}

\newcommand{\oline}{\overline}

\makeatletter
\@namedef{subjclassname@2020}{\textup{2020} Mathematics Subject Classification}
\makeatother

\allowdisplaybreaks

\begin{document}


\makeatletter
\def\sommaire{\@restonecolfalse\if@twocolumn\@restonecoltrue\onecolumn
\fi\chapter*{Sommaire\@mkboth{SOMMAIRE}{SOMMAIRE}}
  \@starttoc{toc}\if@restonecol\twocolumn\fi}
\makeatother

\makeatletter
\def\thebibliographie#1{\chapter*{Bibliographie\@mkboth
  {BIBLIOGRAPHIE}{BI{\rm BL}IOGRAPHIE}}\list
  {[\arabic{enumi}]}{\settowidth\labelwidth{[#1]}\leftmargin\labelwidth
  \advance\leftmargin\labelsep
  \usecounter{enumi}}
  \def\newblock{\hskip .11em plus .33em minus .07em}
  \sloppy\clubpenalty4000\widowpenalty4000
  \sfcode`\.=1000\relax}
\let\endthebibliography=\endlist
\makeatother

\makeatletter
\def\references#1{\section*{R\'ef\'erences\@mkboth
  {R\'EF\'ERENCES}{R\'EF\'ERENCES}}\list
  {[\arabic{enumi}]}{\settowidth\labelwidth{[#1]}\leftmargin\labelwidth
  \advance\leftmargin\labelsep
  \usecounter{enumi}}
  \def\newblock{\hskip .11em plus .33em minus .07em}
  \sloppy\clubpenalty4000\widowpenalty4000
  \sfcode`\.=1000\relax}

\let\endthebibliography=\endlist
\makeatother

\def\lunloc#1#2{L^1_{loc}(#1 ; #2)}
\def\bornva#1#2{L^\infty(#1 ; #2)}
\def\bornlocva#1#2{L_{loc}^\infty(#1 ;\penalty-100{#2})}
\def\integ#1#2#3#4{\int_{#1}^{#2}#3d#4}
\def\reel#1{\R^#1}
\def\norm#1#2{\|#1\|_{#2}}
\def\normsup#1{\|#1\|_{L^\infty}}
\def\normld#1{\|#1\|_{L^2}}
\def\nsob#1#2{|#1|_{#2}}
\def\normbornva#1#2#3{\|#1\|_{L^\infty({#2};{#3})}}
\def\refer#1{~\ref{#1}}
\def\refeq#1{~(\ref{#1})}
\def\ccite#1{~\cite{#1}}
\def\pagerefer#1{page~\pageref{#1}}
\def\referloin#1{~\ref{#1} page~\pageref{#1}}
\def\refeqloin#1{~(\ref{#1}) page~\pageref{#1}}
\def\suite#1#2#3{(#1_{#2})_{#2\in {#3}}}
\def\ssuite#1#2#3{\hbox{suite}\ (#1_{#2})_{#2\in {#3}}}
\def\longformule#1#2{
\displaylines{
\qquad{#1}
\hfill\cr
\hfill {#2}
\qquad\cr
}
}
\def\inte#1{
\displaystyle\mathop{#1\kern0pt}^\circ
}
\def\sumetage#1#2{\sum_{\substack{{#1}\\{#2}}}}
\def\limetage#1#2{\lim_{\substack{{#1}\\{#2}}}}
\def\infetage#1#2{\inf_{\substack{{#1}\\{#2}}}}
\def\maxetage#1#2{\max_{\substack{{#1}\\{#2}}}}
\def\supetage#1#2{\sup_{\substack{{#1}\\{#2}}}}
\def\prodetage#1#2{\prod_{\substack{{#1}\\{#2}}}}

\def\convm#1{\mathop{\star}\limits_{#1}
}
\def\vect#1{
\overrightarrow{#1}
}
\def\Hd#1{{\mathcal H}^{d/2+1}_{1,{#1}}}

\def\derconv#1{\partial_t#1 + v\cdot\nabla #1}
\def\esptourb{\sigma + L^2(\R^2;\R^2)}
\def\tourb{tour\bil\-lon}



\newcommand{\beq}{\begin{equation}}
\newcommand{\eeq}{\end{equation}}
\newcommand{\ben}{\begin{eqnarray}}
\newcommand{\een}{\end{eqnarray}}
\newcommand{\beno}{\begin{eqnarray*}}
\newcommand{\eeno}{\end{eqnarray*}}

\let\al=\alpha
\let\b=\beta
\let\g=\gamma
\let\d=\delta
\let\e=\varepsilon
\let \z=\zeta
\let\lam=\lambda
\let\r=\rho
\let\s=\sigma
\let\f=\phi
\let\vf =\varphi
\let\p=\psi
\let\om=\omega
\let\G= \Gamma
\let\D=\Delta
\let\Lam=\Lambda
\let\S=\Sigma
\let\Om=\Omega
\let\wt=\widetilde
\let\wh=\widehat
\let\convf=\leftharpoonup
\let\tri\triangle

\def\cA{{\mathcal A}}
\def\cB{{\mathcal B}}
\def\cC{{\mathcal C}}
\def\cD{{\mathcal D}}
\def\cE{{\mathcal E}}
\def\cF{{\mathcal F}}
\def\cG{{\mathcal G}}
\def\cH{{\mathcal H}}
\def\cI{{\mathcal I}}
\def\cJ{{\mathcal J}}
\def\cK{{\mathcal K}}
\def\cL{{\mathcal L}}
\def\cM{{\mathcal M}}
\def\cN{{\mathcal N}}
\def\cO{{\mathcal O}}
\def\cP{{\mathcal P}}
\def\cQ{{\mathcal Q}}
\def\cR{{\mathcal R}}
\def\cS{{\mathcal S}}
\def\cT{{\mathcal T}}
\def\cU{{\mathcal U}}
\def\cV{{\mathcal V}}
\def\cW{{\mathcal W}}
\def\cX{{\mathcal X}}
\def\cY{{\mathcal Y}}
\def\cZ{{\mathcal Z}}

\def\virgp{\raise 2pt\hbox{,}}
\def\cdotpv{\raise 2pt\hbox{;}}
\def\eqdef{\buildrel\hbox{\footnotesize d\'ef}\over =}
\def\eqdefa{\buildrel\hbox{\footnotesize def}\over =}
\def\Id{\mathop{\rm Id}\nolimits}
\def\limf{\mathop{\rm limf}\limits}
\def\limfst{\mathop{\rm limf\star}\limits}
\def\sgn{\mathop{\rm sgn}\nolimits}
\def\RE{\mathop{\Re e}\nolimits}
\def\IM{\mathop{\Im m}\nolimits}
\def\im {\mathop{\rm Im}\nolimits}
\def\Sp{\mathop{\rm Sp}\nolimits}
\def\C{\mathop{\mathbb C\kern 0pt}\nolimits}
\def\DD{\mathop{\mathbb D\kern 0pt}\nolimits}
\def\EE{\mathop{\mathbb E\kern 0pt}\nolimits}
\def\K{\mathop{\mathbb K\kern 0pt}\nolimits}
\def\N{\mathop{\mathbb  N\kern 0pt}\nolimits}
\def\Q{\mathop{\mathbb  Q\kern 0pt}\nolimits}
\def\R{{\mathop{\mathbb R\kern 0pt}\nolimits}}
\def\SS{\mathop{\mathbb  S\kern 0pt}\nolimits}
\def\St{\mathop{\mathbb  S\kern 0pt}\nolimits}
\def\Z{\mathop{\mathbb  Z\kern 0pt}\nolimits}
\def\ZZ{{\mathop{\mathbb  Z\kern 0pt}\nolimits}}
\def\H{{\mathop{{\mathbb  H\kern 0pt}}\nolimits}}
\def\PP{\mathop{\mathbb P\kern 0pt}\nolimits}
\def\TT{\mathop{\mathbb T\kern 0pt}\nolimits}
 \def\L {{\rm L}}

\def\h {{\rm h}}
\def\top{{\rm surf}}
\def\surf{{\rm surf}}
\def\bot{{\rm bot}}
\def\app {{\rm app}}
\def\pol {{\rm p}}
\def\rad{{\rho }}
\def\v{{\rm v}}
\def\zrm{{\rm z}}
\def\BL{{\mathcal B \mathcal L}}
\def\curl{{\rm curl}}
\def\L {{\rm L}} 

\newcommand{\ds}{\displaystyle}
\newcommand{\la}{\lambda}
\newcommand{\hn}{{\bf H}^n}
\newcommand{\hnn}{{\mathbf H}^{n'}}
\newcommand{\ulzs}{u^\lam_{\zrm ,s}}
\def\bes#1#2#3{{B^{#1}_{#2,#3}}}
\def\pbes#1#2#3{{\dot B^{#1}_{#2,#3}}}
\newcommand{\ppd}{\dot{\Delta}}
\def\psob#1{{\dot H^{#1}}}
\def\pc#1{{\dot C^{#1}}}
\newcommand{\Hl}{{{\mathcal  H}_\lam}}
\newcommand{\fal}{F_{\al, \lam}}
\newcommand{\Dh}{\Delta_{{\mathbf H}^n}}
\newcommand{\car}{{\mathbf 1}}
\newcommand{\X}{{\mathcal  X}}
\newcommand{\fgl}{F_{\g, \lam}}
\newcommand{\wU}{{\widetilde U  }}
\newcommand{\andf}{\quad\hbox{and}\quad}
\newcommand{\with}{\quad\hbox{with}\quad}

\def\beginproof {\noindent {\it Proof. }}
\def\endproof {\hfill $\Box$}


\def\vp{{\underline v}}
\def\presspO{{{\underline p}_0}}
\def\presspun{{{\underline p}_1}}
\def\wp{{\underline w}}
\def\wpe{{\underline w}^{\e}}
\def\vapp{v_{app}^\e}
\def\vapph{v_{app}^{\e,h}}
\def\vbar{\overline v}
\def\barEE{\underline{\EE}}


\def\demo{d\'e\-mons\-tra\-tion}
\def\dive{\mathop{\rm div}\nolimits}
\def\curl{\mathop{\rm curl}\nolimits}
\def\cdv{champ de vec\-teurs}
\def\cdvs{champs de vec\-teurs}
\def\cdvdivn{champ de vec\-teurs de diver\-gence nul\-le}
\def\cdvdivns{champs de vec\-teurs de diver\-gence
nul\-le}
\def\stp{stric\-te\-ment po\-si\-tif}
\def\stpe{stric\-te\-ment po\-si\-ti\-ve}
\def\reelnonentier{\R\setminus{\bf N}}
\def\qq{pour tout\ }
\def\qqe{pour toute\ }
\def\Supp{\mathop{\rm Supp}\nolimits\ }
\def\coinfty{in\-d\'e\-fi\-ni\-ment
dif\-f\'e\-ren\-tia\-ble \`a sup\-port com\-pact}
\def\coinftys{in\-d\'e\-fi\-ni\-ment
dif\-f\'e\-ren\-tia\-bles \`a sup\-port com\-pact}
\def\cinfty{in\-d\'e\-fi\-ni\-ment
dif\-f\'e\-ren\-tia\-ble}
\def\opd{op\'e\-ra\-teur pseu\-do-dif\-f\'e\-ren\tiel}
\def\opds{op\'e\-ra\-teurs pseu\-do-dif\-f\'e\-ren\-tiels}
\def\edps{\'equa\-tions aux d\'e\-ri\-v\'ees
par\-tiel\-les}
\def\edp{\'equa\-tion aux d\'e\-ri\-v\'ees
par\-tiel\-les}
\def\edpnl{\'equa\-tion aux d\'e\-ri\-v\'ees
par\-tiel\-les non li\-n\'e\-ai\-re}
\def\edpnls{\'equa\-tions aux d\'e\-ri\-v\'ees
par\-tiel\-les non li\-n\'e\-ai\-res}
\def\ets{espace topologique s\'epar\'e}
\def\ssi{si et seulement si}

\def\pde{partial differential equation}
\def\iff{if and only if}
\def\stpa{strictly positive}
\def\ode{ordinary differential equation}
\def\coinftya{compactly supported smooth}

\def\app {{\rm app}}
\def\zrm {{\rm z}}
\def\rmint{{\rm int}}



  \title[Ekman boundary layers in a  domain with topography (draft)]
  {Ekman boundary layers in a   domain with topography}

\author[J.-Y. Chemin]{Jean-Yves  Chemin}
\address[J.-Y. Chemin]%
{Universite Claude Bernard Lyon 1, CNRS, Ecole Centrale de Lyon, INSA Lyon, Université Jean Monnet, ICJ UMR5208,
69622 Villeurbanne, France.}
\email{chemin@math.univ-lyon1.fr}

\author[F. Fanelli]{Francesco Fanelli}
\address[F. Fanelli]{Basque Center for Applied Mathematics, Alameda de Mazarredo 14, E-48009 Bilbao, Basque Country, Spain \\ and 
Ikerbasque -- Basque Foundation for Science, Plaza Euskadi 5, E-48009 Bilbao, Basque Country, Spain \\ and
Universit\'e Claude Bernard Lyon 1, ICJ UMR5208, F-69622 Villeurbanne, France.}
\email{ffanelli@bcamath.org}

\author[I. Gallagher]{Isabelle Gallagher}
\address[I. Gallagher]%
{DMA, \'Ecole normale sup\'erieure, CNRS, PSL Research University, 75005 Paris 
 \\
and UFR de math\'ematiques, Universit\'e Paris--Cit\'e, 75013 Paris, France.}
\email{gallagher@math.ens.fr}



\begin{abstract}
We investigate the asymptotic behaviour  of fast rotating incompressible  fluids with vanishing viscosity, in a  {three dimensional} domain with   topography including the case of land area.  Assuming the initial data is well-prepared, we  prove a convergence theorem     of the velocity fields to a two-dimensional vector field   solving     a linear, damped ordinary differential equation.
The proof is based on a  weak-strong uniqueness argument, combined
with an abstract result implying that the weak convergence of a family
of weak solutions to the Navier-Stokes-Coriolis system can be translated into a form of uniform-in-time convergence.
This argument yields 
strong convergence of the velocity fields, without a precise rate though.
\end{abstract}
\date \today

 \maketitle

\section{Introduction and statement of the main result}
 
The purpose of this text  is the  study of the asymptotic behavior  of  a fast rotating incompressible fluid  
moving in a three dimensional domain having a non-flat bottom. This can be seen as a rough model to describe ocean dynamics
in presence of topography: our problem formulation, which we are now going to present, is devised according to this point of view.

\subsection{Setting of the problem} \label{ss:setting}
We first define the domain in which the fluid evolves. We consider a smooth bounded and real valued function~$\vf$ on~$\R^2$ and we define the ocean area as
\beq
\label {defin_domain_general}
\Omega_\vf\eqdefa	 \bigl\{(x_\h,z)\in \R^2 \times \R\,/\ \vf(x_\h)>0\ \hbox{and}\ z\in ]-\vf(x_\h),0[\bigr\} \, .
\eeq
We assume~$\Omega_\vf$ to be connected. We also introduce the surface of the ocean as the set
$$
\cO_\vf\eqdefa \bigl\{x_\h \in \R^2\,/\ \vf(x_\h)> 0\bigr\}\,,
$$
together with  the land and shore areas
$$
\cL_\vf\eqdefa \R^2 \setminus\cO_\vf
\quad \mbox{and} \quad
\cS_\vf\eqdefa \vf^{-1} \big(\{0\}\big) \, .
$$
Anticipating what will be formalized in Subsection \ref{ss:result}, the land area should be thought of as an island, as depicted in Figure \ref{fig:island} below.
 \vspace{-0.3cm}
 \begin{figure}[H] 
\includegraphics[width=2in]{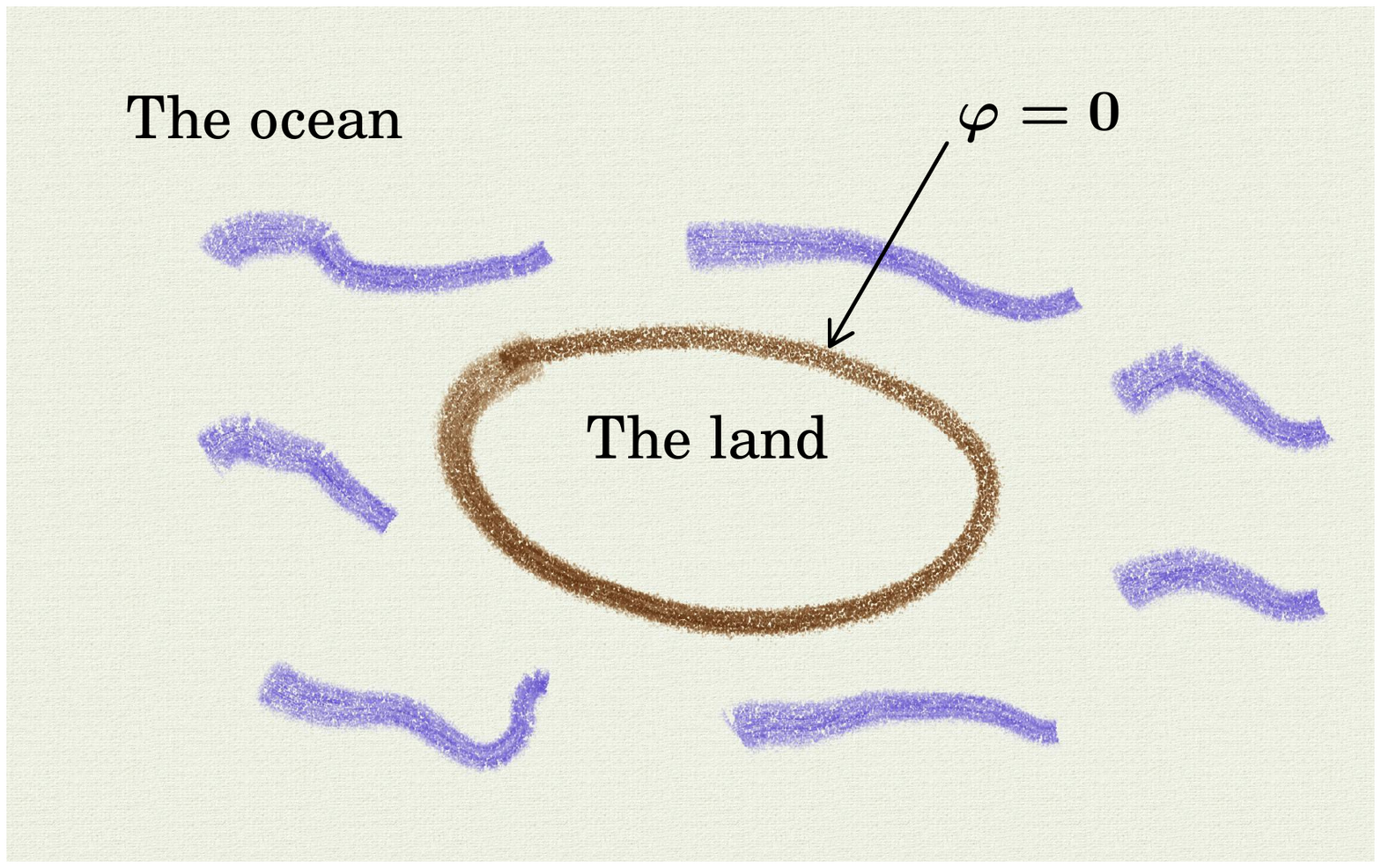} \hspace{2cm} 
\includegraphics[width=1.9in]{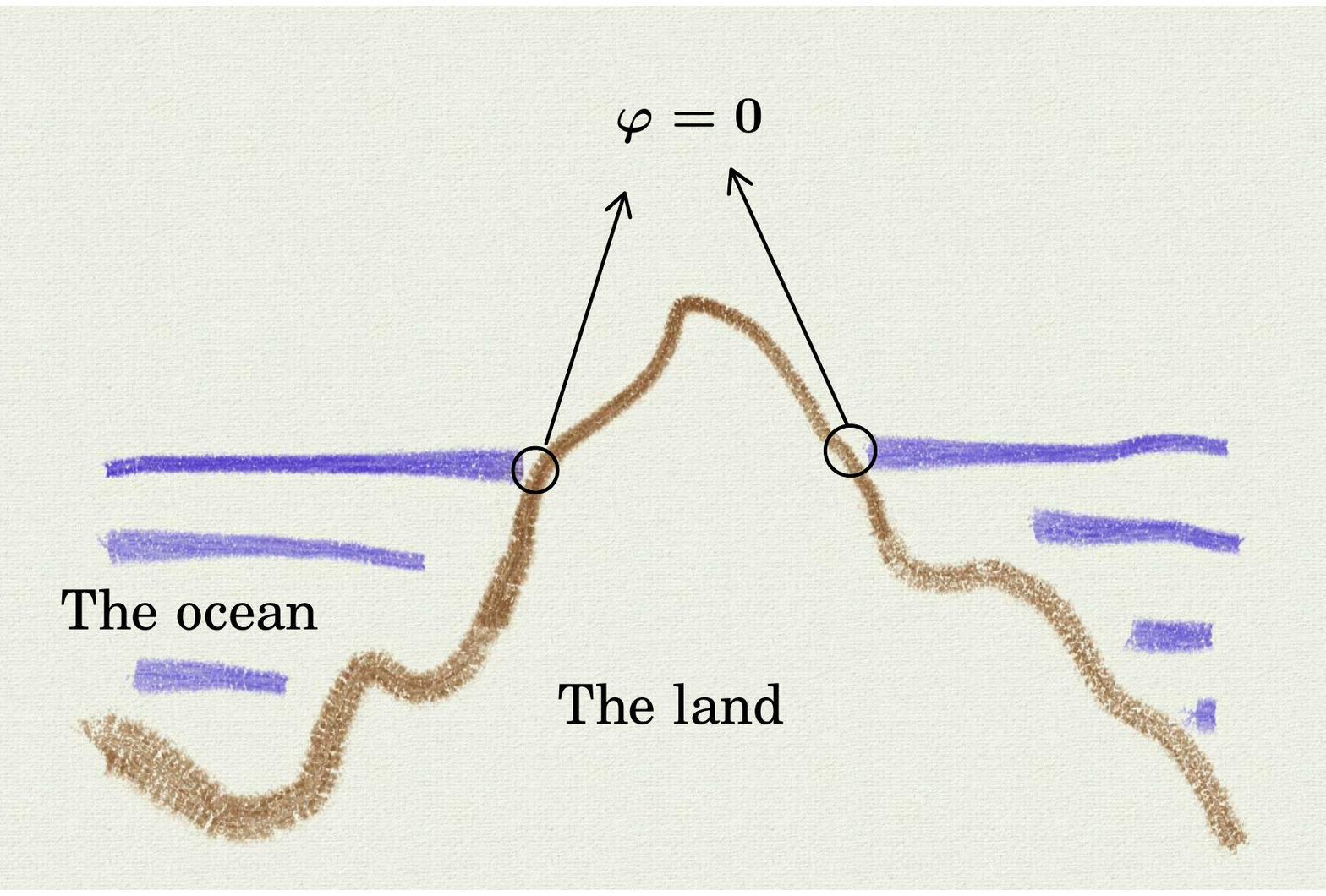} 
\caption{\small In brown, the land: top view (on the left) and side section (on the right).}
\label{fig:island}
\end{figure}



Given two positive real numbers~$\beta$ (fixed) and~$\e$ (the Rossby number, which will tend to zero),
we consider the  following {incompressible} Navier-Stokes-Coriolis system  in~$\Om_\vf$:
\[
\left\{
\begin{array}{c} 
\ds \partial_t u_\e - \e\b\D u_\e  +\frac 1 \e e_{\zrm} \wedge u_\e +u_\e\cdot \nabla u_\e =-\frac 1 \e \nabla p_\e\,,\\
\ds  {u_\e}_{|t=0} = u_0\,,\\
\ds \dive u_\e =0\andf {u_\e}_{|\partial\Om_\vf }=0\,  .
\end{array}
\right.
\leqno(NSC_\e)
\]
Here~$ u_\e$ is a three-component vector field representing the velocity of the fluid, and~$p_\e$ is a function representing its pressure.
The vector $e_{\zrm}$ is the unit vector directed along the vertical axis, i.e. $e_\zrm\, =\,^t(0,0,1)$, and the notation $a \wedge b$
stands for the usual external product in $\R^3$ of two vectors $a$ and $b$. The term $e_\zrm\wedge u_\e$ encodes the action of the Coriolis force
on the fluid; the factor $\e^{-1}$ in front of it is a typical scaling in oceanography.
{We refer to 
the physics books~\cite{CR-B, Ped} for more details (see also Part I of \cite{cdggbook})}.

{Throughout this paper}, we are going to adopt the following notations. For
any three-component vector field~$u$ we define the two-dimensional vector field $u_\h$ as the horizontal component
of~$u$, namely $u_\h=  (u^1,u^2 )$. The ``vertical'' component $u^3$ will be often denoted as~$u^\zrm$. Correspondingly, we define the horizontal gradient~$\nabla_\h= (\partial_{1},\partial_{2})$
and, for a {two-component} vector field~$v_\h=  (v^1,v^2 )$, its horizontal divergence as~$\divh v_\h = \partial_{1}v^1 + \partial_{2}v^2$.

Let us briefly recall the theory of  turbulent solutions\footnote { {We follow the terminology introduced by J. Leray in his seminal paper\ccite {lerayns} {on the incompressible Navier-Stokes system}. } } for such systems. We first precise the  functional spaces we work~with.
\begin{defin}
\label{vhv'}
{\sl
We shall denote by~$\cV$ the space of vector fields, the
components of which belong to~$H^1_0(\Om_\vf) $, and by~$\cV_\sigma$ the space of divergence free vector fields
in~$\cV$.  The closure of ~$\cV_\sigma$ in~$ L^2(\Om_\vf)$ will be denoted~$\cH$.
Finally, we shall denote by~$\cV'_\s$ the dual space of~$\cV_\sigma$. 
}
\end{defin}

{Next, we recall the definition of turbulent solutions to $(NSC_\e)$.}

\begin{defin}
\label {definweaksolbd}
{\sl
We shall say that $u_\e$ is a weak solution
of~$(NSC_\e)$ on~$\R^+\times\Om_\vf$ with an initial
data~$u_0$ in~$\cH$  if and only if $u_\e$ belongs to
the space
\[
C(\R^+;\cV'_\s )\cap L_{\rm loc}^{\infty}(\R^+;\cH)\cap  L^2_{\rm loc}(\R^+;\cV_\sigma)
\]
and, for any~$\Psi$ in~$C^1(\R^+;\cV_\sigma)$, the vector field~$ u_\e$
satisfies the following identity, for all~$t \geq 0$:
\begin{equation}\label{defin1.2_eq}
\begin{aligned}  &\int_{\Omega_\vf}( u_\e \cdot \Psi ) (t,x)\, dx  + 
\int_0^t\! \int_{\Omega_\vf} \biggl( \b\e\nabla u_\e : \nabla \Psi
- u_\e \cdot \Bigl( \partial_t \Psi +  \frac 1 \e e_{\zrm} \wedge \Psi
\Bigr)\biggr)(t',x)\, dx dt'\\
&-\int_0^t \int _{\Om_\vf}    \bigl(u_\e \otimes u_\e : \nabla \Psi\bigr)(t',x)\, dxdt' = 
\int_{\Omega_\vf}   u_0(x) \cdot \Psi(0,x) \, dx\,,
\with\\
&\nabla u_\e:\nabla \Psi \eqdefa \sum_{j,k = 1}^3\partial_{j}u_\e^k\partial_{j}\Psi^k
\quad\hbox{and}\quad
u_\e\otimes u_\e :\nabla \Psi \eqdefa \sum_{j,k=1}^3 u_\e^j u_\e^k \partial_{j}\Psi^{k}\, .
\end{aligned}
\end{equation}
}
\end{defin}

Let us recall a classical theorem of existence of turbulent solutions for such a system  (see for instance\ccite {cdggbook} {for its proof}).
\begin{theo}[\cite{cdggbook}]
\label{theoleraybd}
{\sl
Given~$u_0$ a vector field in~$\cH$, there exists a global weak solution $u_\e$ to~$(NSC_\e)$ in the sense
of Definition\refer  {definweaksolbd}.
Moreover, this solution satisfies the following energy inequality, for all~$ t\geq 0$:
\beq
\label {theoleraybdeq1}
\frac{1}{2} \int_{\Omega_\vf}  |u_\e(t,x)|^2 dx  +\e\b \int_0^t \int_{\Omega_\vf} |\nabla u_\e (t',x)|^2 dx dt'
 \leq    \frac{1}{2}
     \int_{\Omega_\vf}  |u_0(x)|^2 dx\, .
\eeq
}
\end{theo}

We shall focus here on the case of well-prepared initial data~$u_0$, meaning that it is a divergence free vector field in~$\cH$ and satisfies~$e_{\zrm} \wedge u_0 =\nabla p_0$ for some function~$p_0$.
Considering a collection of turbulent solutions~$(u_\e)_{0<\e\leq \e_0}$ in the sense of the above definition, what can be said about its {asymptotic} behaviour in the limit when~$\e$ tends to~$0$?

\subsection{{The flat case, and related studies}} \label{ss:flat}
In the case when there is no topography, i.e.  the case when~$\vf(x_\h)\equiv H
$, where
~$H$ is a given positive real number, {the previous assumption on $u_0$} amounts to assuming that it only depends on  the horizontal variables.
{In this situation, the asymptotic dynamics in the limit of vanishing $\e$ is described by the}
following theorem, proved by E. Grenier and N. Masmoudi in\ccite {grenier_masmoudi}.
\begin{theo}[\cite{grenier_masmoudi}]
\label {grenier_masmoudi} 
{\sl
 {Assume that~$\vf(x_\h)\equiv 
 H>0$ and that~$u_0 =  (  u_{0,\h} ,0)$ is a divergence free vector field in~$\cH$ which does not depend on~$z$. If~$u_0$ is small enough in~$L^\infty(\R ^2)$}, then
$$
\lim_{\e\rightarrow 0} u_\e = \overline u \qquad \hbox{in}\quad  L^\infty\bigl(\R^+;L^2(\R^2 \times ]-h,0[)\bigr) \, , 
$$
where~$\oline u = (\overline u_\h, 0)$ and $\oline u_\h$ is the solution of the following 2D damped Euler equation on~$\R^2$:
\begin{equation}
\label {dampedEuler}
\left\{
\begin{array}{c}
\ds \partial_t \overline u_\h +
\lambda  \overline u_\h +\overline u_\h\cdot \nabla \overline u_\h =-\nabla_\h\overline p\\\\
\dive_\h \overline u_\h =0\andf{ \overline u_\h}_{|t=0} = u_{0,\h} 
\end {array}\right. \with
\lambda\eqdefa \frac {\sqrt {2\b}} 
H \,\cdotp
\end{equation} 
}
\end{theo}

Let us make some comments on this statement,
{which will be interesting to compare with our main result (see Theorem~\ref{Ekman_topo_radial_0} below).} The first point to notice is that the convergence
of~$u_\e$ to $\oline u$ is strong. Moreover, the limit is a two-dimensional vector field which does not depend on the vertical variable~$z$. Thus, it does not fit with the Dirichlet boundary condition, and this leads to the introduction of correctors called boundary layers: these correctors, denoted~$U_{\rm BL,\e}$, compensate the Dirichlet boundary condition while   tending to~$0$  in~$L^\infty\big(\R^+;L^2(\R^2 \times ]-h,0[)\big)$. It turns out that their energy
$$
\e \int_0^t \|\nabla U_{\rm BL,\e}(t')\|_{L^2}^2\, dt'
$$
does not tend to~$0$ and is the source of the damping term $\lambda \oline u_\h$ in the limit  system\refeq {dampedEuler}. This damping is called Ekman pumping and the number~
$\lambda$ is the Ekman pumping coefficient. Let us notice that it is proportional to the inverse of the depth.

\medbreak

{Many works have generalized Theorem~\ref{grenier_masmoudi} in several directions. In \cite{Masm} N. Masmoudi
proved an analogous result in the general case of ill-prepared initial data.
The smallness requirement on the initial data was removed by F. Rousset \cite{Rouss} for well-prepared data, and by F. Rousset and N. Masmoudi \cite{Masm-Rouss}
for the general case.
In \cite{C-D-G-G_ESAIM} the authors considered the case of anisotropic diffusion, where the operator $-\beta\e\Delta$ in $(NSC_\e)$ is replaced by
$-\nu\Delta_\h - \beta\e\partial_z^2$, with $\nu$ a positive constant,
and proved convergence to a $2$-D damped Navier-Stokes equation on $\R^2$ in the framework of ill-prepared data,
by means of dispersive estimates. We refer to the book \cite{cdggbook} for a collection of those results and related proofs.}

{Ekman boundary layers have been studied also in some special cases, for instance when the effect of the Coriolis force becomes degenerate, a
typical situation near the equator, in which case the system $(NSC_\e)$ must be replaced by the so-called $\beta$--plane model (see e.g. \cite{GSR2}).
In~\cite{Dalibard3}, A.-L. Dalibard and L. Saint-Raymond investigated the effect of the Ekman layers on stationary solutions and propagation of Poincar\'e waves
for the $\beta$--plane model, when set in a thin layer.
In \cite{Dalibard2} the same authors considered, instead, the action of a forcing term at the top boundary,
in resonance with the Coriolis force, and studied its effect on the structure
of the solutions to a linearized version of the system $(NSC_\e)$: they exhibited the presence of an additional boundary layer, which coexists with
the Ekman one. 
From an applied standpoint, such a forcing term acting at the top boundary of the domain encodes the effects of the wind stress at the surface of the ocean.
For more results on this situation, we refer to \cite{saint-raymond} by L. Saint-Raymond, who characterized the asymptotic dynamics
of the original non-linear system.}

{We remark that a few studies have been conducted on Ekman boundary layer effects in the context of non-homogeneous fluids, and all  of these
treat the setting of a flat boundary. In the case of compressible flows, paper
\cite{B-D-GV} by D. Bresch, B. Desjardins and D. G\'erard-Varet establishes a (conditional) convergence result for well-prepared initial data,
together with a description
of the Ekman layers; the result was later extended in \cite{B-F-P} to consider a strong stratification regime.
In the context of incompressible fluids, instead, work
\cite{Brav-F} obtained the rigorous derivation of a system encoding the Ekman pumping effect
through a singular limit starting from the density dependent Navier-Stokes-Coriolis system,
set in a thin domain and supplemented with Navier-slip boundary conditions; yet, a precise description
of the structure of the solutions in the boundary layer was elusive within that approach.}

{To conclude this overview, let us mention that Ekman boundary layers are created at the surface and bottom boundaries of the ocean. However,
other bondary layers, the so-called Munk layers, exist in the proximity of the shores, whenever a vertical wall appears at the boundary.
This situation has also been the object of some investigation: we refer for instance to the above mentioned work \cite{B-D-GV} by D. Bresch, B. Desjardins and D. G\'erard-Varet, and to~\cite{Dalibard4} by A.-L. Dalibard and L. Saint-Raymond. In the present work, we avoid the appearance of Munk boundary layers by imposing
that the function $\varphi$ encoding the topography must be smooth and bounded over $\mc O_\varphi$, together with all its derivatives.}

\medbreak
{In contrast with the huge amount of literature related to the study of Ekman boundary layers for fast rotating fluids, less results are
available in the case of a varying bottom. To the best of our knowledge, the first study in this direction was due to D. G\'erard-Varet in~\cite{gerard_varet_2003}.
There, the author considered the framework of small, periodic perturturbations of the flat case, of size $O(\e)$.
The asymptotic study in \cite{gerard_varet_2003} was conducted for well-prepared data and yielded a result similar to
Theorem~\ref{grenier_masmoudi}, yet at the price of higher complications
in the arguments of the proof. We refer to \cite{gerard_varet_2006} by E. Dormy and D. G\'erard-Varet for additional investigations in that setting.
On a different but related context, we mention that, in \cite{Dalibard1}, A.-L. Dalibard and D. G\'erard-Varet addressed
the well-posedness of the non-linear static problem
in the domain~$\big\{x\in\R^3\, / \; \gamma(x_\h)< x^\zrm\big\}$, namely in absence of an upper boundary but for a generic function $\gamma$ of order $O(1)$
(see also \cite{Dal-Prange} by A.-L. Dalibard and C. Prange for a preliminary investigation of the linear problem).}

%

\subsection{Statement of the main result} \label{ss:result}
In this paper we are interested in the case when the topography is not flat (actually nowhere flat). {In relation to the discussion
of the previous subsection, we point out that we will not require here any periodicity, nor symmetry, nor smallness of the function describing the topography.}
{However, the geometric set up must be} chosen in a special way, which we shall describe here and justify later.

The function~$\vf$ introduced in\refeq  {defin_domain_general}, which represents the depth of the ocean, {is assumed to} satisfy the following condition: a function~$F$ exists such that 
\beq
\label {Cond_geo_fund}
\forall x_\h\in \cO_\vf\,, \quad \ |\nabla_\h \vf(x_\h)|^2 =F(\vf(x_\h))\,  .
\eeq
{By choosing $\vf$ as a composition, i.e. $\vf = \f\circ\rho$, we see that it is enough for $F$ to depend only on the function $\rho$. Indeed, if~$x_\h$ is a point where the gradient of~$\vf$  does not vanish, then both $\phi'(\rho(x_\h))$ and $\nabla_\h\rho(x_\h)$ do not vanish either;
then condition \eqref{Cond_geo_fund} yields
\[
|\nabla_\h\rho(x_\h)|^2 = \frac{F(\f(\rho(x_\h)))}{\big(\phi'(\rho(x_\h))\big)^2} = G(\rho(x_\h))\,.
\]
Notice that, in order to write the last equality, it is enough that $F$ depends on $\rho$, as claimed. 
In addition, if $x_\h$ is a point 
as above, then the function~$G$ does not vanish near the value~$\rho(x_\h)$.
Let~$H$ be  a primitive of~$G^{-\frac 12}$ which also has value non~$0$, then~$|\nabla_\h H(\rho(x_\h))|=1$.
In what follows, we will need to distinguish
between the function describing the depth, namely $\vf$, and the one encoding the geometry, namely $\rho$, and this remark will play an important role.
As a conclusion, it is therefore natural to assume that the function~$\vf$ is of the~form}
\beq
\label {Cond_geo_fund_2}
\vf(x_\h)= \f \big(\rho(x_\h)\big)\,, \with  \forall x_\h \in \cO_\vf\,,\ |\nabla_\h\rho(x_\h)|=1\, .
\eeq

Let us present a large class of examples of such topography. We consider a compact convex set~$\G$ of~$\R^2$ such that the boundary  of~$\G$ is a smooth  curve. Let us define the following function:
 $$
 \rho(x_\h)\eqdefa {\bf d}(x_\h,\G)\,,
 $$
 where the distance is the euclidean distance on~$\R^2$. We claim that~$\rho$ is smooth on~$\R^2\setminus\G$ and~that
 \beq
 \label {fod_grad_rho} \forall x_\h \in \R^2\,,\qquad |\nabla_\h \rho  (x_\h)|=1\, .
 \eeq 
Let us justify this. We consider a parametrization~$\g$ of the boundary of~$\G$ such that 
$$
\g\ :\  \R/L\Z \longmapsto \G
$$
  is one to one, smooth	and satisfies~$|\g'(\omega)|=1$   (here~$L$ is the  length of the boundary of~$\G$). Let us consider
the map
$$
X:\quad \left \{\begin {array}{ccl}
(\R/L\Z )\times ]0,\infty[ & \longrightarrow & \R^2\setminus\G\\
 (\omega,\lam) &\longmapsto & \g'(\omega)-\lam (\g'(\omega))^\perp\, .
\end {array}
\right.
$$
The point is that~$X$ is  {one-to-one}  and onto. Moreover, it is smooth and its differential is invertible. In addition,
it is obvious that~$\rho (X(\omega,\lam)) = \lam$.
Thus we infer that
$$
\nabla_\h \rho (X(\omega,\lam)) \cdot\partial_\omega X(\omega,\lam)=0 \andf  \nabla_\h \rho  (X(\omega,\lam)) \cdot  (\g'(\omega))^\perp=-1\, ,
$$
 where, for any two component vector field~$u = (u^1,u^2)$, we have denoted~$u^\perp\eqdefa  (-u^2,u^1)$.
Therefore, $ \nabla_\h\rho (X(\omega,\lam))$ is collinear to~$(\g'(\omega))^\perp$ and has no component in the direction~$\g'(\omega)$. As the norm of~$\g'(\omega)$ is~$1$, we get\refeq {fod_grad_rho}.

\medbreak

In the following, we thus restrict ourselves to the particular case when {the isobaths are parallel to the coast of the island $\Gamma$,
i.e. $\varphi(x_\h) = \phi\big(\rho(x_\h)\big)$.
In addition, we consider a thickening of the land in order to avoid the singularity of the distance function at the shore.
More precisely, 
 we consider two positive real numbers~$\rho_0$ and~$H$, 
and we assume that~$\f$ is a smooth function on~$[\rho_0,\infty[$,  bounded  as well as all its derivatives,  such that
\beq
\label {Defin_Domain_0}
\f([\rho_0,\infty[) \subset [0,H] \andf \f^{-1} (0)=\{\rho_0\}\, .
\eeq
Remark that this excludes the case when~$\f$ is a constant.
 Finally, we  define
\beq
\label {Defin_Domain}
{\cO\eqdefa \R^2\setminus \big(\rho^{-1}( [0,\rho_0])\big)} \andf \Om_\f \eqdefa \bigl\{(x_\h,z)\in  {\mc O \times \R}\,/ -\f(\rho(x_\h)) <z<0\bigr\}.
\eeq
{Figure \ref{fig:new_land} below represents the thickened land and the (new) ocean surface $\mc O$.}
We   assume in the following  that the set~$\bigl\{x_\h\in \cO\,/\ \f'(\rho(x_\h))=0\bigr\}$ is negligible in~$\R^2$  {(see for instance
Figure~\ref{fig:island}, the picture on the right, and Figure \ref{fig:cut-off} in Section \ref{s:ansatz})}.

 \vspace{-0.5cm}
 \begin{figure}[H] 
\includegraphics[width=2in]{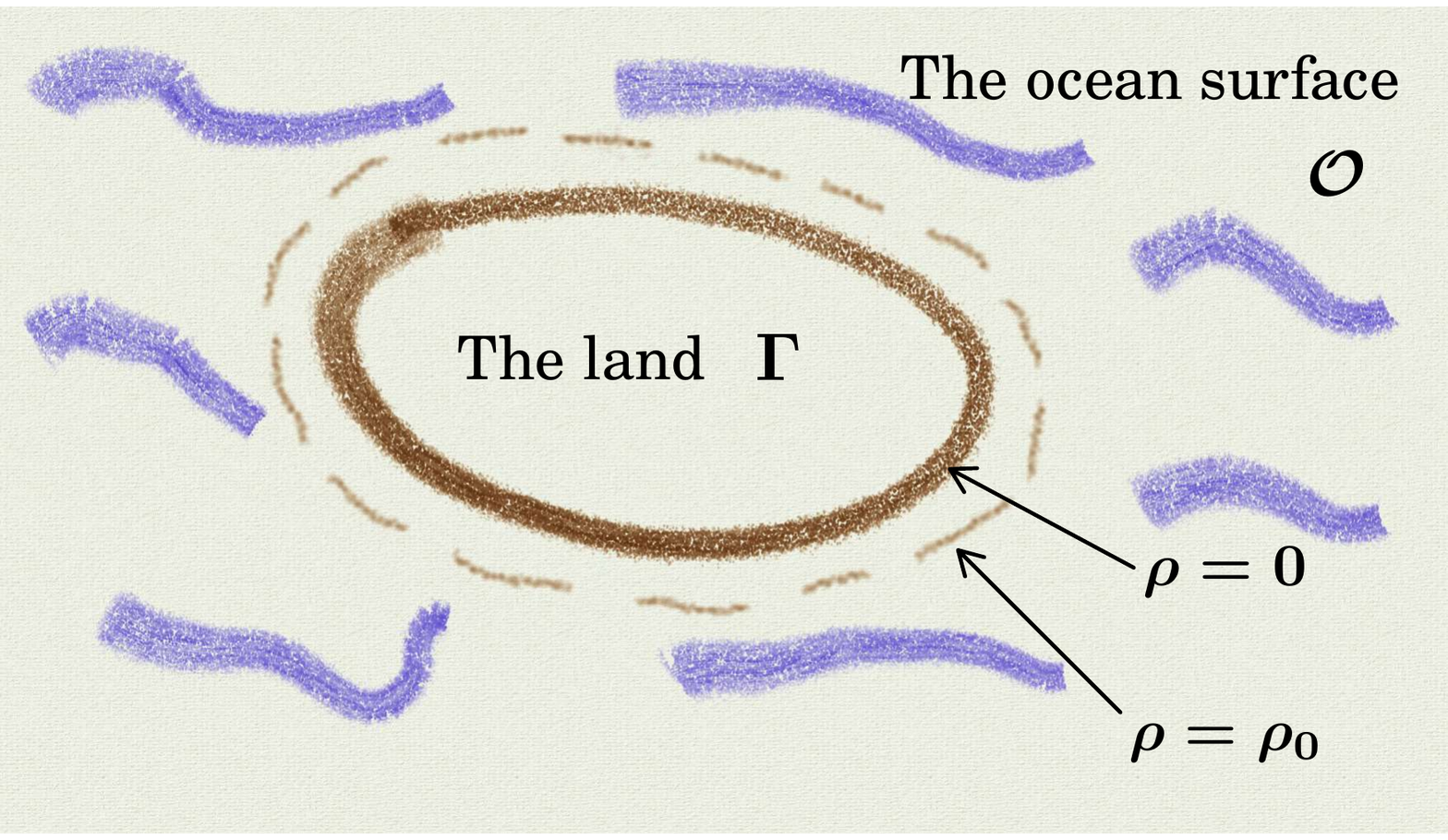}
\caption{\small The thickening of the land (confined by the dashed line) and the (new) ocean surface $\mc O$.}
\label{fig:new_land}
\end{figure}

We now consider the system~$(NSC_\e)$ set in~$\Om_\f$, and Definitions~\ref{vhv'} and~\ref{definweaksolbd}  
 {have to be} understood with~$\Om_\varphi$ replaced with~$\Om_\f$.

\medbreak
Anticipating on later notation,  in the next statement  we denote with  a superscript~$\theta$ the second component of a vector field in the orthonormal basis 
$$
\bigl(\nabla_\h \rho (x_\h), \nabla_\h^\perp\rho(x_\h), e_\zrm\bigr)\, .
$$

Our main  result is the following. 
\begin{theo}
\label {Ekman_topo_radial_0}
\sl
Consider a domain~$\Om_\f$ as defined in\refeq  {Defin_Domain}. There is a positive constant~$c$ such that the following holds. Let~$u_0 $ be a divergence free vector field in~$\cH$ 
of the form 
\begin{equation}\label{wellprepared initial data}
u_0(x_\h,z) = u_0(x_\h) \eqdefa  u_0^\theta\big(\rho(x_\h)\big)\nabla_\h^\perp\rho(x_\h)
\end{equation}
where~$u_0^\theta$ is a  smooth function {defined on~$[\rho_0,\infty[$}. Let us assume in addition that~$u_0$ and all its    its derivatives belong to~$L^2(\cO)$, and that the function~$ u_0^\theta/\f$  is bounded. 
If~$\|u_0\|_{L^\infty (\mc O)} \leq c \beta$
and if~$(u_\e)_{0< \e\leq \e_0}$ is a family of solutions of~$(NSC_\e)$ associated with the  initial data~$u_0$ as constructed in Theorem~\ref{theoleraybd},  then
\begin{equation}
\label {form of the limit}
\begin{aligned}
&\lim_{\e\rightarrow 0} \|u_\e-\overline u \|_{L^\infty_{\rm loc}(\R^+;L^2(\Om_\f ))}  =  0\,, \with {\overline u(t,x_\h,z) = \overline u ^\theta\big(t,\rho(x_\h)\big)\nabla_\h^\perp\rho(x_\h)} \, , \\
\overline u ^\theta (t,r) &    \eqdefa   { u^\theta_0(r)}    \exp \bigl( -t\lambda_\f (r)\bigr)
 \andf 
  \lambda_\f(r)    \eqdefa    \frac {\sqrt {2\b}}   {\f(r)} \,\frac {1+\sqrt[4]{1+\f '^2(r)}  } {2} \,\cdotp
\end{aligned}
\end{equation}
\end{theo}

 Some comments on the previous theorem are in order.
First, let us explain why we make the hypothesis\refeq {Cond_geo_fund}.  As we shall see later on  (see in particular Proposition~\ref {Taylor-Proudman} and Remark\refer {proof_TP_remark1}),
in a general domain~$\Om_\vf$, an element of the weak closure of the family~$(u_\e)_{0<\e\leq\e_0}$  is of the form~$G(\vf)\nabla_\h^\perp\vf$.  In Definition\refeq  {form of the limit}, the term~$\bigl(\f'\big(\rho(x_\h)\big)\bigr)^2$ {must} be understood as~$|\nabla _\h\vf(x_\h)|^2$ (recall
that we have taken $\vf = \f\circ \rho$). Then Condition\refeq  {Cond_geo_fund} seems mandatory,   {hence Condition~\eqref{Cond_geo_fund_2} as well}.

Let us remark that the equation on~$\overline u$ is the  linear ordinary differential equation 
\begin{equation}
\label {form of the limi:linear}
\frac {d \overline u} {dt}  + {\lambda_\f } \overline u =0
\end{equation}
and not the damped Euler equation~(\ref{dampedEuler}) as in the flat case. As already noticed, the hypothesis\refeq {Defin_Domain_0} on~$\f$ excludes the flat case. However the above theorem can be seen as a generalization of Theorem\refer  {grenier_masmoudi}  insofar as, in the (exceptional) points where~$\f'(\rho(x_\h))=0$, we recover the expression of the Ekman pumping  {coefficient~$\lambda$} of Theorem\refer  {grenier_masmoudi}.
The fact that the limit equation~\eqref{form of the limi:linear} is linear is linked to the presence of a non-trivial topography,
which imposes a strong constraint on elements of the kernel of the Coriolis operator ({namely, as already remarked above,
those elements must be of the form $G(\vf) \nabla^\perp_\h\vf$}).
The situation is similar in the case of a variable rotation vector~\cite{GSR1,GSR2} or
in presence of variable densities which oscillate around a non-constant profile  {(see e.g.~\cite{F-G-GV-N, F-L-N} for the case
of compressible flows, \cite{FG, C-F} for the case of incompressible density dependent fluids)}.

We now comment on Assumption \eqref{wellprepared initial data} on the initial data. 
This assumption {precisely} means that the initial data belongs to the kernel of the Coriolis operator.  This is a classical hypothesis of ``good preparation'', which avoids the problem of fast oscillations in time. These time oscillations make the problem much more delicate to treat even in the case without topography (see\ccite  {cdggbook}).  Moreover, Formula\refeq  {form of the limit} defining the limit~$\overline u$ shows  that, for any positive time {$t$}, the vector field~$\overline  u(t)$ vanishes on the shore. Thus the hypothesis that~$u_0/\f $ is bounded, which implies of course that~$u_0$ vanishes on the shore, can be understood as
a reinforcement of the hypothesis of good preparation.

\subsection*{Structure of the paper}
The paper is organised in the following way.
In the next section we   present the statement of three auxiliary results, namely Propositions \ref{Taylor-Proudman} to~\ref{structure_NLterm},
and, thanks to them, conclude the proof of Theorem \ref{Ekman_topo_radial_0}. Then, in Section \ref{proof TP}
we   perform the proof of Proposition~\ref{Taylor-Proudman}. Finally, the proof of Propositions \ref{prop_U_app} and~\ref{structure_NLterm} is  given in Sections \ref{s:ansatz}
to \ref{s:error-nonlin}, with some technicalities postponed to the Appendix.

\subsection*{Acknowledgements}
The authors wish to express their gratitude to P.-D. Thizy for interesting discussions on the geometric meaning of condition~\eqref{Cond_geo_fund}
and on Segre's theorem.

The work of the second author has been partially supported by the project CRISIS (ANR-20-CE40-0020-01),
operated by the French National Research Agency (ANR)

\bigskip
\bigskip


\section {Proof of Theorem\refer {Ekman_topo_radial_0}} \label{s:structure}

The purpose of this section is to reduce the proof of Theorem\refer {Ekman_topo_radial_0} to the proof of the  {three} propositions below.
The first ingredient   is the following  {compactness} result{, where we have used   the notation introduced before the statement of Theorem~\ref{Ekman_topo_radial_0}.}
\begin{prop}
\label {Taylor-Proudman}
{\sl Let~$v$ be  an element of the weak closure~$\cW$ of the  family~$(u_\e)_{0< \e\leq \e_0}$ constructed in Theorem~\ref{theoleraybd}. Then~$v$ writes under the form
$$
v(t,x) = \big(v_\h, 0 \big) (t,x_\h)= \big( {v^\theta}\big(t,\rho\big) \nabla_\h^\perp \rho ,0\big)(x_\h)\, . 
$$
Moreover, for any function~$g$ in~$L^\infty_{\rm loc}(\R^+;L^2(\Omega_\f))$ and any family~$(w_\e)_{0< \e\leq \e_0}$ converging weakly to an element of~$\cW$, the function
$$
R_\e(t)\eqdefa   \big( w_\e   |   g \nabla_\h \rho\big)_{L^2(]0,t[ \times\Omega_\f )} $$
converges uniformly to~0 on any compact interval.}
\end{prop}

The first part of this proposition is the well known Taylor-Proudman theorem in the presence of topography. Let us notice that the topography induces an additional constraint compared with   the flat case, where the elements of the weak closure~$\cW$ are under the form
$$
\bigl (u_\h(t,x_\h),0\bigr) \with \dive_\h u_\h =0\, .
$$
The second part of the proposition claims that the weak convergence of a family~$(w_\e)_{0< \e\leq \e_0}$ to some element of~$\cW$ can be translated into a form of uniform-in-time convergence. As we shall see below, this plays an important role in the final proof of Theorem\refer {Ekman_topo_radial_0}.

\medbreak

The second ingredient is the main, and most technical, argument leading to the proof of Theorem\refer {Ekman_topo_radial_0}. It consists in the construction of a family of  approximate solutions of the linear Stokes-Coriolis system, namely the system
\[
\left\{
\begin{array}{c} 
\ds \partial_t u_\e - \e\b\D u_\e  +\frac 1 \e e_{\zrm} \wedge u_\e + =-\frac 1 \e \nabla p_\e\,,\\
\ds  {u_\e}_{|t=0} = u_0\,,\\
\ds \dive u_\e =0\andf {u_\e}_{|\partial\Om_\vf }=0\,  .
\end{array}
\right.
\leqno(SC_\e)
\]
The properties of those approximate solutions is described in the following proposition.

\begin {prop}
\label {prop_U_app}
{\sl
Under the assumptions of Theorem~\ref{Ekman_topo_radial_0}, there is a  family~$(U_{\app,\e})_{0< \e\leq \e_0} $ of vector fields in~$C^1(\R^+;\cV_\sigma)$, which solve {the linear system $(SC_\e)$ in an approximate way,} in the   sense that they satisfy, for any $0<\e\leq\e_0$, the equations 
\[ 
  \left\{ \begin {array} {c} \ds \partial_t U_{\app,\e}  - \e\b\D U_{\app,\e}   +\frac 1 \e e_{\zrm} \wedge U_{\app,\e}   =
  -\frac 1 \e \nabla p_\e  +
    E_\e \\\\
  \ds  \lim_{\e\rightarrow0} \bigl\| {U_{\app,\e}} _{ |t=0}  -u_0\|_{L^2(\Om_\f)}=0\,,
\end {array}
\right.
\leqno(SC_{\app,\e})
\]
   where~$ E_\e$ converges to zero in~$L^1(\R^+; L^2(\Omega_\f ))$ and $(p_\e)_{0< \e\leq \e_0} $ is a    family of functions in the space~$L^2_{\rm loc}(\R^+;L^2(\Omega_\f ))$. 
The family~$U_{\app,\e}  $ is a good approximation of~$\overline u  $ as defined in Theorem~\ref{Ekman_topo_radial_0},
 in the sense~that
\beq
\label {prop_U_app_eq2}
\lim_{\e\to 0} \bigl\| U_{\app,\e}   - \overline u\bigr\|_{L^\infty(\R^+;L^2(\Om_\f))} = 0 \, .
\eeq
The approximate solution~$U_{\app,\e}  $ is made of a principal boundary layer  term~$U_{{\rm BL},\e}$, such that 
 \begin{equation}
 \label{bdd gradients}
 \big(\nabla(U_{\app,\e}  - U_{{\rm BL},\e})\bigr)_{0< \e\leq \e_0}  \quad \mbox{ is bounded in}\quad L^1(\R^+;L^\infty(\Om_\f)) \, , 
 \end{equation}
 and  the family~$(U_{{\rm BL},\e})_{0< \e\leq \e_0} $ satisfies 
  \begin{equation}\label{small BL}
\bigl\|  \sqrt {d_\f} U_{{\rm BL},\e} \bigr\|_{L^\infty(\R^+; L^\infty_\h L^2_\v(\Om_\f))}  
\leq C\e\|u_0\|_{ L^\infty(\Omega_\f)}+C\e^2\|u_0\|_{ {W^{1,\infty}}(\Omega_\f)}\,,
 \end{equation}
 where~$d_\f$ is defined by $d_\f(x_\h, z) \eqdefa \min\{ -z, \f(\rho(x_\h))+z\}$
and  the space~$L^\infty_\h L^2_\v(\Om_\f)$ is defined by the norm
$$
\|F\|_{L^\infty_\h L^2_\v(\Om_\f)} \eqdefa \sup_{x_\h\in \cO} \biggl(\int_{-\f(\rho(x_\h))} ^0 {\left|F(x_\h,z)\right|^2} dz\biggr)^{\frac 12}\, .
$$
}
\end {prop} 

The third proposition describes the structure of the non linear quantity~$U_{\app,\e} \cdot \nabla U_{\app,\e}$. It is not a gradient, but it is close to the term~$g\nabla_\h\rho$ that appears in the statement of Proposition\refer   {Taylor-Proudman}. 
\begin {prop}
\label {structure_NLterm}
{\sl Let~$(U_{\app,\e})_{0< \e\leq \e_0} $ be the family given by Proposition\refer  {prop_U_app}. Then we have, with notation~\eqref{form of the limit},
$$
\lim_{\e\rightarrow 0} \bigl\| U_{\app,\e} \cdot \nabla U_{\app,\e} +  \bigl(\overline u^\theta(t,\rho)\bigr)^2\D_\h\rho \nabla_\h\rho \bigr\|_{L^1(\R^+,L^2(\Om_\f))} =0\,.
$$}
\end {prop}
Before proving Theorem\refer  {Ekman_topo_radial_0} using the above three  {statements},  let us make some comments about this proposition. For every smooth vector field of the form~$G(\rho)\nabla_\h^\perp \rho$, we have
$$
G(\rho)\nabla_\h^\perp \rho \cdot \nabla_\h \big(G(\rho)\nabla _\h^\perp \rho\big)= 
G^2(\rho) \nabla_\h^\perp \rho \cdot \nabla\nabla _\h^\perp \rho\,.
$$
At this point, it is interesting to compute, for any function~$\vtheta$, the quantity~$ \nabla_\h^\perp \vtheta \cdot \nabla\nabla _\h^\perp \vtheta$. In cartesian coordinates, we have
\ben
\nonumber \nabla_\h^\perp \vtheta \cdot \nabla\nabla _\h^\perp \vtheta & = & \bigl(-\partial_2\vtheta\partial_1+\partial_1\vtheta\partial_2\bigr)\begin {pmatrix} -\partial_2\vtheta \\ \partial_1\vtheta \end {pmatrix}\\
\nonumber & = & \begin {pmatrix} \partial_2\vtheta\partial_1\partial_2\vtheta- \partial_1\vtheta\partial^2_2\vtheta \\ 
- \partial_2\vtheta\partial_1^2\vtheta+ \partial_1\vtheta\partial_1\partial^2_2\vtheta \end {pmatrix}\\
\label {Euler_Station_fund}
& =& \frac 1 2 \nabla_\h\bigl( |\nabla_\h\vtheta|^2\bigr)-\D_\h \vtheta \nabla_\h \vtheta\, .
\een
{Then} a vector field of the type~$\nabla_\h^\perp\vtheta$ is a stationary solution  of the 2D Euler equation  as soon as a function~$g$ exists such that~$\D_\h \vtheta =g(\vtheta)$.

In the particular case when~$\vtheta=\rho$, because the norm of the gradient of~$\rho$ is~$1$, we get
\beq
\label {deriv_nabla_rho} 
\nabla_\h^\perp \rho\cdot \nabla_\h (\nabla_\h^\perp \rho ) = -\D_\h\rho \nabla_\h\rho\, .
\eeq
To ensure that~$\D_\h\rho \nabla_\h\rho $ is a gradient, the assumption is that~$\D_\h\rho=G(\rho)$. Segre's theorem (see\ccite  {Brustad})  claims in particular  that, if a smooth function~$\vtheta$ on a bi-dimensional  domain satisfies 
$$
\D_\h \vtheta =F(\vtheta)\andf |\nabla_\h\vtheta|^2 = G(\vtheta) \,,
$$
then~$\vtheta$ is radial or linear. Here our setting is more general, thus~$\D_\h\rho \nabla_\h\rho $ is not a gradient.

As we shall see below, the form of~$U_{\app,\e} \cdot \nabla U_{\app,\e}$ plays a crucial role in the proof of Theorem\refer  {Ekman_topo_radial_0},
{which we are now going to present.}

\begin {proof}[Proof of Theorem{\rm\refer  {Ekman_topo_radial_0}} admitting Propositions{\rm\refer  {Taylor-Proudman}--\ref  {structure_NLterm}}]
Our approach is inspired by the method used to prove weak-strong stability results. The idea is to use a ``regular" solution as a test function in the definition of a turbulent solution. Here the role of the regular solution is played by the approximate solution~$U_{\app,\e}  $. This method is classical (see for instance\ccite {cdggbook} or\ccite {grenier_masmoudi}) and, to the best of our knowledge, was used for the first time in the work~\cite{vonwahl} for the proof of weak-strong uniqueness to the three-dimensional Navier-Stokes equations.

We want to prove that 
$$\d_\e   \eqdefa u_\e -U_{\app,\e}
$$
converges to zero in the energy space, where~$U_{\app,\e}  $ is the approximate  solution  defined in Proposition\refer {prop_U_app}. Usually, to prove uniqueness for evolution problems of the type
$$
Lu=Q(u,u)
$$
where~$L$ is a linear operator and~$Q$ is quadratic, we write the equation on~$w=u-v$ by writing 
$$
Lw = Q(w,u) +Q(v,w)\, .
$$
In that case the two solutions~$u$ and~$v$ play the same role. As we shall see, the weak-strong uniqueness method {and the structure of
the non-linear term $Q$} make the role of the two solutions non symmetric. Let us  define
$$
\D_\e(t) \eqdefa \frac 12 \|\d_\e(t)\|_{L^2(\Om_\f)}^2 +\b \e\int_0^t \|\nabla \d_\e(t')\|_{L^2(\Om_\f)}^2 dt'\, .
$$
 Expanding the square, let us write that
$$
\begin {aligned}
\D_\e(t)  & = \frac 12 \|u_\e(t)\|_{L^2(\Om_\f)}^2 +\b \e\int_0^t \|\nabla u_\e(t')\|_{L^2(\Om_\f)}^2 dt'\\
& \qquad{}
+\frac 12 \|U_{\app,\e}  (t)\|_{L^2(\Om_\f)}^2 +\b \e\int_0^t \|\nabla U_{\app,\e}  (t')\|_{L^2(\Om_\f)}^2 dt'\\
 &\qquad\qquad{} -  \bigl(u_\e(t)\big | U_{\app,\e} (t)\bigr)_{L^2(\Om_\f)}
-2\b \e  \int_0^t \bigl(\nabla u_\e(t')\big | \nabla U_{\app,\e} (t')\bigr)_{L^2(\Om_\f)} dt'.
\end {aligned}
$$
Thanks to the energy estimate\refeq  {theoleraybdeq1},
there holds
$$
\frac 12 \|u_\e(t)\|_{L^2(\Om_\f)}^2 +\b \e\int_0^t \|\nabla u_\e(t')\|_{L^2(\Om_\f)}^2 dt' \leq  \frac 12 \|u_0\|_{L^2(\Om_\f)}^2\, .
$$
Moreover $ U_{\app,\e} $ satisfies the linear system~$(SC_{\app,\e} )$. This implies that 
$$
\longformule
{
\frac 12 \|U_{\app,\e}  (t)\|_{L^2(\Om_\f)}^2 +\b \e\int_0^t \|\nabla U_{\app,\e}  (t')\|_{L^2(\Om_\f)}^2 dt'
}
{ {}
\leq \frac 12 \|u_{0,\e}\|_{L^2(\Om_\f)}^2+\int_0^t\int_{\Om_\f}  E_\e \cdot U_{\app,\e} (t',x) \,  dt' dx\, .
}
$$
Since, by definition,~$ E_\e$ converges to zero in~$L^1(\R^+; L^2(\Omega_\f ))$ and, by construction,~$U_{\app,\e}$ is uniformly bounded in the space~$L^\infty(\R^+;L^2(\Om_\f))$, there holds
$$
\lim_{\e \to 0}\sup_{t \geq 0} \Big| \int_0^t\int_{\Om_\f}  E_\e \cdot U_{\app,\e} (t',x) \,  dt' dx  \Big|= 0 \,.
$$

In all that follows, $ r_\e$ denotes a generic function which satisfies
\begin{equation} \label{eq:poubelle}
\forall T>0 \, , \quad \lim_{\e \to 0}\sup_{t \in [0,T]} r_\e(t) = 0 \,.
\end{equation}
Collecting all the previous information, we  find that
\beq
\label {proof_converge_theo_nonlinear_eq3}
\begin {aligned}
\D_\e(t) & \leq  \frac 12 \|u_0\|_{L^2(\Om_\f)}^2 +\frac 12 \|u_{0,\e}\|_{L^2(\Om_\f)}^2  +  r_\e(t)\\
 &\qquad\quad{} -  \bigl(u_\e(t)\big | U_{\app,\e} (t)\bigr)_{L^2(\Om_\f)}
-2\b \e  \int_0^t \bigl(\nabla u_\e(t')\big | \nabla U_{\app,\e} (t')\bigr)_{L^2(\Om_\f)} dt'.
\end  {aligned}
\eeq

Now, we use the approximate solution~$U_{\app,\e}$ as a test function in formula\refeq  {definweaksolbd}. This gives
\begin{equation}
\label{WS_demo_eq1}
\begin{aligned}  &
(u_0|U_{\app,\e}(0))_{L^2(\Om_\f)} = ( u_\e(t) | U_{\app,\e} (t))_{L^2(\Om_\f)} -\int_0^t \int _{\Om_\f}    \bigl(u_\e \otimes u_\e : \nabla U_{\app,\e}\bigr)(t',x)\, dxdt' \\
& {}\qquad\qquad\quad 
 + 
\int_0^t\! \int_{\Omega_\f} \biggl( \b\e\nabla u_\e : \nabla U_{\app,\e}
- u_\e \cdot \Bigl( \partial_t U_{\app,\e} +  \frac 1 \e e_{\zrm} \wedge U_{\app,\e}
\Bigr)\biggr)(t',x)\, dx dt'.
\end{aligned}
\end{equation}
As we have
$$
\int_{\Omega_\f} \bigl( \b\e\nabla u_\e : \nabla U_{\app,\e}\bigr)(t',x) dx = -\b\e \langle \D U_{\app,\e} (t'), u_\e(t')\rangle _{H^{-1}\times H^1_0}\, ,
$$
we infer from the approximate Stokes-Coriolis system~$(SC_{\app,\e} ) $ that 
\[
\begin {aligned}
\cL_\e(t) &\eqdefa  \int_0^t \int_{\Omega_\f} \biggl( \b\e\nabla u_\e : \nabla U_{\app,\e}
- u_\e \cdot \Bigl( \partial_t U_{\app,\e}+  \frac 1 \e e_{\zrm} \wedge U_{\app,\e}
\Bigr)\biggr)(t',x)\, dx dt'\\
& =  - \int_0^t   \Big\langle    \partial_tU_{\app,\e}+\b\e \D U_{\app,\e}+  \frac 1 \e e_{\zrm} \wedge U_{\app,\e}, u_\e(t')\Big\rangle_{{H^{-1}\times H^1_0}} dt'\\
& =  2 \int_0^t \int_{\Omega_\f} \bigl( \b\e\nabla u_\e : \nabla U_{\app,\e}\bigr)(t',x) dt'dx
+\int_0^t (u_\e(t')|E_\e(t'))_{L^2(\Om_\f)}dt'.
\end{aligned}
\]
As~$\ds \lim_{\e\rightarrow0} \|E_\e\|_{L^1(\R^+;L^2(\Om_\f)}=0$, we infer that 
$$
\cL_\e(t) = 2  \int_0^t \int_{\Omega_\f} \bigl( \b\e\nabla u_\e : \nabla U_{\app,\e}\bigr)(t',x)\, dx dt'
+r_\e(t) \, .
$$
Inserting this into\refeq {WS_demo_eq1} gives
$$
\longformule{
(u_0|U_{\app,\e}(0))_{L^2(\Om_\f)} = ( u_\e(t) | U_{\app,\e} (t))_{L^2(\Om_\f)} + 2  \int_0^t \int_{\Omega_\f} \bigl( \b\e\nabla u_\e : \nabla U_{\app,\e}\bigr)(t',x)\, dx dt'  }
{  {}
-\int_0^t \int _{\Om_\f}    \bigl(u_\e \otimes u_\e : \nabla U_{\app,\e}\bigr)(t',x)\, dxdt' 
+r_\e(t) \, .
}
$$
Plugging the above relation into\refeq  {proof_converge_theo_nonlinear_eq3} ensures that
\beno
\D_\e(t) & \leq  &  \frac 12 \|u_0\|_{L^2(\Om_\f)}^2 +\frac 12 \|u_{0,\e}\|_{L^2(\Om_\f)}^2 -(u_0|U_{\app,\e}(0))_{L^2(\Om_\f)}
 +  r_\e(t)\\
 & &\qquad\qquad\qquad\qquad\qquad\qquad\qquad
  {} +\int_0^t \int _{\Om_\f}    \bigl(u_\e \otimes u_\e : \nabla U_{\app,\e}\bigr)(t',x)\, dxdt' 
\\
& \leq &   \frac 12 \|u_0-U_{\app,\e}(0)\|^2_{L^2(\Om_\f)}
 +  r_\e(t) +\int_0^t \int _{\Om_\f}    \bigl(u_\e \otimes u_\e : \nabla U_{\app,\e}\bigr)(t',x)\, dxdt'\,.
  \eeno
  Using Assertion~\eqref{prop_U_app_eq2}  of Proposition\refer  {prop_U_app}, we infer that 
\beq
\label {WS_demo_eq2} 
\D_\e(t)  \leq r_\e(t) +\int_0^t \int _{\Om_\f}    \bigl(u_\e \otimes u_\e : \nabla U_{\app,\e}\bigr)(t',x)\, dxdt'\,.
\eeq
Now, let us study the non-linear term of the above inequality in light of Proposition\refer  {structure_NLterm}.  Let us observe that, if~$a$ and~$b$ are
two vector fields in~$\cV_\s$, we have
$$
\begin{aligned}
\int_{\Om_\f}  a\otimes a: \nabla b(x) dx & =  (a\cdot \nabla b|a)_{L^2} \\
& =    \bigl( (a-b)\cdot \nabla b \big | a\bigr)_{L^2}   + ( b\cdot \nabla b \big | a)_{L^2}\\
& = \bigl( (a-b)\cdot \nabla b \big | (a-b)\bigr)_{L^2}   + \big( b\cdot \nabla b \big | a \big)_{L^2}\,.
\end{aligned}
$$
Applying this with~$a=u_\e$ and~$b=U_{\app,\e}$ and plugging  {the resulting expression} into\refeq {WS_demo_eq2}, we deduce that
\[
\begin {aligned}
\D_\e(t) \leq    &  \int_0^t \int_{\Om_\f} ( \d_\e \otimes \d_\e)   :   \nabla U_{\app,\e}  (t',x) \, dxdt'    \\
& \qquad\qquad\qquad {}+\int_0^t \int_{\Omega_\f} \bigl( ( U_{\app,\e}\cdot\nabla U_{\app,\e})\cdot u_\e\bigr)(t',x)dt'dx + r_\e(t)
 \,.
\end {aligned}
 \]
 Then, Proposition\refer  {structure_NLterm} implies that 
\[
\begin {aligned}
\D_\e(t) \leq    &  \int_0^t \int_{\Om_\f} ( \d_\e \otimes \d_\e)   :   \nabla U_{\app,\e}  (t',x) \, dxdt'    \\
& \qquad\qquad\qquad {}+\int_0^t \int_{\Omega_\f} \bigl( (u^\theta)^2\D_\h\rho \nabla_\h\rho \cdot u_\e\bigr)(t',x)dt'dx + r_\e(t)
 \,.
\end {aligned}
 \]
 Proposition\refer  {Taylor-Proudman} ensures that 
 \beq
\label {estim_Delta_WS}
\D_\e(t) \leq       \int_0^t \int_{\Om_\f} ( \d_\e \otimes \d_\e)   :   \nabla U_{\app,\e}  (t',x) \, dxdt'   + r_\e(t)
 \,.
\eeq

At this point, we use the decomposition of Proposition\refer {prop_U_app} and the Cauchy-Schwarz inequality:  we get, from Equation~\eqref{estim_Delta_WS}, that 
\beq
\label  {estim_Delta_WS_eq1}
\begin{aligned}
\D_\e(t) & \leq  
\int_0^t  \|\d_\e(t')\|^2_{L^2(\Om_\f)} \bigl \|\nabla ( U_{\app,\e} -U_{{\rm BL},\e} )  (t',\cdot)\bigr\|_{L^\infty(\Om_\f)} dt'   + \cQ_{{\rm BL},\e}(t)
+ r_\e(t)\,,  \\
&\with  \cQ_{{\rm BL},\e}(t) \eqdefa \int_0^t \int_{\Om_\f} ( \d_\e \otimes \d_\e)   :   \nabla U_{{\rm BL},\e}  (t',x) \, dxdt' .
\end {aligned}
\eeq
By integration by parts  {and thanks to the divergence free condition satisfied by~$\d_\e$}, we have
$$
 \cQ_{{\rm BL},\e}(t) = \int_0^t \int_{\Om_\f} ( \d_\e \cdot \nabla  \d_\e) (t',x) \cdot  U_{{\rm BL},\e}(t',x) (t',x) \, dxdt'.
$$
This term is estimated thanks to the following lemma, which is a variation of a result which is classical in the flat case (see for instance\ccite  {cdggbook}, Lemma~7.4 page~169).  We admit it for the time being.
\begin{lemme}
\label {BLEkmantoporadpSmall}
{\sl
Let~$\Om_\vf$ a domain satisfying\refeq {defin_domain_general}.  Let us consider a vector field~$ \delta$ in~$\cV_\sigma $ and~$ w$ a bounded vector field. Then
$$
 \big | \big(\delta \cdot \nabla \delta  | w
\big)_{L^2} \big | \leq   \|\nabla \d\|^2_{L^2} \|\sqrt {d_\vf} w\|_{L^\infty_\h L^2_\v(\Om_\vf)}\, ,
$$
{where the notations $d_\vf$ and $L^\infty_\h L^2_\v(\Om_\vf)$ are defined as $d_\f$ and $L^\infty_\h L^2_\v(\Om_\f)$ in Proposition \ref{prop_U_app},
but using the function $\vf$ instead of $\f(\rho)$.}
}
\end{lemme}
Let us apply this lemma  with~$\d=\d_\e$ and~$w= U_{{\rm BL},\e}$.  Thanks to Proposition\refer {prop_U_app}, Inequality\refeq  {estim_Delta_WS_eq1} becomes
$$
\longformule{
\D_\e(t)  \leq  
\int_0^t \|\d_\e(t')\|^2_{L^2(\Om_\f)} \bigl \|\nabla    {(U_{\app,\e} -U_{{\rm BL},\e})}  (t',\cdot)\bigr\|_{L^\infty(\Om_\f)} dt'   
}
{ {} + 
\bigl (C\e\|u_0\|_{ L^\infty(\Omega_\f)}+C\e^2\|u_0\|_{ {W^{1,\infty}}(\Omega_\f)}\bigr)\int_0^t  \|\nabla \d_\e(t')\|^2_{L^2(\Om_\f)} dt' 
+ r_\e(t)\, .
}
$$
Choosing~$\e$ less than~$\b /\big(4C\|u_0\|_{ {W^{1,\infty}}(\Omega_\f)}\big)$ and using that~$C   \|u_0\|_{L^\infty(\Om_\f)}$ is less than~$\beta/4$, then we~get
$$
\D_\e(t)  \leq  
2 \int_0^t   \|\d_\e(t')\|^2_{L^2(\Om_\f)}  \|\nabla   {(U_{\app,\e} -U_{{\rm BL},\e})} (t',\cdot)\|_{L^\infty(\Om_\f)} dt'   
+ r_\e(t)\, .
$$
Thanks to Proposition \ref{prop_U_app} we know that the family~$\bigl(\nabla   {(U_{\app,\e} -U_{{\rm BL},\e})\bigr)_\e} $ is bounded in the space~$L^1(\R^+;L^\infty(\Om_\f))$.
Gr\"onwall's lemma then concludes the proof of Theorem\refer  {Ekman_topo_radial_0}, provided of course that we prove  Lemma\refer  {BLEkmantoporadpSmall}.
\end {proof}

\medbreak
\begin{proof}[Proof of Lemma\rm\refer  {BLEkmantoporadpSmall}]
Let~$a$  be a function of~$H^1_0(\Om_\f)$. Because~$a$ vanishes at the boundary, we can write
$$
a (x_\h,z) =-\int_z^0 \partial_\zrm a(x_\h,z') dz'= \int_{-\vf(x_\h)} ^z \partial_\zrm a(x_\h,z') dz'.
$$
The Cauchy-Schwarz inequality implies that
\beno
\bigl | a(x_\h,z) \bigr | &  \leq &  \bigl (\min \{-z,\vf(x_\h)+z\} \bigr)^{\frac 12} \biggl( \int_{-\vf(x_\h)} ^0 \bigl | \partial_\zrm a(x_\h,z')\bigr |^2 dz'\biggr)^{\frac 12}\\
&  \leq & d^{\frac 12} _\vf(x_\h,z)  \biggl( \int_{-\vf(x_\h)} ^0 \bigl | \partial_\zrm a(x_\h,z')\bigr |^2 dz'\biggr)^{\frac 12}.
\eeno
Applying this inequality to~$a = \delta^k$ we get
\beno
\bigl | \left(\delta \cdot \nabla \delta  | w
\right)_{L^2} \bigr | & = & \Big | \sum_{1\leq j,k\leq 3}
 \int_{\R^2_\h} \biggl( \int_{-\vf(x_\h)}^0  \delta^k(x_\h,z) \partial_k \delta^j(x_\h,z) w^j (x_\h,z) dz\biggr) dx_\h \Big |  \\
 & \leq  & \sum_{1\leq j,k\leq 3}
 \int_{\cO_\vf}  \biggl( \int_{-\vf(x_\h)} ^0 \bigl | \partial_\zrm \d^k (x_\h,z)\bigr |^2 dz\biggr)^{\frac 12} \\
 && \qquad\qquad\qquad  {} \times  \int_{-\vf(x_\h)}^0 \big|  \partial_k \delta^j(x_\h,z)  d^{\frac 12} _\vf(x_\h,z)  w^j (x_\h,z)  \big|dz dx_\h \,.
 \eeno
 Then the Cauchy-Schwarz inequality again implies that
 $$
 \longformule{
  \big|  \left(\delta \cdot \nabla \delta  | w \right)_{L^2} \big|   \leq
 \sum_{1\leq j,k\leq 3}
 \int_{\cO_\vf}  \biggl( \int_{-\vf(x_\h)} ^0 \bigl | \partial_\zrm \d^k (x_\h,z)\bigr |^2 dz\biggr)^{\frac 12}
 }
 { {} \times
 \biggl( \int_{-\vf(x_\h)}^0  \big(\partial_k \delta^j(x_\h,z)\big)^2 dz\biggr)^{\frac 12}  \biggl( \int_{-\vf(x_\h)}^0  d _\vf(x_\h,z)  (w^j)^2 (x_\h,z) dz\biggr)^{\frac 12}dx_\h\,.
}
$$
The lemma is thus proved.
\end{proof}

\medbreak
The proof of Proposition~\ref{Taylor-Proudman} is performed in Section~\ref{proof TP}. This proof  is rather short and classical; it is based on an Ascoli argument.
{In Section \ref{s:ansatz}, we present the structure of the proof of Proposition \ref {prop_U_app}, which is the main technical part of the paper and
will be performed in Sections \ref{interior0} to \ref{estim_error_terms}. Finally, in Section \ref{s:error-nonlin} we present the proof
of Proposition \ref{structure_NLterm}.}


\section{Proof of Proposition~\ref{Taylor-Proudman} }
\label{proof TP}

Proposition \ref{Taylor-Proudman} is, in some sense, a refined version of the classical Taylor-Proudman theorem. Let us start by using
Equation\refeq {defin1.2_eq} mutiplied by~$\e$. This implies that for any test function~$\Psi$, there is a constant~$C$ such that
$$
\biggl |\int_0^t \int _{\Om_\f} \bigl ((e_\zrm\wedge u_\e) \cdot \Psi \bigr)(t',x)dt'dx \biggr | \leq C \e\, .
$$

Thus, if~$w$ belongs to~$\cW$, a sequence~$\suite \e n \N$ tending to~$0$  exists such that
$$
\lim_{n\rightarrow \infty} \int_0^t \int _{\Om_\f} \bigl ((e_\zrm\wedge u_{\e_n}) \cdot \Psi \bigr)(t',x)dt'dx = \int_0^t \int _{\Om_\f} \bigl ((e_\zrm\wedge w) \cdot \Psi \bigr)(t',x)dt'dx = 0\, .
$$
Let us consider any function~$\psi$ in~$\cD(]0,\infty[\times \Om_\f)$, and  apply the above relation with
$$
\Psi= (-\partial_2 \psi, \partial_1 \psi,0)\,,\ \Psi= (-\partial_\zrm \psi, 0,\partial_1\psi)\andf  \Psi= (0, -\partial_\zrm \psi, \partial_2\psi)\,.
$$
By definition of the  derivation in the sense of distributions, this gives
$$
\langle \dive_\h w_\h ,\psi\rangle =- \langle \partial_\zrm w^2 ,\psi\rangle =  \langle \partial_\zrm w^1,\psi\rangle=0\, .
$$
As~$w$ belongs to~$L^2([0,T];\cH )$, then the vector field~$w(t,\cdot)$ belongs to~$\cH $  for almost every~$t$. In particular, it is divergence free, so, for almost every positive~$t$,
$$
\dive_\h w_\h(t,\cdot) = -\partial_\zrm w^\zrm (t,\cdot) =0 \, .
$$
Then, the vector field~$w$ must be of the form
\beq
\label {Taylor_Proudman_basic}
\bigl(w_\h(x_\h),w^\zrm (x_\h)\bigr) \with\dive_\h w_\h=0\, .
\eeq
The fact that, for almost every~$t$, the vector field~$w(t,\cdot)$ belongs to~$\cH $ implies also that
$$
w^\zrm (t,x_\h)= w_\h(t,x_\h)\cdot \nabla _\h\bigl(\f(\rho(x_\h))\bigr)=0\, .
$$
As we have assumed that the set~$\bigl\{x_\h\in \cO\,/\ \f'(\rho(x_\h))=0\bigr\}$ is negligible,  we infer that
the vector field~$w$ is, for almost every~$(t,x_\h)$, of the form $w=\big(w_\h,0\big)$, with
$$
w_\h(t,x_\h)=w^\theta(t,x_\h) \nabla_\h^\perp \rho(x_\h)\, .
$$
As~$w_\h$ is divergence free, we have also
\beno
 0 = \dive_\h w_\h (t,x_\h) = \nabla _\h w^\theta(t,x_\h)\cdot \nabla_\h^\perp \rho(x_\h)\, .
\eeno
This implies that~$w^\theta$ is constant on the  curves~$ \rho(x_\h)\equiv C$. Thus, we have
$$w^\theta(t,x_\h)=g(t,\rho(x_\h))\, .$$ The first part of the proposition is hence proved.
\begin {remark}
\label {proof_TP_remark1}
{\sl The proof above works in any domain~$\Om_\vf$ of the form\refeq  {defin_domain_general}. Thus, for any  {such domain $\Omega_\vf$}, an element of the kernel of the Coriolis operator is of the form~$\nabla_\h^\perp F(\vf)$.
}
\end {remark}

In order to prove the second part of the proposition, let us write $R_\e$ as
$$
R_{\e}(t) = \bigl(w_{\e}\big | {\bf 1}_{[0,t]} g \nabla \rho\bigr)_{L^2(\R^+\times \Om_\f)}\,.
$$
We take a sequence~$\big(w_{\e_n}\big)_n$ which tends weakly to~$w$ in~$L^2(\R^+\times \Om_\f)$. Then we have
$$
w(t,x_\h) = g(t,\rho(x_\h))\ \big(\nabla_\h^\perp \rho(x_\h),0\big) \andf \lim_{n\rightarrow \infty} R_{\e_n}(t) = \bigl(w_\h\big | {\bf 1}_{[0,t]} g \nabla_\h \rho\bigr)_{L^2(\R^+\times \Om_\f)}\, .
$$
As~$w_\h$  is colinear to~$\nabla_\h^\perp \rho$, we infer that
\beq
\label {proof TP_demoeq1}
\forall t \in [0,\infty[\,,\ \lim_{n\rightarrow\infty} R_{\e_n} (t) =0\,.
\eeq
Moreover, the  Cauchy-Schwarz inequality implies that, for all~$t' \leq t$,
\beno
\bigl |R_\e(t) -R_\e(t') \bigr| & \leq & \left|\int_{[t',t]\times \Om_\f}  g(s,x_\h,z) w_{\e,\h}(s,x_\h,z)\cdot\nabla_\h\rho(x_\h)dx_\h dzds\right|\\
& \leq & |t-t'| \, \|w_\e\|_{L^\infty(\R^+; L^2(\Om_\f))} \|g\|_{L^\infty(\R^+; L^2(\Om_\f))}\,.
\eeno
Then Ascoli's theorem implies that the set~$(R_\e)_{0<\e\leq \e_0}$ is a relatively compact subset of the space of continuous real valued functions on~$[0,T]$,
for any fixed time $T>0$. Because of\refeq {proof TP_demoeq1}, the sequence~$(R_{\e_n})_{n\in \N}$ tends uniformly to~$0$ on~$[0,T]$, and Proposition \ref{Taylor-Proudman} is proven. \qed

\medbreak
\begin{remark} {\sl Let us point out that we have no rate of convergence of~$R_\e$ to~$0$.   As we  have seen in the previous
section, namely Relation \eqref{estim_Delta_WS}, the rate of convergence of~$\ds \sup_{t\in [0,T] } |R_{\e}(t)|$ to~$0$
 {determines, or rather imposes a constraint on,} the rate of convergence of~$u_{\e}$ to~$\overline u$ in~$L^\infty([0,T];\cH(\Om_\f))$.}
\end{remark}


\section {The process of construction of the approximate solutions}
\label {s:ansatz}

We now start the proof of  {Propositions \ref{prop_U_app} and \ref{structure_NLterm}}. It consists of two parts: the first one is the precise construction of the family
of approximate solutions $U_{\app,\e}$  {(from Sections~\ref{BL0} to~\ref{full approximate})}, the second one is the estimate of the error terms  (Sections~\ref{estim_error_terms} and~\ref{s:error-nonlin}).
The goal of the present section is to explain the general strategy and the main ideas of the proof.

Before going into the details of the process, let us precise some notations and conventions which will be used in all that follows. First of all, as mentioned in the introduction, we are going to decompose any vector at a point~$(x_\h,z)$ of~$\Om_\f$  in the orthonormal basis 
$$
\bigl(\nabla_\h \rho (x_\h), \nabla_\h^\perp\rho(x_\h), e_\zrm\bigr)\, .
$$
\\
\noindent
Given a three-dimensional vector field~$U= \big(U^1,U^2,U^3\big)$, we denote by~$ \big(U^\rho,U^\theta,U^\zrm\bigr)$ its components in this basis. This means that 
\beq
\label {frame_coodinates}
U(x_\h,z) = U^\rho(x_\h,z) \nabla_\h \rho  (x_\h)+ U^\theta(x_\h,z) \nabla_\h^\perp \rho (x_\h)+ U^\zrm(x_\h,z) e_\zrm \, .
\eeq
Let us notice that, in this frame, if the component~$U^\theta$ is a function of~$\rho(x_\h)$ only {(which will turn out to be the case in what follows)}, then the divergence of~$U $ writes 
\beq
\label {div_rotating_frame_coodinates}
\begin {aligned} 
& \dive \big(U^\rho(x_\h,z) \nabla_\h \rho  (x_\h)+ U^\theta (\rho(x_\h))\nabla_\h^\perp \rho (x_\h) +U^\zrm e_\zrm \bigr) \\
&\qquad \qquad\qquad \qquad\qquad \qquad \qquad {}= \dive_\h \big(U^\rho (x_\h,z)\nabla_\h \rho  (x_\h)\bigr)+\partial_\zrm U^\zrm(x_\h,z)\,.
\end  {aligned} 
\eeq

For simplicity of notation, for a function of the type~$f(\rho(x_\h))$, we shall often omit to note explicitly the dependence in~$x_\h$ and simply write~$f(\rho)$.

 Moreover, 
we will adopt the following notation for boundary layers: 
given~$g^\top$ and~$g^\bot$ two functions defined
on~$\R ^2\times \R^-$, representing boundary layer terms respectively near the surface~$\{z=0\}$ and near the bottom~$\{z=-\phi(\rho)\}$ of the domain~$\Omega_\phi$, we denote
\beq
\label {defin_form_BL}
g_{{\rm BL}}^\top (x_\h, z) \eqdefa g^\top \biggl(x_\h, \frac z {\sqrt {E}}\biggr) \andf 
g_{{\rm BL}}^\bot (x_\h, z) \eqdefa g^\bot \biggl(x_\h, -\frac  {z+\f(\rho(x_\h)) } { \d(\rho(x_\h))\sqrt {E} }\biggr)\,,
\eeq
where~$E$ is the Ekman number, defined  {(this will be justified later, see~(\ref{defE2}))} as~$E\eqdefa 2\b\e^2$
and~$\d=\d(\rho)$ is a function on~$[\rho_0,\infty[ $, which will be determined later on, see~(\ref{curved_BL2}).
Let us notice that the functions~$g^\surf$ and~$g^\bot $ are always assumed (sometimes implicitly) to have limit~$0$ when the variable~$\zeta$ tends to~$-\infty$.\footnote{In fact, they are always exponentially decaying at infinity.}  Here and in all that follows, the ``fast variable" is denoted~$\zeta$.
Notice also that the functions with an index~${\rm BL}$ depend on~$\e$ through the size of the boundary layer (respectively~$\sqrt E$ and~$ \d \sqrt {E} $).

Let us point out a major difference between boundary layer terms related to a flat boundary and those related to a curved boundary. In the case of a flat boundary, the derivative with respect to the horizontal variable~$x_\h$ does not generate terms of order~$-1$, see~$g^\top_{{\rm BL}}$ defined above. In the case of the term $g_{{\rm BL}}^\bot$ related to the curved boundary, instead, we have, for $j$ in~$\{1,2\}$,
the formula
\begin{align*}
\partial_{j} \bigl(g_{{\rm BL}}^\bot (x_\h, z) \bigr)
&=
(\partial_{j} g^\bot)\biggl(x_\h, -\frac  {z+\f(\rho(x_\h)) } { \d(\rho(x_h)) \sqrt {E}}\biggr)  \\
&\qquad -\partial_{j} \biggl (\frac  {z+\f(\rho(x_\h)) } { \d(\rho(x_h)) \sqrt {E} }\biggr) 
(\partial_\zeta g^\bot)\biggl(x_\h, -\frac  {z+\f(\rho(x_\h)) } { \d(\rho(x_h)) \sqrt {E}}\biggr) \\
& =(\partial_{j} g^\bot)\biggl(x_\h, -\frac  {z+\f(\rho(x_\h)) } { \d(\rho(x_h)) \sqrt {E}}\biggr)  \\
&\qquad +
 (\partial_{j}\rho)(x_\h)\frac { \d'(\rho(x_h))}{ \d(\rho(x_h))} \frac {z+\f(\rho(x_\h))} { \d(\rho(x_h)) \sqrt {E}}
(\partial_\zeta g^\bot)\biggl(x_\h, -\frac  {z+\f(\rho(x_\h)) } { \d(\rho(x_h)) \sqrt {E}}\biggr) \\
&\qquad\qquad
-(\partial_{j}\rho)(x_\h) \frac {\phi'(\rho(x_\h))} { \d(\rho(x_h)) \sqrt {E}}  (\partial_\zeta g^\bot)\biggl(x_\h, -\frac  {z+\f(\rho(x_\h)) } { \d(\rho(x_h)) \sqrt {E}}\biggr)\cdotp
\end{align*}
Notice that this can be written in a more compact way as
\beq
\label {derrBL1}
\partial_{j} g_{{\rm BL}}^\bot = \biggl(\partial_{j} g^\bot -   \frac { \d'(\rho) } { \d(\rho) } \partial_{j}\rho \ \zeta \partial_\zeta g^\bot\biggr)_{{\rm BL}} -
\frac {\phi'(\rho) } { \d(\rho) \sqrt {E} } \partial_{j}\rho \ (\partial_\zeta g^\bot)_{{\rm BL}}\,.
\eeq
Observe that the first term in~\eqref{derrBL1} is of order~$0$, whereas the second one is of order~$-1$  {due to the presence of~$\sqrt {E}$ at the denominator}, but vanishes identically  as soon as
the bottom is flat ($\f '\equiv 0$ is that case).

The fact that  horizontal derivatives of boundary layers at the bottom generate terms of order~$\e^{-1}$ is the reason why we assume that the viscosity is of size~$\e$ in all   directions and not only in the vertical one, as in\ccite {cdggbook} for instance.
This fact has a deep consequence also on the computation of the divergence of the boundary layer vector fields, which plays a crucial role in
the determination of the Ekman pumping term. Indeed, using the  above formula\refeq {derrBL1}, we infer~that
\beq
\label {compute_div_BL}
\begin{aligned}
\dive U_{{\rm BL}}^{\top} & = 
\big(\divh U^{\h,\top}\big)_{{\rm BL}} +\frac 1 {\sqrt {E}} \bigl(\partial_\zeta  U^{\zrm,\top}\bigr)_{{\rm BL}} \andf \\
\dive U_{{\rm BL}}^{\bot} & = 
\Big(\divh  U^{\h,\bot}-
\frac {\d'(\rho) } { \d(\rho)}  \zeta \ \partial_\zeta U^{h,\bot} \cdot\nabla_\h\rho \Big)_{{\rm BL}} \\
& \qquad\qquad\qquad\qquad\qquad\qquad{}
-\frac{\phi'(\rho)}{ \d(\rho) \sqrt {E}} \Bigl(\partial_\zeta U^{\zrm,\bot} + \nabla_\h\rho\cdot \partial_\z U^{\h, \bot}\Bigr)_{{\rm BL}}. 
\end{aligned}
\eeq

After these clarifications, let us present the general strategy of the proof of Proposition \ref{prop_U_app} and
how it is developed. 
Classically, in order for $U_{\app,\e}$ to be a good approximate solution, the leading order term should be close to the expected limit $\oline u$,
which should lie in the weak closure of the family~$(u_\e)_{0<\e\leq\e_0}$, hence, according to Proposition~\ref{Taylor-Proudman}, of the form
\begin{equation} \label{def:u^theta}
\oline u \eqdefa u^\theta(t,\r) \nabla_\h^\perp \r 
\qquad  \mbox{ where the function $u^\theta(t,\r)$ must be found. } 
\end{equation}
Thus,
we will choose the term of order $0$ in the interior, denoted~$U_{0,\rmint}$, to be close (in a sense specified below) to the profile~$\oline u$. Of course, this profile does not satisfy the Dirichlet boundary conditions, neither at the surface nor at the bottom of the ocean, so we have to introduce
correctors in the definition of~$U_{\app,\e}$, in the form of boundary layers.

Section\refer {interior0} is devoted to the construction of the leading order terms in velocity and pressure, denoted~$U_{0,\rmint} $ and~$P_{0,\rmint} $. 

Section\refer {BL0} starts with the computation of the boundary layer of order~$0$ on the surface. Despite the fact that this is classical (see for instance\ccite{cdggbook}), we expose it here as a warm up, and also as an opportunity to get familiar with the use of the frame~$\big(\nabla_\h \rho (x_\h), \nabla_\h^\perp\rho(x_\h), e_\zrm\bigr)$. Then we compute the boundary layer at the bottom and  determine the value of the function~$\d$, which is the cause of the term~$\ds \sqrt [4]{1+(\f' )^2}$ appearing in the definition of the modulated Ekman pumping term~$\lambda_\f$ (see Formula\refeq  {form of the limit} of the statement of Theorem\refer {Ekman_topo_radial_0}).

At the end of that section, we have  {computed} the  boundary layer term of order~$0$ at the surface and at the bottom in terms of~$U_{0,\rmint}$, which is still to be fully determined, in order to have   \beq
\label {defin_U_0BL}
\big(U_{0,\rmint} +U_{0,\rm BL}^\surf +U_{0,\rm BL}^\bot\bigr) _{| \partial \Om_\f} \sim 0 \, ,
\eeq
in the sense that it is exponentially small with~$\e$.
Note that this decomposition of the velocity field  gives rise to a similar decomposition of the pressure under the form
\[
P_{0,\rmint} +P_{0,\rm BL}^\surf +P_{0,\rm BL}^\bot\, .
\]

In   Section~\ref{cut-off at the shore}, we  deal with the problem of the shore.  As  the two boundary layers constructed previously  are of  size~$\e$, they  meet near the shore. In particular the property that a boundary layer on one boundary should be small near the opposite one, is no longer valid near the shore, when the distance between the surface and the bottom is of size  {smaller than}~$\e$. In order to bypass  this difficulty and to reattach the boundary layers, we introduce two cut-offs for each  boundary layer, one at a distance~$\e^{1-a}$ from the surface  or the bottom, and another one to avoid the shore.

More rigorously, for some positive~$a$ sufficiently close to~$1$ ({whose precise value will be fixed} in Section\refer {estim_error_terms}),  let us define 
\begin{equation} \label{def:O_e}
\cO_\e\eqdefa \bigl\{ x_\h\in \R_\h^2\,/\ \f(\rho(x_\h))\geq 2\e^{1-a}\bigr\}\, .
\end{equation}
This set represents the parts of the ocean with depth greater than or equal to~$2\e^{1-a}$,  {see Figure \ref{fig:cut-off} below.}
 \vspace{-0.8cm}
 \begin{figure}[H] 
\includegraphics[width=2.5in]{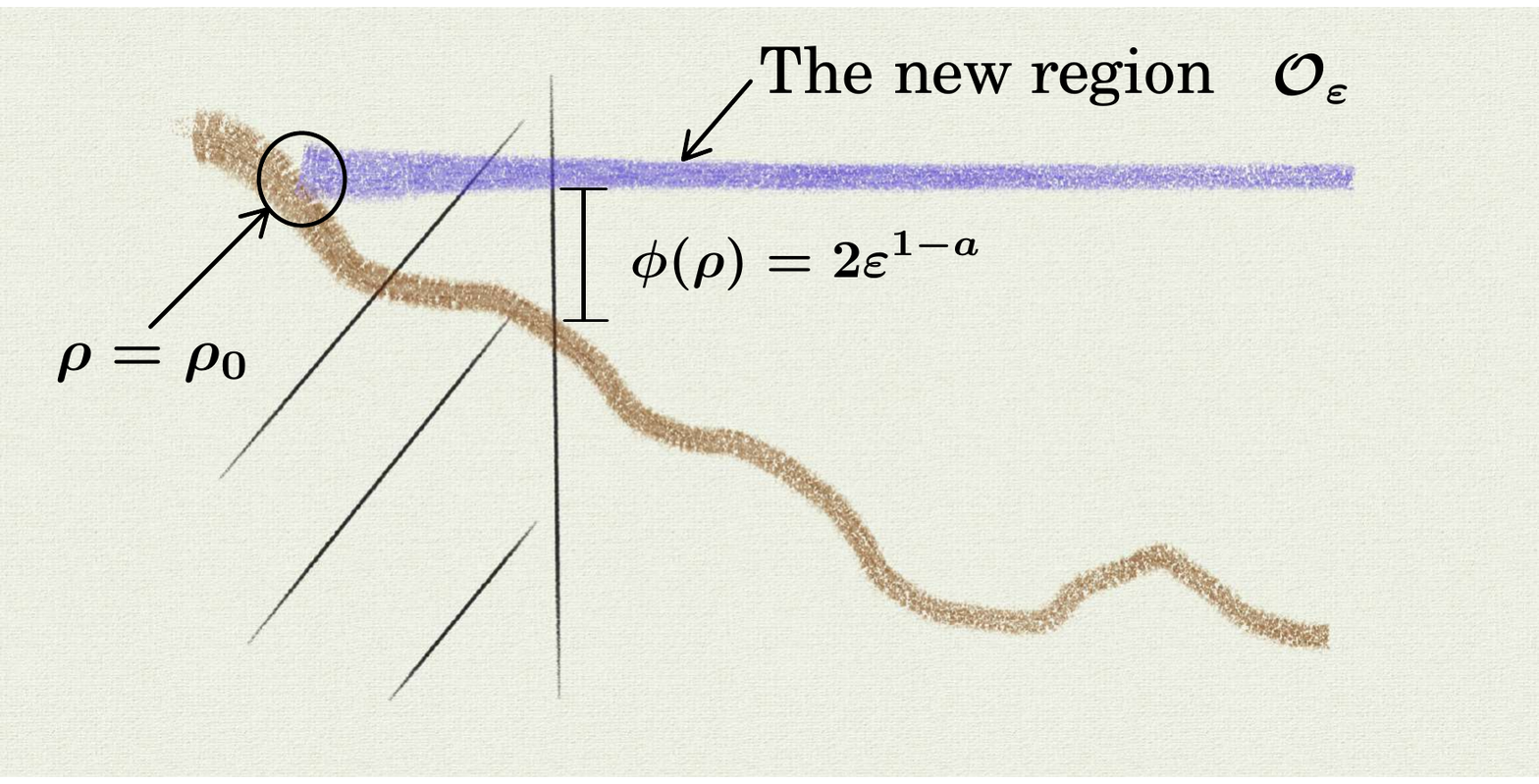} \vspace{-0.8cm}
\caption{\small The cut-off and the ocean region $\mc O_\e$.}
\label{fig:cut-off}
\end{figure}

Then, 
two functions~$\wt g^\surf$ and~$\wt g^\bot$ on~$\R^+\times \cO\times\R^-\times \R^+$ being given, representing respectively the boundary layers  {(with cut-offs)} at the surface and at the bottom of the domain~$\Om_\f$, we define 
\beq
\label {defin_BLtilde}
\begin {aligned}
\wt g_{{\rm BL}}^\top (t,x_\h,z) &\eqdefa  \wt g^\surf \biggl(t, x_\h,\frac z {\sqrt E}\virgp-\frac z {\e^{1-a}}\biggr)\andf\\
 \wt g_{\rm BL}^\bot  (t,x_\h,z) &\eqdefa \wt g^\bot \biggl(t,x_\h,-\frac {z+\f(\rho)}  {\d\sqrt E}\virgp\frac {z+\f(\rho)} {\e^{1-a}}\biggr)\,\cdotp
\end  {aligned}
\eeq
 These new, truncated boundary layers, now depend on~$\e$ both through the size of the boundary layers and through the cut-off. 
 The main point of the cut-off is that, as claimed by the forthcoming Proposition\refer  {Dirichlet_shore}, if two functions~$\wt g^\surf$ and~$\wt g^\bot$ have both their support included in~$\R^+\times \cO_\e\times\R^-\times ]0,1[$, then 
$$
\Supp \wt  g_{\rm BL}^\surf\cap  \Supp \wt g_{\rm BL}^\bot =\emptyset\, .
$$
The form\refeq {defin_BLtilde} of these new boundary layers leads us to introduce{, in order to ensure the divergence free condition,} an Ansatz which is different from the classical one (see for instance\ccite  {cdggbook}) and is of the following form:
\beq
\label {defin_new_Ansatz}
  {\wt U_{0,\rmint,\e} }  +
\wt U_{0,{\rm BL}}^\top+ \e^a\wt U_{a,{\rm BL}}^\top +\e \wt U^\top_{1, {\rm BL}} + \wt U_{0,{\rm BL}}^\bot +\e^a\wt U_{a,{\rm BL}}^\top+\e \wt U^\bot_{1, {\rm BL}} +\cdots,
\eeq
  {and, for the pressure function,
$$
\wt P_{0,\rmint,\e}
 + \wt P_{0, {\rm BL}}^\top +\e \wt P_{1, {\rm BL}}^\top +\wt P_{0, {\rm BL}}^\bot +\e \wt P_{1, {\rm BL}}^\bot+\cdots.
$$
Note that, a priori, we expect the pressure   not  to be so much affected by the truncation of the velocity. This will have to be confirmed in the computations that follow.}
We remark that, owing to the presence of the cut-off near the shore,   the leading order term~$\wt U_{0,\rmint,\e} $ now depends on $\e$: we set
\begin{equation} \label{newdefU0int}
\wt U_{0,\rmint,\e} \eqdefa u^\theta_\e(t,\r) \nabla_\h^\perp \r\,,  \qquad \mbox{ \tsl{i.e.} }\qquad
\wt U^\theta_{0,\rmint,\e}\eqdefa u^\theta_\e(t,\r)\,,
\end{equation}
where the function $u^\theta_\e$ is supported on $\R^+\times\mc O_\e$ and is an approximation of $u^\theta$ from \eqref{def:u^theta} on that set. 
Note that~$\wt U_{0,\rmint,\e} $ is divergence free. All the functions appearing in the Ansatz are now functions of~$u^\theta_\e$, hence of~$u^\theta$, to be determined.

In  Section~\ref{BL_order1}, we observe that the vector field defined {by $\wt U_{0,\rmint,\e} + \wt U_{0,{\rm BL}}^\top + \wt U_{0,{\rm BL}}^\bot$
now satisfies the equality in\refeq {defin_U_0BL}}, but does not satisfy the divergence free condition. Imposing the equality
$$
\dive  \left (\wt U_{0,\rmint,\e}   +
\wt U_{0,{\rm BL}}^\top+ \e^a\wt U_{a,{\rm BL}}^\top +\e \wt U^\top_{1, {\rm BL}} + \wt U_{0,{\rm BL}}^\bot +\e^a\wt U_{a,{\rm BL}}^\top+\e \wt U^\bot_{1, {\rm BL}} \right)=0
$$
allows to determine~$\wt U_a^\surf$ and~$\wt U_a^\top$  {in terms of~$\wt U_{0,\rmint,\e} $} and introduces some constraints on~$\wt U^\top_{1}$ and~$\wt U^\bot_{1}$.

In   Section~\ref{s:Dir-1}, we observe that the correction made previously to ensure the divergence free condition  leads to {the violation of} the Dirichlet boundary condition. Thus we introduce a vector field~$\e { \wt U_{1,\rmint,\e}}$  at the interior  to ensure the Dirichlet boundary condition. Let us notice that plugging this term~$\e\wt U_{1,\rmint,\e}$ {into $(SC_\e)$ (recall its definition just above Proposition \ref{prop_U_app})}
will produce a term of order~1 in~$\e$, namely~$  e_{\zrm} \wedge \wt U_{1,\rmint,\e}$. This term determines the equation satisfied by~$ \wt U_{0,\rmint,\e} $ and puts in light the Ekman pumping phenomenon. Thanks to $\wt U_{1,\rmint,\e}$, we also fully determine the boundary layer terms
of order $1$, namely $\wt U_{1,\rm BL}^\top$ and $\wt U^\bot_{1,\rm BL}$.

Section~\ref{full approximate} is devoted to the end of the construction of the full approximate solution couple~$(U_{\app,\e} ,P_{\app,\e}) $, which requires a last correction
in order to ensure the divergence free constraint, yet without violating the Dirichlet boundary conditions.

In  Section~\ref {estim_error_terms}, we prove that the approximated solution constructed in the previous sections is truly a good approximation
of the target profile $\oline u$, by a precise estimation of the error terms. This will prove Statement~\eqref {prop_U_app_eq2} of Proposition~\ref{prop_U_app}.
Statements~\eqref{bdd gradients} and~\eqref{small BL} are also proved in that section.
 
Finally, in Section~\ref{s:error-nonlin} we check that the couple~$(U_{\app,\e} ,P_{\app,\e}) $ not only solves approximately the linear problem~$(SC_\e)$ (recall the system~$(SC_{\app,\e} )$ given in Proposition~\ref{prop_U_app}),
but also the nonlinear system~$(NSC_\e)$, in some sense. This ends the proof of Propositions~\ref{prop_U_app} and~\ref {structure_NLterm}.

\section  {The interior terms at order~$0$}
\label{interior0}

The first step of the analysis consists in inserting   the first order term of the Ansatz in velocity and pressure, namely the couple~$\big(U_{0,\rmint} ,P_{0,\rmint}  \big)$, into the linear equations~$(SC_\e)$, {which we recall here for the reader's convenience:
\[
\left\{
\begin{array}{c} 
\ds \partial_t u_\e - \e\b\D u_\e  +\frac 1 \e e_{\zrm} \wedge u_\e + =-\frac 1 \e \nabla p_\e\,,\\
\ds  {u_\e}_{|t=0} = u_0\,,\\
\ds \dive u_\e =0\andf {u_\e}_{|\partial\Om_\vf }=0\,  .
\end{array}
\right.
\leqno(SC_\e)
\]
} Identifying and canceling the highest order terms provides, in the frame~$\big(\nabla_\h\r,\nabla^\perp_\h\r,e_\zrm\big)$,
$$
 \left\{
\begin{array} {c}
\ds  -  U_{0,\rmint} ^{\theta} =- \nabla_\h P_{0,\rmint} \cdot  \nabla_\h\r \\\\
\ds   U_{0,\rmint} ^{\rho }  = - \nabla_\h P_{0,\rmint} \cdot  \nabla^\perp_\h\r \\\\
\ds 0 = - \partial_{\rm z} P_{0,\rmint}\\\\ 
\ds \dive U_{0,\rmint} = 0 \, .
\end{array}
\right.
$$
Thus as in the proof of Proposition~\ref{Taylor-Proudman} we   find    that
\beq\label{defU0thetaint}U_{0,\rmint} ^{\theta } (t,x_\h,z)=  u^\theta(t,\rho(x_\h)) \, , \quad U_{0,\rmint} ^{\rho } = U_{0,\rmint} ^{\rm z }= 0\, .
\eeq  
Finally
\begin{equation}\label{defp1eps}
{P_{0,\rm int} (t,\rho,z)} \eqdefa   -\int^\rho_0 u^\theta(t,\s)\,d\s \, .
\end{equation}


\section  {The Boundary layers at order~$0$}
\label{BL0}

 We are looking for~$U_{\app,\e}$ as an approximation of~$U_{0,\rmint}$ computed in~\eqref{defU0thetaint} above in terms of an unknown function~$u^\theta$. As it does not satisfy the boundary conditions, neither on the surface nor at the bottom,   we need to introduce boundary layer corrections for the velocity field and the pressure, under the form \refeq {defin_form_BL}. In order to determine those corrections, we insert those terms in the equation~$(SC_\e)$ and try to make each term of the expansion, in   powers of~$\e^{-1}$, equal to~$0$.
As recalled above, from now on all the vector fields will be expressed in the frame~$\big(\nabla_\h\r,\nabla^\perp_\h\r,e_\zrm\big)$  rather than the usual Cartesian frame.

\medskip

Let us start with the easier case of the boundary terms on the surface; the computations are classical (see for instance\ccite {cdggbook}) and we reproduce them as a warm up.  {We assume that~$E/\e$ goes to zero with~$\e$. Then the highest order term is the  term of power~$(\e\sqrt {E})^{-1} $}, appearing in the equation on the third component of system~$(SC_\e)$:
$
\displaystyle -\frac 1 {\e \sqrt {E}} \partial_\zeta P_{0}^\top =0
$.
This implies that~$\partial_\zeta P_{0}^\top=0$ , and since~$P_{0}^\top$ tends to~$0$ when~$\zeta$ tends to~$-\infty$,\beq\label{P0surfzero}
P_{0,{\rm BL}}^\top= 0\, .
\eeq
Next, we want to cancel the  {next terms concerning the top boundary layer, in~$(SC_\e)$ and on the equation on the divergence}.  This  implies the relations
$$
 \left\{
\begin{array} {c}
\ds  -\frac {\e \b}  E \partial_\zeta^2U_0^{\rho ,\top} - \frac 1\e U_0^{\theta,\top} =0 \\\\
\ds -\frac {\e \b}  E  \partial_\zeta^2U_0^{\theta,\top} + \frac 1\e U_0^{\rho ,\top}  = 0\\\\
\ds  -\frac {\e \b}  E  \partial_\zeta^2U_0^{\zrm,\top}  = -\frac 1 {\sqrt E } \partial_\zeta P_1^\top\\\\
\ds - \frac 1 {\sqrt E } \partial_\zeta U_0^{\zrm,\top} =0\,.
\end{array}
\right.
$$
The  first two equations imply that the two terms~$\ds \frac {\e \b}  E $ and $\ds \frac 1 \e$ must be equivalent. This justifies the choice
\beq
\label{defE2}
E\eqdefa 2\b\e^2\,,
\eeq
which is the classical definition of the Ekman number. The system  thus becomes
\beq
\label{BLsurf}
 \left\{
\begin{array} {c}
\ds  -\frac 12 \partial_\zeta^2U_0^{\rho ,\top} -  U_0^{\theta,\top} =0 \\\\
\ds -\frac 12 \partial_\zeta^2U_0^{\theta,\top} + U_0^{\rho ,\top}  = 0\\\\
\ds  -\frac 12 \partial_\zeta^2U_0^{\zrm,\top}  = -\frac 1 {\sqrt {2\b} } \partial_\zeta P_1^\top\\ 
\ds - \frac 1 {\sqrt {2\b} } \partial_\zeta U_0^{\zrm,\top} =0\,.
\end{array}
\right.
\eeq
Using again  the fact that the boundary layer functions have fast decay at infinity,  we infer~that 
\begin{equation} \label{eq:U_0-P_1-SURF}
U_0^{\zrm,\top} \equiv 0\andf P_1^\top\equiv 0\, .
\end{equation}
The above system then becomes
\begin{equation} \label{BL0system}
\left\{
\begin{array} {c}
\ds - \frac 12 \partial_\zeta^2U_0^{\rho ,\top} -  U_0^{\theta,\top} =0 \\\\
\ds- \frac 12  \partial_\zeta^2U_0^{\theta,\top} + U_0^{\rho ,\top}  = 0\,.
\end{array}
\right.
\end{equation}
Let us look for~$U_{0,{\rm BL}}^\top$  in the basis~$\big(\nabla_\h\r,\nabla^\perp_\h\r,e_\zrm\big)$,
under the form $$\begin{pmatrix} 
 U_{0}^{\rad,\surf}\\U_{0}^{\theta,\top} 
 \end{pmatrix}
 =M^{\top} (\zeta)
 \begin{pmatrix} 
0\\u^\theta\,,
\end{pmatrix}
$$
where~$u^\theta$ is the function introduced above in~(\ref{def:u^theta}), to be determined. The matrix~$M^\top(\z)$ will be determined in the following steps.
System~(\ref{BL0system}) becomes
\beq
\label {flat_BL1}
\left \{\begin{array} {c}
\ds -\frac 1 2 \frac {d^2} {d\zeta^2}\ M^\top(\zeta)+  {\mc R M^\top (\zeta) }=0\, \\\\
M^\top (0)=-\Id_{\R^2}\andf M^\top (-\infty)=0\, ,
\end{array}
\right.
\eeq
 {where the rotation matrix $\mc R$ is defined by $\mc R =\begin{pmatrix} 0 & -1\\ 1 & 0 \end{pmatrix} $}.
The solution to System \eqref{flat_BL1} is~$\ds 
M^\top (\zeta) = -e^\zeta \begin{pmatrix} \cos \zeta & -\sin\zeta\\ \sin \zeta & \cos \zeta
\end{pmatrix}$.
This implies
\beq
\label {flat_BL2}
\begin{aligned}
U_0^\top (t,\rho,\zeta) &= -u^\theta (t,\rho)   e^\zeta
\begin{pmatrix}
-\sin\zeta \\ \cos\zeta \\ 0
\end{pmatrix}\,,
\quad \hbox{thus}\\
  U_{0,{\rm BL}}^\top (t,\rho,z) &= 
-u^\theta (t,\rho)    e^{\frac z {\sqrt E}} \begin{pmatrix}
\ds -\sin\biggl( \frac z {\sqrt E}\biggr) \\ \ds \cos\biggl( \frac z {\sqrt E}\biggr) \\ 0
\end{pmatrix} \quad \mbox{ in the basis }\quad \big(\nabla_\h\r,\nabla_\h^\perp\r,e_\zrm\big)\,.
\end{aligned}
\eeq

Now, let us study the boundary layer terms at the bottom. Again, we start by considering the terms of order~$-2$ {, assuming a priori that~$\delta$ is of order~$0$}: the equation on the third component of~$(SC_\e)$ ensures that 
$\displaystyle
\frac 1  {\e\d\sqrt {E}} \partial_\zeta P_{0}^\bot \equiv 0$, thus~$\partial_\zeta P_{0}^\bot$ is identically~$0$. Since~$P_{0}^\bot$ must vanish for~$\zeta=-\infty$, we find
\beq\label{P0botzero}
P_{0,{\rm BL}}^\bot = 0 \, .
\eeq
Next, we compute the term of order~$-1$ in the equation of the bottom boundary layer. As for the surface case, we want to cancel the terms of order~$-1$
of the equation at the bottom boundary layer and of the divergence of~$U_{0,{\rm BL}}^\bot$. Using Formula\refeq {derrBL1} and its
corollary\refeq {compute_div_BL} about the divergence, we infer that 
$$
 {({\rm BL})_{\bot}}\quad\left\{
\begin{array} {c}
\ds -\frac {1+\f '^2}{ 2 \d^2} \partial_\zeta^2 U_{0}^{\rho,\bot} -  U_0^{\theta,\bot}  = \frac {\f '} {\d\sqrt {2\b} } \partial_\zeta P_1^\bot \\\\
\ds -\frac {1+\f '^2}{ 2 \d^2} \partial_\zeta^2U_0^{\theta,\bot} +  U_{0}^{\rho,\bot}  = 0\\\\
\ds -\frac {1+\f '^2}{ 2 \d^2} \partial_\zeta^2U_0^{\zrm,\bot}  = \frac 1 {\d\sqrt {2\b} } \partial_\zeta P_1^\bot\\\\
\ds -\frac {\f '} {\d\sqrt {2\b} } \partial_\zeta  U_{0}^{\rho,\bot}  - \frac 1 {\d\sqrt {2\b} } \partial_\zeta U_0^{\zrm,\bot} =0\,.
\end{array}
\right.
$$
Because the boundary layer functions have value~$0$  at~$-\infty$, the last two equations become
\beq
\label {curved_BL1}
U_0^{\zrm,\bot} =-\f '  U_{0}^{\rho,\bot} \andf P_1^{\bot} = -\frac {\sqrt {2\b}}  2\frac {1+\f '^2} \d  \f ' \partial_\zeta  U_{0}^{\rho,\bot}\, .
\eeq

\begin{remark}
{\sl Because of Equation~\eqref{curved_BL1}, we see that the pressure term~$P_{1,{\rm BL}}^\bot$ is not identically~$0$. This marks a difference with
the classical case of a flat bottom, see Formula~\eqref{eq:U_0-P_1-SURF}.}
\end{remark}

Relation~\eqref{curved_BL1} allows to recast the above system in the following reduced form:
$$ 
\left\{
\begin{array} {c}
\ds -\frac {1+\f '^2}{ 2 \d^2} \partial_\zeta^2 U_{0}^{\rho,\bot} -  U_0^{\theta,\bot}  = \frac {\f '^2(1+\f '^2)} {2\d^2 } \partial_\zeta^2 U_{0}^{\rho,\bot}\\\\
\ds -\frac {1+\f '^2}{ 2 \d^2} \partial_\zeta^2U_0^{\theta,\bot} +  U_{0}^{\rho,\bot}  = 0\, ,
\end{array}
\right.
$$
which in turn writes as
\beq
\label {curved_BL01}
\left\{
\begin{array} {c}
\ds -\frac {(1+\f '^2)^2}{ 2 \d^2} \partial_\zeta^2 U_{0}^{\rho,\bot} -  U_0^{\theta,\bot}  = 0\\\\
\ds -\frac {1+\f '^2}{ 2 \d^2} \partial_\zeta^2U_0^{\theta,\bot} +  U_{0}^{\rho,\bot}  = 0 \,.
\end{array}
\right.
\eeq

We reduce the above~$2\times 2$ linear system of order~$2$ to a linear ordinary differential equation of order~$4$,
that is
$$
\partial_\zeta^4 U_0^{\theta,\bot} =-\frac {4\d^4} {(1+\f '^2)^3}U_0^{\theta,\bot} \,.
$$
Looking for the function~$U_0^\bot$ under the form~$v(\rho) g(\zeta)$, we see that the only choice for~$\d$ is 
\beq
\label {curved_BL2}
\d^4  =(1+\f '^2)^3 \qquad\hbox{i.e.}\qquad \d\big(\rho(x_\h)\big)= \Big(1+(\f ')^2\big(\rho(x_\h)\big)\Big)^{\frac 3 4}.
\eeq
Moreover, the solutions of the ordinary differential equation~$\om^{(4)}=-4\om$ are of the
form~$\om=\sum_{\pm} e^{\pm \zeta} \bigl(A_\pm \cos \zeta+B_\pm \sin \zeta\bigr)$.
As the function~$U_0^{\theta,\bot} $ must tend to~$0$ when~$\zeta$ tends to~$-\infty$, it must be the form
\[
U_0^{\theta,\bot} = e^{ \zeta} \bigl(A \cos \zeta+B\sin \zeta\bigr).
\]
As we have~$(\partial_\zeta^2 U_0^{\theta,\bot})_{|\z=0}=0$ by the second equation in~\eqref{curved_BL01}, we deduce that~$B=0$. In addition, the fact
that~$U_0^{\theta,\bot}(t,\rho,0)=-u^\theta  (t,\rho) $ implies that~$U_0^{\theta,\bot}(t,\rho,\zeta) = -u^\theta  (t,\rho)  e^\zeta\cos\zeta$.
In the end, we get the formula, as usual expressed in the reference frame ~$\big(\nabla_\h\r,\nabla_\h^\perp\r,e_\zrm\big)$,
\[
U_0^\bot (t,\rho,\zeta) = -u^\theta (t,\rho)   e^{\zeta} \begin{pmatrix}
-\d^{-\frac 23} \sin\zeta\\
\cos\zeta\\
\f '\d^{-\frac 23} \sin\zeta
\end{pmatrix}\,,
\]
which can also be written, again in the basis~$\big(\nabla_\h\r,\nabla_\h^\perp\r,e_\zrm\big)$,  as
\beq
\label {curved_BL4}
U_{0,{\rm BL}}^\bot (t,\rho,z) = -u^\theta (t,\rho)   e^{- \frac {z+\f } {\d \sqrt E} }  \begin{pmatrix}
\ds \d^{-\frac 23} \sin \biggl ( \frac {z+\f } {\d \sqrt E} \biggr)\\
\ds \cos\biggl ( \frac {z+\f } {\d \sqrt E} \biggr)\\
\ds -\f '\d^{-\frac 23} \sin \biggl ( \frac {z+\f } {\d \sqrt E} \biggr)
\end{pmatrix}.
\eeq
Recall that $\phi=\phi(\rho)$ and $\d=\d(\rho)$ in the formulas above, as well as~$E= 2\beta\e^2$.

Finally, returning to the approximate pressure, we recall that
$$
P_{0,{\rm BL}}^{\surf}  = P_{1,{\rm BL}}^{\surf}= P_{0,{\rm BL}}^{\bot}  = 0 $$
and we have, thanks to~\eqref{curved_BL1},
\beq
\label {P1bot}
P_{1,{\rm BL}}^{\bot} (x_\h,z)= -\frac {\sqrt {2\b}}  2\frac {1+\f '^2} {\d^\frac53}  \f ' 
u^\theta (t,\rho)    e^{- \frac {z+\f } {\d \sqrt E} }  \Big(
 \sin \bigl ( \frac {z+\f } {\d \sqrt E} \bigr)+ \cos \bigl ( \frac {z+\f } {\d \sqrt E} \bigr)
\Big)\, .
\eeq


\section {The cut-off of boundary layers of order~$0$ near the shore} \label{cut-off at the shore}

In this section,  we introduce a cut-off near the shore, in order to restrict~$x_\h$   to~$\cO_\e$, the part of the ocean with depth greater than or equal to~$2\e^{1-a}$, {recall Definition~(\ref{def:O_e}) above.}
Let us consider a non-negative function~$\chi$ of~$\cD([0,1[)$ with value~$1$ in a neighbourhood of~$[0,1/2]$, and let us define  the function 
\begin{equation} \label{def:u^theta_e}
u^\theta_\e(t,\rho) \eqdefa \left(1 - \chi\Big(\frac{\phi(\rho)}{\e^{1-a}}\Big)\right) \ u^\theta(t,\rho)\,,
\end{equation}
where $u^\theta(t,\rho)$ is  {the function appearing in~\eqref{def:u^theta}  and which} has to be determined.
Let us observe that
\[
\supp u^\theta_\e \subset \cO_\e\,.
\]
Accordingly we set, as presented in~(\ref{newdefU0int}),
\beq\label{newdefU0int2}
\wt U_{0,\rmint,\e} {(t,x_\h,z) \eqdefa u^\theta_\e(t,\r(x_\h)) \nabla_\h^\perp \r(x_\h)}\, ,
\eeq
and~\eqref{defp1eps} becomes
\begin{equation}\label{new defp1eps}
\wt P_{0,\rm int,\e}  {(t,\rho,z)} \eqdefa  -\int^\rho_0 u_\e^\theta(t,\s)\,d\s \, .
\end{equation}

Then, in agreement with\refeq  {flat_BL2},\refeq {curved_BL4}  and\refeq {defin_BLtilde}, we define 
  \beq
\label {flat_BL2_cut-off}
\begin {aligned}
\wt U_{0,{\rm BL}}^\top (t,\rho,z) &\eqdefa
-u_\e^\theta (t,\rho)  \chi \biggl( -\frac z {\e^{1-a}}\biggr)   e^{\frac z {\sqrt E}} \begin{pmatrix}
\ds -\sin\biggl( \frac z {\sqrt E}\biggr) \\ \ds \cos\biggl( \frac z {\sqrt E}\biggr) \\ 0
\end{pmatrix} \andf\\
\wt U_{0,{\rm BL}}^\bot (t,\rho,z) &\eqdefa -u_\e^\theta (t,\rho)   \chi \biggl( \frac {z+\f} {\e^{1-a}}\biggr)  e^{- \frac {z+\f } {\d \sqrt E} }  \begin{pmatrix}
\ds \d^{-\frac 23} \sin \biggl ( \frac {z+\f } {\d \sqrt E} \biggr)\\
\ds \cos\biggl ( \frac {z+\f  } {\d \sqrt E} \biggr)\\
\ds -\f '\d^{-\frac 23} \sin \biggl ( \frac {z+\f } {\d \sqrt E} \biggr)
\end{pmatrix}.
\end{aligned}
\eeq
 and the corresponding profiles
 \beq
\label {flat_BL2_cut-off_0}
\begin{aligned}
\wt U_0^\top (t,\rho,\zeta_1,\zeta_2) &\eqdefa -u_\e^\theta (t,\rho)   \chi (\zeta_2)e^{\zeta_1}
\begin{pmatrix}
-\sin\zeta_1 \\ \cos\zeta_1 \\ 0
\end{pmatrix}\andf\\
 \wt U_0^\bot (t,\rho,\zeta_1,\zeta_2) &\eqdefa -u_\e^\theta (t,\rho)  \chi (\zeta_2)  e^{\zeta_1} 
 \begin{pmatrix}
-\d^{-\frac 23} \sin\zeta_1\\
\cos\zeta_1\\
\f '\d^{-\frac 23} \sin\zeta_1
\end{pmatrix},
\end{aligned}
\eeq
where we recall that the components of the vectors refer to the basis~$\big(\nabla_\h\rho,\nabla_\h^\perp\rho,e_\zrm\big)$. 
From now on, we denote by~$\zeta_1$ the boundary layer variable, and by~$\zeta_2$ the cut-off variable.
The main interest of introducing the cut-off $\chi$ is given by the following proposition.
\begin{prop}
\label {Dirichlet_shore}
{\sl
With the above notation, one has
$$
\bigl(\wt U_{0,\rmint,\e} +\wt U_{0,{\rm BL}}^\top+\wt U_{0,{\rm BL}}^\bot\bigr)_{|\partial\Om_\f} =0\, ,
$$
 {and the supports of~$\wt U_{0,{\rm BL}}^\top$ and~$\wt U_{0,{\rm BL}}^\bot$ are disjoint.}
}
\end {prop}

\begin {proof}
By construction, we already know that  
$$
(\wt U_{0,\rmint,\e} +\wt U^\surf_{0,{\rm BL}} )_{|z=0} = (\wt U_{0,\rmint,\e} + \wt U^\bot_{0,{\rm BL}} )_{|z=-\f}=0\, .
$$
Let us check that~${\wt U^\surf_{0,{\rm BL}}{}}_{|z=-\f}= {\wt U^\bot_{0,{\rm BL}}{}}_{|z=0}=0$. As a matter of fact, we notice that,
for any~$\rho$ in~$\f^{-1} ([2\e^{1-a},H])$,  the following properties hold true:
\beq
\label {Dirichlet_shore_demoeq1}
\begin {aligned}
z\geq -\e^{1-a} & \Longrightarrow   z+\f(\rho) \geq \f(\rho) -\e^{1-a}\geq \e^{1-a} \andf \\
z+\f(\rho)\leq \e^{1-a} & \Longrightarrow   z\leq \e^{1-a} -\f(\rho)\leq -\e^{1-a}.
\end {aligned}
\eeq
These properties imply that the support of the two boundary layers~$\wt U_{0,{\rm BL}}^\surf$ and~$\wt U_{0,{\rm BL}}^\bot$ are disjoint if the support of~$u^\theta_\e$ is included in~$\R^+\times\cO_\e $. Indeed, we have
\beno
z\geq -\e^{1-a} & \Longrightarrow & \chi \biggl( \frac {z+\f(\rho)} {\e^{1-a}}\biggr)=0  \andf \\
z+\f(\rho)\leq \e^{1-a} & \Longrightarrow &   \chi \biggl( -\frac z {\e^{1-a}}\biggr) =0\,,
\eeno
thus concluding the proof of the proposition.
\end {proof}


\section {The divergence free condition at order~$0$ and boundary layer terms of order~$a$ and~$1$ }
\label {BL_order1}

As the support of the two boundary layers~$\wt U_{0,{\rm BL}}^\surf$ and~$\wt U_{0,{\rm BL}}^\bot$ are disjoint (see Proposition~\ref{Dirichlet_shore}),
these two terms do not interact with each other. However, as we shall see below, the introduction of the cut-off entails the appearance of new terms of order $\e^{-1+a}$
when computing the divergence of the boundary layers. In order to cancel out  those terms, we must ``correct'' again our Ansatz.
Thus, adopting the notation introduced before Proposition \ref{Dirichlet_shore},
we shall look for the approximate solution $U_{\app,\e}$ under the form \eqref{defin_new_Ansatz}, namely
$$
\wt U_{0,\rmint,\e}   + \wt U_{0,{\rm BL}}^\top +  \wt U_{0,{\rm BL}}^\bot + \e^a\wt  U^\top_{a, {\rm BL}}  +\e^a \wt U^\bot_{a, {\rm BL}} + \e \wt U_{1,\rmint,\e}  
+ \e\wt  U^\top_{1, {\rm BL}}  +\e \wt U^\bot_{1, {\rm BL}}+\ldots.
$$
Recall that the terms~$\wt U_{0,\rmint,\e}$, $\wt U_{0,{\rm BL}}^\top$ and $\wt U_{0,{\rm BL}}^\bot$ have already been computed  in~(\ref{newdefU0int2}) and~(\ref{flat_BL2_cut-off})
(modulo the expression of $u^\theta$, which will be determined later). We now look for the other terms
of the expansion, in order for the divergence to vanish.

\medskip
 Notice that the vector field~$U_{0,\rmint,\e}$ is divergence free.  So let us  check how far we are from canceling the divergence of~$U_{\app,\e}$ by computing  the   quantities~${\rm Div}^\top_\e$ and~${\rm Div}^\bot_\e$ respectively defined by
\beq
\label {BL_order1_eq1}
\begin {aligned}
({\rm Div}^\top_\e)_{{\rm BL}} &\eqdefa \dive \bigl (\wt U_{0,{\rm BL}}^\top + \e^a\wt  U^\top_{a, {\rm BL}}  
+ \e\wt  U^\top_{1, {\rm BL}}  \bigr)\andf \\
({\rm Div}^\bot_\e)_{{\rm BL}} & \eqdefa \dive \bigl ( \wt U_{0,{\rm BL}}^\bot   +\e^a \wt U^\bot_{a, {\rm BL}}   +\e \wt U^\bot_{1, {\rm BL}}\bigr).
\end {aligned}
\eeq
We point out that, throughout this section, all the computations will be performed in the set of variables $\big(t,x_\h,\z_1,\z_2\big)$
(or $\big(t,\rho(x_\h),\z_1,\z_2\big)$ when convenient), as introduced in the notation~\eqref {defin_BLtilde}.
  
\medbreak
We start by considering the divergence at the surface of the ocean. By definition \eqref{flat_BL2_cut-off_0}
of the boundary layers and using Formula\refeq {div_rotating_frame_coodinates} for the divergence, assuming that the~$\theta$-components of~$\widetilde U_{a}^{\top}$ and~$\widetilde U_{1}^{\top}$ are of the form
$$
\widetilde U_{a}^{\theta,\top}= \widetilde U_{a}^{\theta,\top}(t,\rho,\z_1,\z_2)\andf \widetilde U_{1}^{\theta,\top}= \widetilde U_{1}^{\theta,\top}(t,\rho,\z_1,\z_2)\, ,
$$
 we get
\begin{align*}
{\rm Div}^\top_\e & = \dive_\h \bigl (u^\theta_\e \nabla_\h \rho  (x_\h)\bigr)  \chi(\zeta_2)e^{\zeta_1} \sin \zeta_1 \\
&\qquad + \e^a \dive_\h \bigl ( \widetilde U_{a}^{\rad,\surf}  \nabla_\h \rho  (x_\h)\bigr)
+\e \dive_\h \bigl ( \widetilde U_{1}^{\rad,\surf}  \nabla_\h \rho  (x_\h)\bigr)
\\
&\qquad\qquad{}- \e^{-1+2a} \partial_{\zeta_2} \widetilde U_{a}^{\zrm,\surf}
-\e^{a} \partial_{\zeta_2}\widetilde U_1^{\zrm,\surf} + \e^{-1+a} \partial_{\zeta_1} \widetilde U_{a}^{\zrm,\surf}+ \frac  1 {\sqrt {2\b}} \partial_{\zeta_1} \widetilde U_{1}^{\zrm,\surf} .
\end{align*}
The terms~$  \partial_{\zeta_1} \widetilde U_{a}^{\zrm,\surf}$ and $  \partial_{\zeta_2} \widetilde U_{a}^{\zrm,\surf}$ must therefore be~$0$. As~$ \widetilde U_{a}^{\zrm,\surf}$ must tend
to~$0$ when~$\zeta_1$ tends to infinity and similarly when $\z_2$ tends to infinity, we find
\beq \label{Uazzero}
 \widetilde U_{a,\rm BL}^{\zrm,\surf}= 0 \, .
\eeq 
Thus the above formula reduces to
\begin{align*}
{\rm Div}^\top_\e  &=   \dive_\h \bigl (u^\theta_\e  \nabla_\h \rho  (x_\h)\bigr)  \chi(\zeta_2)e^{\zeta_1} \sin \zeta_1
+\e^a \dive_\h \bigl (\widetilde U_{a}^{\rad,\surf}  \nabla_\h \rho  (x_\h)\bigr) \\
&\qquad\qquad\qquad
+\e \dive_\h \bigl ( \widetilde U_{1}^{\rad,\surf} \nabla_\h \rho  (x_\h)\bigr)
-\e^{a} \partial_{\zeta_2}\widetilde U_1^{\zrm,\surf} 
+ \frac  1 {\sqrt {2\b}} \partial_{\zeta_1} \widetilde U_{1}^{\zrm,\surf} .
\end{align*}
In order to cancel the terms of order~$0$ and~$a$, we need
\beq
\label {BL_order1_eq2}
\begin {aligned} 
 \partial_{\zeta_1} \widetilde U_{1}^{\zrm,\surf}  & = -\sqrt {2\b} \dive_\h \bigl (u^\theta_\e  \nabla_\h \rho  (x_\h)\bigr)  \chi(\zeta_2)e^{\zeta_1} \sin \zeta_1\andf \\
 \dive_\h \bigl ( \widetilde U_{a}^{\rad,\surf}  \nabla_\h \rho  (x_\h)\bigr) &=  \partial_{\zeta_2}\widetilde U_1^{\zrm,\surf} .
\end {aligned}
\eeq
This gives
\beq
\label {BL_order1_eq3}
\begin {aligned} 
 \widetilde U_{1}^{\zrm,\surf}  & = \frac {\sqrt {2\b}} 2  \dive_\h \bigl (u^\theta_\e  \nabla_\h \rho  (x_\h)\bigr)  \chi(\zeta_2)e^{\zeta_1} (\cos\zeta_1-\sin \zeta_1)\andf \\
 \widetilde U_{a}^{\rad,\surf}  &=  \frac {\sqrt {2\b}} 2   u^\theta_\e  \chi'(\zeta_2)e^{\zeta_1} (\cos\zeta_1-\sin \zeta_1)\, . \end {aligned}
\eeq
With these choices, we finally have
\beq
\label {BL_order1_eq3a}
{\rm Div}^\top_\e =\e \dive_\h \bigl (\widetilde U_{1}^{\rad,\surf} \nabla_\h \rho  (x_\h)\bigr),
\eeq
that is,
$$({\rm Div}^\top_\e)_{{\rm BL}} =\e \dive_\h \bigl (\wt U_{1,{\rm BL}}^{\rad,\surf} \nabla_\h \rho  (x_\h)\bigr)\, .$$
 Note that the function $ \wt U_{1,{\rm BL}}^{\rad,\surf}$ still has to be fixed.

\medskip

The case of the boundary layer  at the curved bottom requires more care. 
Again, we assume that~$\widetilde U_{a}^{\theta,\bot}$ and~$\widetilde U_{a}^{\theta,\bot}$ take the form
$$
\widetilde U_{a}^{\theta,\bot}= \widetilde U_{a}^{\theta,\bot}(t,\rho,\z_1,\z_2) \andf \widetilde U_{a}^{\theta,\bot}= \widetilde U_{1}^{\theta,\bot}(t,\rho,\z_1,\z_2)\, .
$$
Using\refeq  {compute_div_BL} and\refeq {flat_BL2_cut-off_0}, we can compute
\begin{align*}
{\rm Div}^\bot_\e & = \dive_\h \bigl (\widetilde U_{0}^{\rad,\bot} \nabla_\h \rho  (x_\h)\bigr) -\frac {\d'} \d \zeta_1\partial_{\zeta_1} \widetilde U_{0}^{\rad,\bot}
+\e^a \Bigl(\dive_\h \bigl ( \widetilde U_{a}^{\rad,\bot}  \nabla_\h \rho  (x_\h)\bigr)-\frac {\d'} \d \zeta_1\partial_{\zeta_1}  \widetilde U_{a}^{\rad,\bot}\Bigr)\\
&\qquad{}\ 
+\e \Bigl(\dive_\h \bigl ( \widetilde U_{1}^{\rad,\bot}  \nabla_\h \rho  (x_\h)\bigr) -\frac {\d'} \d \zeta_1\partial_{\zeta_1}  \widetilde U_{1}^{\rad,\bot}\Bigr)
+ \e^{-1+a} \partial_{\zeta_2} \bigl( \widetilde U_{0}^{\zrm,\bot} +\f' \widetilde U_{0}^{\rad,\bot}\bigr)\\
&\qquad\qquad{}
+\e^{-1+2a} \partial_{\zeta_2} \bigl( \widetilde U_{a}^{\zrm,\bot} +\f' \widetilde U_{a}^{\rad,\bot}\bigr)
+ \e^{a} \partial_{\zeta_2} \bigl( \widetilde U_{1}^{\zrm,\bot} +\f' \widetilde U_{1}^{\rad,\bot}\bigr)\\
&\qquad\qquad\qquad{}- \frac {1} {\d \sqrt E} \partial_{\zeta_1} \bigl( \widetilde U_{0}^{\zrm,\bot} +\f' \widetilde U_{0}^{\rad,\bot}\bigr)
-  \frac {\e^{a}} {\d \sqrt E} \partial_{\zeta_1} \bigl( \widetilde U_{a}^{\zrm,\bot} +\f'  \widetilde U_{a}^{\rad,\bot}\bigr)\\
&\qquad\qquad\qquad\qquad\qquad\qquad\qquad\qquad\qquad\qquad\qquad{}
-\frac {\e} {\d \sqrt E}\partial_{\zeta_1} \bigl( \widetilde U_{1}^{\zrm,\bot} +\f'  \widetilde U_{1}^{\rad,\bot}\bigr)\, .
\end{align*}
We know from\refeq {curved_BL1} that $ \widetilde U_{0}^{\zrm,\bot} +\f' \widetilde U_{0}^{\rad,\bot}=0$. On the other hand, the terms of size~$\e^{-1+2a}$
and of size~$\e^{-1+a}$ must vanish, which imposes that the condition
\begin{equation} \label{eq:U_a-BL}
 \widetilde U_{a}^{\zrm,\bot} +\f'\widetilde U_{a}^{\rad,\bot}=0
\end{equation}
must hold true as well. This in turn yields the equality
\begin{align*}
{\rm Div}^\bot_\e & = \dive_\h \bigl ( \widetilde U_{0}^{\rad,\bot} \nabla_\h \rho  (x_\h)\bigr) -\frac {\d'} \d \zeta_1\partial_{\zeta_1} \widetilde U_{0}^{\rad,\bot}
+\e^a \Bigl(\dive_\h \bigl (\widetilde U_{a}^{\rad,\bot}\nabla_\h \rho  (x_\h)\bigr) -\frac {\d'} \d \zeta_1\partial_{\zeta_1} \widetilde U_{a}^{\rad,\bot}\Bigr)\\
&{}\qquad
+\e \Bigl( \dive_\h \bigl( \widetilde U_{1}^{\rad,\bot} \nabla_\h \rho  \bigr) -\frac {\d'} \d \zeta_1\partial_{\zeta_1} \widetilde U_{1}^{\rad,\bot}\Bigr)
+ \e^{a} \partial_{\zeta_2} \bigl( \widetilde U_{1}^{\zrm,\bot} +\f' \widetilde U_{1}^{\rad,\bot}\bigr)\\
&\qquad\qquad\qquad\qquad\qquad\qquad\qquad\qquad\qquad\qquad{}
-\frac {1} {\d \sqrt {2\b}}\partial_{\zeta_1} \bigl( \widetilde U_{1}^{\zrm,\bot} +\f' \widetilde U_{1}^{\rad,\bot}\bigr)\, .
\end{align*}
The fact that  the terms of size~$0$ and~$\e^a$ must vanish gives
\beq
\label {BL_order1_eq4}
\begin {aligned}
\partial_{\zeta_1} \bigl(\widetilde U_{1}^{\zrm,\bot} +\f' \widetilde U_{1}^{\rad,\bot}\bigr) & = 
\d\sqrt {2\b} \Bigl(\dive_\h \bigl ( \widetilde U_{0}^{\rad,\bot} \nabla_\h \rho  (x_\h)\bigr) -\frac {\d'} \d \zeta_1\partial_{\zeta_1}  \widetilde U_{0}^{\rad,\bot} \Bigr)\andf\\
\partial_{\zeta_2} \bigl( \widetilde U_{1}^{\zrm,\bot} +\f' \widetilde U_{1}^{\rad,\bot}\bigr) &= 
- \dive_\h \bigl ( \widetilde U_{a}^{\rad,\bot}\nabla_\h \rho  (x_\h)\bigr) + \frac {\d'} \d \zeta_1\partial_{\zeta_1} \widetilde U_{a}^{\rad,\bot}\, .
\end {aligned}
\eeq
Observe that the components of $ \widetilde U_{0}^{\bot}$ are already known  {thanks to~(\ref{flat_BL2_cut-off_0})}. Our goal consists in solving the previous equations in order to find the precise expressions
of $ \widetilde U_{a}^{\bot}$ and $\widetilde U_{1}^{\bot}$.

Let us focus on the first equation appearing in \eqref{BL_order1_eq4}. The point is to write the term
$$
\dive_\h \bigl ( \widetilde U_{0}^{\rad,\bot} \nabla_\h \rho  (x_\h)\bigr) -\frac {\d'} \d \zeta_1\partial_{\zeta_1} \widetilde U_{0}^{\rad,\bot}
$$
as the sum of a function of~$\zeta_1$ and~$\zeta_2$, times the horizontal divergence of a vector field  of the type $g(\rho) \nabla_\h \rho  $,
plus the derivative of a function of~$\zeta$ which vanishes both at $-\infty$  and at~$0$. 
Let us  observe that, as~$ \zeta_1 \partial_{\zeta _1} f(\zeta_1)= \partial_{\zeta_1} (\zeta_1 f(\zeta_1))-f(\zeta_1)$, the first equation
of\refeq   {BL_order1_eq4} becomes
\beq
\label {curved_BL5}
\partial _{\zeta _1} \bigl( \widetilde U_{1}^{\zrm,\bot} +\f ' \widetilde U_{1}^{\rho,\bot} \bigr) = \sqrt{2\b}
\left(\d   \dive_\h \bigl ( \widetilde U_{0}^{\rad,\bot} \nabla_\h \rho  (x_\h)\bigr) +   \d' \widetilde U_{0}^{\rho,\bot}  
- \d'  \partial_{\zeta_1}   \bigl( \zeta_1 \widetilde U_{0}^{\rho,\bot} \bigr) \right)\, .
\eeq
Remark that, as~$|\nabla_\h \rho  (x_\h)|^2=1$, we can write
\[
\d'\widetilde U_{0}^{\rho,\bot}  = \d' \widetilde U_{0}^{\rho,\bot}  |\nabla_\h \rho  (x_\h)|^2 = 
 \widetilde U_{0}^{\rho,\bot} \nabla_\h\rho\cdot\nabla_\h(\d(\rho))\, .
\]
{Plugging this expression into\refeq {curved_BL5} and using \eqref{flat_BL2_cut-off_0} yield}
\begin{align*}
\partial _{\zeta _1} \bigl( \widetilde U_{1}^{\zrm,\bot} +\f '\widetilde U_{1}^{\rho,\bot} \bigr) &= 
\sqrt {2\b} \chi (\zeta_2)  e^{\zeta_1} \dive_\h \bigl(\d^{\frac 13} u^\theta_\e (t,\rho) \nabla_\h \rho  \bigr) 
 \sin\zeta_1 \\
&\qquad\qquad\qquad\qquad\qquad\qquad{}
- \sqrt {2\b} \frac {\d' } {\d^{\frac 2 3}}  u^\theta_\e(t,\rho) \chi(\zeta_2)\partial _{\zeta _1} \bigl(\zeta_1 e^{\zeta_1} \sin \zeta_1\bigr).
\end{align*}
By integration and because the functions must vanish at~{$\zeta_1=-\infty$}, we finally infer that 
\begin{align*}
 \widetilde U_{1}^{\zrm,\bot} +\f '\widetilde U_{1}^{\rho,\bot} & =  - \frac{\sqrt {2\b}}{2}   \dive_\h \bigl(\d^{\frac 13} u^\theta_\e (t,\rho) \nabla_\h \rho  \bigr)\chi(\zeta_2)e^{\zeta_1} (\cos\zeta_1 -\sin\zeta_1) \\
&\qquad\qquad\qquad\qquad\qquad\qquad{}
- \sqrt {2\b} \frac {\d' } {\d^{\frac 2 3}}  u^\theta_\e(t,\rho) \chi(\zeta_2)\zeta_1 e^{\zeta_1} \sin \zeta_1\, .
\end{align*}
In order to solve the second equation  of\refeq {BL_order1_eq4},  
we start by computing
\begin{align*}
\partial_{\z_2}\big( \widetilde U_{1}^{\zrm,\bot} +\f '\widetilde U_{1}^{\rho,\bot}\big) &=
 -\frac{\sqrt {2\b}}{2}  \dive_\h \bigl(\d^{\frac 13} u^\theta_\e (t,\rho) \nabla_\h \rho  \bigr)\chi'(\zeta_2) e^{\zeta_1} (\cos\zeta_1 -\sin\zeta_1) \\
&\qquad\qquad\qquad\qquad\qquad\qquad{}
- \sqrt {2\b} \frac {\d' } {\d^{\frac 2 3}}  u^\theta_\e(t,\rho) \chi'(\zeta_2)\zeta_1 e^{\zeta_1} \sin \zeta_1\,.
\end{align*}
Similarly as above, we can write the right-hand side as
\begin{align*}
&- \dive_\h \bigl ( \widetilde U_{a}^{\rad,\bot}\nabla_\h \rho  (x_\h)\bigr) + \frac {\d'} \d \zeta_1\partial_{\zeta_1} \widetilde U_{a}^{\rad,\bot} \\
&\qquad =
- \left(\dive_\h \bigl ( \widetilde U_{a}^{\rad,\bot}\nabla_\h \rho  (x_\h)\bigr) + \frac {\d'} \d \widetilde U_{a}^{\rad,\bot} \right)
+ \frac{\d'}{\d} \partial_{\z_1}\big(\zeta_1 \widetilde U_{a}^{\rad,\bot}\big) \\
&\qquad =
- \frac{1}{\d} \left(\d \dive_\h \bigl (\widetilde U_{a}^{\rad,\bot}\nabla_\h \rho  (x_\h)\bigr) + \d' \widetilde U_{a}^{\rad,\bot} \right)
+ \frac{\d'}{\d} \partial_{\z_1}\big(\zeta_1 \widetilde U_{a}^{\rad,\bot}\big) \\
&\qquad =
- \frac{1}{\d} \dive_\h \bigl (\d \widetilde U_{a}^{\rad,\bot}\nabla_\h \rho  (x_\h)\bigr)
+ \frac{\d'}{\d} \partial_{\z_1}\big(\zeta_1 \widetilde U_{a}^{\rad,\bot}\big)\,.
\end{align*}
Using those relations, we can write the second equation in \eqref{BL_order1_eq4} in the following way:
\begin{align*}
&-\dive_\h \bigl (\d \widetilde U_{a}^{\rad,\bot}\nabla_\h \rho  (x_\h)\bigr)
+\d'\partial_{\z_1}\big(\zeta_1 \widetilde U_{a}^{\rad,\bot}\big) \\
&\qquad\qquad\qquad\qquad = -\frac{\sqrt {2\b}}{2}\ \d \  \dive_\h \bigl(\d^{\frac 13} u^\theta_\e (t,\rho) \nabla_\h \rho  \bigr)\chi'(\zeta_2) e^{\zeta_1} (\cos\zeta_1 -\sin\zeta_1) \\
&\qquad\qquad\qquad\qquad\qquad\qquad\qquad\qquad\qquad\qquad{}
- \sqrt {2\b} \ \d^{\frac13}\ \d'  u^\theta_\e(t,\rho) \chi'(\zeta_2)\zeta_1 e^{\zeta_1} \sin \zeta_1 \\
&\qquad\qquad\qquad\qquad = -\frac{\sqrt {2\b}}{2} \dive_\h \bigl(\d^{\frac 43} u^\theta_\e (t,\rho) \nabla_\h \rho  \bigr)\chi'(\zeta_2) e^{\zeta_1} (\cos\zeta_1 -\sin\zeta_1) \\
&\qquad\qquad\qquad\qquad\qquad\qquad\qquad\qquad\qquad\qquad{}
- \frac{\sqrt {2\b}}{2} \ \d^{\frac13}\ \d'  u^\theta_\e(t,\rho) \chi'(\zeta_2)\zeta_1 e^{\zeta_1} \sin \zeta_1\,.
\end{align*}
Therefore    we set
$$
\widetilde U_{a}^{\rad,\bot} = \frac{\sqrt {2\b}}{2} \d^{\frac 1 3} u^\theta_\e(t,\rho) \chi'(\zeta_2) e^{\zeta_1} (\cos\zeta_1 -\sin\zeta_1)\,.
$$
 
All in all, recalling also \eqref{eq:U_a-BL}, with the choice
\beq
\label {curved_BL6}
\begin {aligned}
\widetilde U_{a}^{\rad,\bot} & = -\frac {\sqrt {2\b} } 2 \d^{\frac 1 3} u^\theta_\e \chi'(\zeta_2) e^{\zeta_1} (\cos\zeta_1 -\sin\zeta_1)\\
\widetilde U_{a}^{\zrm,\bot} & = \f' \frac {\sqrt {2\b} } 2 \d^{\frac 1 3} u^\theta_\e \chi'(\zeta_2) e^{\zeta_1} (\cos\zeta_1 -\sin\zeta_1)\andf\\
\widetilde U_{1}^{\zrm,\bot} +\f '\widetilde U_{1}^{\rho,\bot} &=  - \frac {\sqrt {2\b}} 2   \dive_\h \bigl(\d^{\frac 13} u^\theta_\e (t,\rho) \nabla_\h \rho  \bigr)\chi(\zeta_2)e^{\zeta_1} (\cos\zeta_1 -\sin\zeta_1) \\
&\qquad\qquad\qquad\qquad\qquad\qquad\qquad {}- \sqrt {2\b} \frac {\d' } {\d^{\frac 2 3}}  u^\theta_\e \chi(\zeta_2)\zeta_1 e^{\zeta_1} \sin \zeta_1\,,
\end  {aligned}
\eeq
we infer that 
\beq
\label {curved_BL7}
{\rm Div}^\bot_\e = 
\e \Bigl( \dive_\h \bigl( \widetilde U_{1}^{\rad,\bot} \nabla_\h \rho  \bigr) -\frac {\d'} \d \z_1 \partial_{\zeta_1} \widetilde U_{1}^{\rad,\bot}\Bigr)\,.
\eeq
Observe that the function $ \widetilde U_{1}^{\rad,\bot}$ still has to be found  {(as well as other correctors to ensure that the divergence is exactly zero)}.

\medskip 

To conclude this part, we   set
\begin{equation} \label{eq:U_a-1^theta}
\begin{aligned}
\widetilde U_{a }^{\theta,\top}  = \widetilde U_{a }^{\theta,\bot} = \widetilde U_{1 }^{\theta,\top} =  \widetilde U_{1 }^{\theta,\bot} = 0\,,
\end{aligned}
\end{equation}
so that, finally,
\begin{align*}
\widetilde U_{a }^{\top} &=\begin{pmatrix}
\displaystyle \frac {\sqrt {2\b}} 2   u^\theta_\e  \chi'(\zeta_2)e^{\zeta_1} (\cos\zeta_1-\sin \zeta_1) \\ 0\\ 0
\end{pmatrix}  
\andf \\
\widetilde U_{a }^{\bot} &=\begin{pmatrix}
\displaystyle -\frac {\sqrt {2\b} } 2 \d^{\frac 1 3} u^\theta_\e \chi'(\zeta_2) e^{\zeta_1} (\cos\zeta_1 -\sin\zeta_1) \\ 0\\\displaystyle  \f' \frac {\sqrt {2\b} } 2 \d^{\frac 1 3} u^\theta_\e \chi'(\zeta_2) e^{\zeta_1} (\cos\zeta_1 -\sin\zeta_1)
\end{pmatrix}
\end{align*}
and, as for~$\widetilde U_{1 }^{\bot} $ and~$\widetilde U_{1 }^{\bot} $, the~$\theta$ components vanish, $ \widetilde U_{1}^{\zrm,\surf}$ is given by~(\ref{BL_order1_eq3}), and~$ \widetilde U_{1}^{\zrm,\bot}$ and~$ \widetilde U_{1}^{\rho,\bot}$
   are related by~(\ref{curved_BL6}). 
Returning to the original variables,
   \begin{equation} \label{eq:def Uasurfbot}
   \begin{aligned}
\widetilde U_{a ,\rm BL}^{\top} &=  u^\theta_\e \begin{pmatrix}
\displaystyle \frac {\sqrt {2\b}} 2  \chi'  \biggl( -\frac z {\e^{1-a}}\biggr)e^{\frac z {\sqrt E}} (\cos\frac z {\sqrt E}-\sin \frac z {\sqrt E}) \\ 0\\ 0
\end{pmatrix}   \quad \mbox{and}
 \\\widetilde U_{a  ,\rm BL}^{\bot} &= u^\theta_\e\begin{pmatrix}
\displaystyle -\frac {\sqrt {2\b} } 2 \d^{\frac 1 3} \chi' \biggl( \frac {z+\f} {\e^{1-a}}\biggr) e^{- \frac {z+\f } {\d \sqrt E} } \Big(\cos\big( \frac {z+\f } {\d \sqrt E} \big) +\sin\big( \frac {z+\f } {\d \sqrt E}\big) \Big) \\ 0\\\displaystyle  \f' \frac {\sqrt {2\b} } 2 \d^{\frac 1 3} \chi'\biggl( \frac {z+\f} {\e^{1-a}}\biggr) e^{- \frac {z+\f } {\d \sqrt E} } \Big(\cos\big( \frac {z+\f } {\d \sqrt E}\big)  +\sin\big( \frac {z+\f } {\d \sqrt E} \big)\Big)
\end{pmatrix}\, .
\end{aligned}
\end{equation}

\begin{remark} \label{r:exp-small}
{\sl It follows from the definitions given in \eqref{curved_BL6}  that the terms~$\wt U^\surf_{a,{\rm BL}}$ and~$\wt U^\bot_{a,{\rm BL}}$
satisfy the Dirichlet boundary condition, owing to Proposition\refer {Dirichlet_shore}. Moreover, these terms are exponentially small as well as all their derivatives,
as
\beq
\label {exp_small_BL_a}
e^{\frac z  {\sqrt E} }\biggl | \chi'\biggl (-\frac z {\e^{1-a}}\biggr)\biggr |\leq C e^{-\frac {\sqrt {2\b}} {2}  \e^{-a}}\andf
e^{-\frac {z+\f}  {\d\sqrt E} }\biggl |  \chi'\biggl (\frac {z+\f}  {\e^{1-a}}\biggr)\biggr | \leq C e^{-\frac {\sqrt {2\b}} {2(1+\|\d\|_{L^\infty})}  \e^{-a}}.
\eeq
}
\end{remark}


\section {The Dirichlet boundary condition at order~$1$ and interior terms of order~$1$ }
\label{s:Dir-1}

 The previous section led the divergence to be small thanks to new correctors, but now
  the Dirichlet boundary condition is no longer satisfied. Indeed, Relations \eqref{BL_order1_eq3} and\refeq {curved_BL6} applied on the boundary give
\beq
\label {curved_BL6b}
\begin{aligned}
&  \qquad\qquad\quad  \wt U_{1,\rm BL}^{\zrm ,\top} (t, \rho, 0) 
 = \frac{ \sqrt {2\b}}{2} \dive_\h \bigl (u^\theta_\e(t,\rho)  \nabla_\h \rho \bigr)    \andf \\
& \wt U_{1,\rm BL} ^{\zrm,\bot}(t,\rho,-\f(\rho)) +\f ' \wt U_{1,\rm BL}^{\rho ,\bot} (t,\rho,-\f(\rho)) 
=  - \frac {\sqrt {2\b}} 2  \dive_\h \bigl(\d^{\frac 13} u^\theta_\e (t,\rho) \nabla_\h \rho  \bigr).
\end{aligned}
\eeq
The boundary condition at order~$1$ will be ensured by the introduction of terms of order~$1$ at the interior, namely~$\e \wt U_{1,\rmint,\e}  $ and $\e \wt  P_{1,\rmint,\e}$,
which we are now going to compute.

At first glance, it seems natural to look for~$\wt  U_{1,\rmint,\e}$ as a function of~$(t,\rho,z)$. Of course, the vector field~$\wt  U_{1,\rmint,\e}$  {must be} divergence free, which imposes, because of Formula\refeq  {div_rotating_frame_coodinates},
 \beno
 \dive \wt U_{1,\rmint,\e} & = & \dive_\h  \wt U^\h_{1,\rmint,\e} +\partial_z \wt U^\zrm_{1,\rmint,\e}\\
 & = & \partial_\rho \wt U^\h_{1,\rmint,\e} + \wt U^\h_{1,\rmint,\e}\D_\h\rho  +
\partial_z \wt U^\zrm_{1,\rmint,\e}.
 \eeno
{However}, the remark after Equation\refeq {deriv_nabla_rho}  points out the fact that, except in the  particular radial case when~$\rho(x_\h)=|x_\h|$, the laplacian of~$\rho$ is never a function of~$\rho$. Thus, we  {look} for the vector field~$\wt U_{1,\rmint,\e}$ of the form
\[
 \wt U_{1,\rmint,\e} (t,x_\h,z) = \begin{pmatrix}    \wt U_{1,\rmint,\e}^\rho(t,\rho(x_\h),z) \\   \wt U_{1,\rmint,\e}^\theta(t,\rho(x_\h),z) \\    \wt U_{1,\rmint,\e}^\zrm(t,x_\h,z) \end{pmatrix}
\]
in the basis $\big(\nabla_\h\rho,\nabla_\h^\perp\rho,e_\zrm\big)$, that is, only the ``radial'' and ``azimuthal'' components, $ \wt U_{1,\rmint,\e}^\rho$ and~$ \wt U_{1,\rmint,\e}^\theta$ respectively,
are assumed to depend on $\rho$, whereas the vertical component $ \wt U_{1,\rmint,\e}^\zrm$ may depend on~$x_\h$ in a free way.
 {Imposing the divergence free constraint over $\wt U_{1,\rmint,\e}$ gives the relation}
\begin{equation} \label{eq:int-div-U_1}
\divh\big(  \wt U_{1,\rmint,\e}^\rho \nabla_\h\rho\big) + \partial_\zrm  \wt U_{1,\rmint,\e}^\zrm = 0\,. 
\end{equation}
On the other hand, we notice that the gradient of $\wt P_{1,\rmint,\e}$ writes as $$\nabla \wt P_{1,\rmint,\e}= \partial_\rho \wt P_{1,\rmint,\e} \nabla_\h\rho + \partial_\zrm\wt P_{1,\rmint,\e} e_\zrm\, .$$
We deduce that the terms of order $0$ in the interior must satisfy the equations
\begin{equation} \label{eq:int-P_1-U_1-rho}
\partial_\rho \wt P_{1,\rmint,\e} -  \wt U_{1,\rmint,\e}^\theta = 0
\end{equation}
on the $\nabla_\h\rho$ component (where we have used that $\big(\nabla_\h^\perp\rho\big)^\perp = - \nabla_\h\rho$) and
\begin{equation} \label{eq:int-U_1-theta}
\partial_tu^\theta_\e +  \wt U_{1,\rmint,\e}^\rho = 0
\end{equation}
on the $\nabla_\h^\perp\rho$ component. We see that those two equations are independent one from the other. 
At the same time, the $\zrm$ component of the system reduces to the equation $\partial_\zrm\wt P_{1,\rmint,\e}=0$, thus~$\wt P_{1,\rmint,\e}=\wt P_{1,\rmint,\e}(t,\rho)$.
Therefore, in \eqref{eq:int-P_1-U_1-rho}
we make the simple choice $ \wt U_{1,\rmint,\e}^\theta = 0$, that is
\beq \label{P1intzero}
\wt P_{1,\rmint,\e} = 0\, .
\eeq
In Equation \eqref{eq:int-U_1-theta}, instead, we use the fact that $u^\theta_\e= u^\theta_\e(t,\rho)$, see \eqref{def:u^theta_e},
to deduce that~$\partial_\zrm  \wt U_{1,\rmint,\e}^\rho = 0$, hence $ \wt U_{1,\rmint,\e}^\rho =  \wt U_{1,\rmint,\e}^\rho (t,\rho)$.
Since this function must be equal to $ {-{ \wt U}^{\rho,\surf}_{1,\rm BL}}_{|z=0}$
 at the surface and to~${ -{ \wt U}_{1,\rm BL}^{\rad,\bot}}_{|z=-\f}$ at the bottom
in order to enforce the Dirichlet boundary condition,
in turn we get the fundamental equality
\[   {\ { \wt U}^{\rho,\surf}_{1,\rm BL}}_{|z=0}= { \ { \wt U}_{1,\rm BL}^{\rad,\bot}}_{|z=-\f}
 \,.
\]
Using also Relation \eqref{eq:int-div-U_1} and the first equality in \eqref{curved_BL6b}, we finally infer the form of $ \wt  U_{1,\rmint,\e} $, namely
\beq
\label  {curved_BL8}
\begin{aligned}
 \wt  U_{1,\rmint,\e} (t,x_\h,z)&= \begin{pmatrix} -     \wt  U_{1,\rm BL}^{\rad,\surf}(t,\rho,0) \\ 0\\ - \wt   U_{1,\rm BL}^{\zrm,\surf}(t,\rho,0) 
+z\dive_\h \bigl (  \wt  U_{1,\rm BL}^{\rad,\surf}(t,\rho,0) \nabla_\h \rho \bigr) \end{pmatrix}  \\
&= \begin{pmatrix} -  \wt  U_{1,\rm BL}^{\rad,\bot}(t,\rho,-\f)\\ 0\\ \ds -\frac{\sqrt {2\b}}{2}  \dive_\h \bigl (u^\theta_\e  \nabla_\h \rho\bigr) +
z\dive_\h \bigl (   \wt U_{1,\rm BL}^{\rad,\surf}(t,\rho,0) \nabla_\h \rho \bigr)\end{pmatrix}\, ,
\end{aligned}
\eeq
where we recall that the expression is given in the basis $\big(\nabla_\h\rho,\nabla_\h^\perp\rho,e_\zrm\big)$ of $\R^3$.
Observe that, obviously, we have
$$
{\wt U}^\rho_{1,\rmint}(t,\rho,0)= {\wt U}^{\rho}_{1,\rmint}(t,\rho,-\f(\rho))= -{\ {\wt U}_{1,\rm BL}^{\rad,\surf}}_{|z=0}=- {\ {\wt U}_{1,\rm BL}^{\rad,\bot}}_{|z=-\f}\,,
$$
and also~$\wt U^{\zrm}_{1,\rmint}(t,\rho,0)=- \wt U_{1}^{\zrm,\surf}(t,\rho,0)$. To ensure the full Dirichlet boundary conditions, 
we must have $\wt U^{\zrm}_{1,\rmint}(t,\rho,-\phi) = -\wt  U_{1}^{\zrm,\bot}(t,\rho,0)$, which gives
\beq
\label  {curved_BL7a}
\begin {aligned}
\wt U_{1}^{\zrm,\bot}(t,\rho,0) & = \frac{\sqrt {2\b}}{2}  \dive_\h \bigl (u^\theta_\e  \nabla_\h \rho\bigr)   +
\f(\rho)  \dive_\h \bigl (\wt  U_{1}^{\rad,\surf}(t,\rho,0) \nabla_\h \rho \bigr)\, .
\end {aligned}
\eeq
Plugging this expression in the third equation of\refeq  {curved_BL6} yields
$$
\f(\rho)\  \dive_\h \bigl( {\wt U}_{1}^{\rad,\bot}(t,\rho,0) \nabla_\h \rho  \bigr) + \f ' {\wt U}_{1}^{\rho ,\bot} (t,\rho,0) =
-\frac{\sqrt {2\b}}{2}  \dive_\h \bigl( (1+\d^{\frac 13}) u^\theta_\e(t,\rho) \nabla_\h\rho  \bigr)\, .
$$
Arguing similarly as we did after Equation \eqref{curved_BL5}, we see that the left-hand side of the previous relation can be written as a total
divergence:
\[
 \f(\rho)\  \dive_\h \bigl(   {\wt U}_{1}^{\rad,\bot}(t,\rho,0) \nabla_\h \rho  \bigr) + \f '  {\wt U}_{1}^{\rho ,\bot} (t,\rho,0) =
 \divh\big(\phi(\rho) \  {\wt U}_{1}^{\rho,\bot}(t,\rho,0)\ \nabla_\h\rho\big).
\]
This implies that 
\beq
\label  {curved_BL9}
 {\wt U}_{1}^{\rad,\surf}(t,\rho,0)=  {\wt U}_{1}^{\rad,\bot}(t,\rho,0) = 
 {- \lambda_\f(\rho)\, u^\theta_\e(t,\rho)\,,}
\eeq
where the function $\lambda_\f$ has been defined in \eqref{form of the limit}. We recall its definition here:
\begin{equation}\label{form of the limit-bis}
  \lambda_\f(r)    \eqdefa    \frac {\sqrt {2\b}}   {\f(r)} \,\frac {1+\sqrt[4]{1+\f '^2(r)}  } {2} \,\cdotp
\end{equation}
In the end, using also \eqref{curved_BL8}, we see that the vector field $ \wt  U_{1,\rmint,\e} $ is given,
in the coordinate framework $\big(\nabla_\h\rho, \nabla_\h^\perp\rho,e_\zrm\big)$, by the formula
\beq
\label {curved_BL10}
 \wt  U_{1,\rmint,\e}  (t,x_\h,z)=\begin{pmatrix} 
 {\lambda_\f(\rho)} u^\theta_\e (t,\rho)    \\ 0\\ \ds -\frac{\sqrt {2\b}}{2}  \dive_\h \bigl (u^\theta_\e(t,\rho)  \nabla_\h \rho \bigr) -z\dive_\h \biggl(
{\lambda_\f(\rho)} u^\theta_\e (t,\rho) \nabla_\h \rho    \biggr) \end{pmatrix} .
\eeq
This point is very important, because it allows us to determine the function~$u^\theta$, hence $u^\theta_\e$.
Indeed, inserting the expression for the ``radial'' component (\tsl{i.e.} the component along $\nabla_\h\rho$) into Equation \eqref{eq:int-U_1-theta}
and recalling the definition of $u^\theta_\e$ from \eqref{def:u^theta_e}, we find
\beq 
\label  {curved_BL12}
\partial_t u^\theta + 
{\lambda_\f(\rho)} u^\theta =0\,,
\quad\hbox{which gives}\quad 
u^\theta (t,\rho) =  e^{-t {\lambda_\f(\rho)} }\ u_{0}^\theta(\rho)\,,
\eeq
in agreement with Formula \eqref{form of the limit-bis}.

\medbreak
Before going further, let us sum up the formulas which we have found for the terms (both boundary layer and interior terms) of order $1$. In doing so, we shall also introduce
a suitable cut-off function, whose importance will clearly appear in the next section.

Let us consider a function~$k$ in~$\cD(]-2,0])$ such that
\beq
\label {cond_BL_withoutR}
k(\z)=1\quad\hbox{on $[-1,0]$} \quad \andf \quad  \int_{\R^-} k(\zeta) d\zeta=\int_{\R^-} \zeta k'(\zeta) d\zeta
=0.
\eeq
Then, we define the following vector fields, expressed in the system of coordinates related to the basis $\big(\nabla_\h\rho,\nabla_\h^\perp\rho,e_\zrm\big)$:
\beq 
\label  {curved_BL11}
\begin {aligned}
\widetilde U_{1,\rm BL}^{\top} 
& = \begin{pmatrix} \ds-{\lambda_\f(\rho)} u^\theta_\e   k\Bigl(\frac z {\sqrt E}\Bigr)\\ 0\\ \ds - \frac {\sqrt {2\b}} 2   \dive_\h \bigl (u^\theta_\e  \nabla_\h \rho\bigr) \chi\biggl( \frac {z+\f} {\e^{1-a}}\biggr)  e^{\frac z {\sqrt E}} \Bigl(\cos\frac z {\sqrt E} -\sin\frac z {\sqrt E} \Bigr) \biggr)\end{pmatrix}\, ,  \quad \mbox{and}
  \\
\widetilde U_{1,\rm BL}^{\bot} 
& = \begin{pmatrix} \ds-{\lambda_\f(\rho)} u^\theta_\e   k\Big(- \frac {z+\f } {\d \sqrt E} \Big)\\ 0\\ \ds - \frac {\sqrt {2\b}} 2   \dive_\h \bigl (\d^{\frac 13}u^\theta_\e  \nabla_\h \rho  \bigr) \chi\biggl( \frac {z+\f} {\e^{1-a}}\biggr) e^{- \frac {z+\f } {\d \sqrt E} } \biggl(\cos\Bigl( \frac {z+\f } {\d \sqrt E} \Bigr) +\sin\Big(\frac {z+\f } {\d \sqrt E} \Big)\bigg ) \\{}\ds\qquad\qquad + \mc U_{1}^{\bot} \Bigl(t,\rho, - \frac {z+\f } {\d \sqrt E}  , \frac {z+\f} {\e^{1-a}}\Bigr)
\end{pmatrix}
\end {aligned}
\eeq
with
\beq 
\label  {curved_BL111}
 {\mc U}_{1}^{\bot} (\zeta_1,\zeta_2) \eqdefa    u^\theta_\e     \biggl(\f' {\lambda_\f(\rho)} k(\zeta_1) -\frac {\sqrt {2\b}} 2  \frac {\d' } {\d^{\frac 2 3}}  \chi(\zeta_2)\zeta_1 e^{\zeta_1}  \sin\zeta_1 \biggr)\, .
\eeq
Now, let us define
\beq 
\label  {curved_BL14}
U_{\app,\e,1}  \eqdefa \sum_{j=0}^1 \e^j\bigl(\wt U_{j,\rmint,\e}+\wt U_{j,\rm BL}^{\surf}+\wt U_{j,\rm BL}^\bot\bigr)
+\e^a \bigl (\wt U_{a,\rm BL}^{\surf}+\wt U_{a,\rm BL}^\bot\bigr)\,,
\eeq
where~$\wt U_{0,\rmint,\e}=  \big(0,u^\theta_\e,0)$ with~$u^\theta_\e$  given by\refeq {curved_BL12},~$ \wt U_{1,\rmint,\e}$ is given by\refeq  {curved_BL10}, $\wt U_{0,\rm BL}^\surf$ and~$\wt U_{0,\rm BL}^\bot$ by\refeq {flat_BL2_cut-off},
$\wt U_{a,\rm BL}^{\surf}$  and~$\wt U_{a,\rm BL}^\bot$
 by~(\ref{eq:def Uasurfbot}), ~$\wt U_{1,\rm BL}^\surf$ by\refeq  {curved_BL11} and~$\wt U_{1,\rm BL}^\bot$ by~\refeq  {curved_BL11}-(\ref{curved_BL111}). Observe that, thanks to Proposition\refer {Dirichlet_shore}, we have
\beq 
\label  {curved_BL15}
{U_{\app,\e,1} }_{|\partial\Om_\f} =0.
\eeq
Let us also recall that, as already observed in Remark \ref{r:exp-small},  the two terms~$\wt U_{a,\rm BL}^{\surf}$ and~$\wt U_{a,\rm BL}^\bot$
are exponentially small as well as all their derivatives.


%
%

\section {The divergence term of order~$1$  and the boundary layer terms  of order~$2$}
\label{full approximate}

We remark that the vector field~$U_{\app,\e,1} $  defined in~(\ref{curved_BL14}) does not satisfy  the divergence free condition. 
Indeed, if we compute its divergence, using\refeq  {BL_order1_eq3a} and\refeq {curved_BL7} we get
\beno 
\dive U_{\app,\e,1}(u^\theta_{0,\e}) & =&  \e \bigl( \dive_\h \bigl( \widetilde U_{1}^{\rad,\surf} \nabla_\h \rho  \bigr)\bigr)_{\rm BL}+ \e 
\bigl( \dive_\h \bigl( \widetilde U_{1}^{\rad,\bot} \nabla_\h \rho  \bigr)\bigr)_{\rm BL}
-\e\frac {\d'}{\d} \bigl(\zeta\partial_\zeta \widetilde U_{1}^{\rad,\bot}\bigr)_{\rm BL}\\
& = &  -\e \divh\Bigl({\lambda_\f(\rho)} u^\theta_\e \nabla_\h\rho \Bigr)  k\biggl(\frac z {\sqrt E}\biggr)
-\e \divh\Bigl({\lambda_\f(\rho)} u^\theta_\e \nabla_\h\rho \Bigr)  k\biggl(-\frac {z+\f}  {\d\sqrt E}\biggr)
\\&&\qquad\qquad\qquad\qquad\qquad\qquad\qquad{}
+\e \frac {\d'}{\d} {\lambda_\f(\rho)} u^\theta_\e  \frac {z+\f}  {\d\sqrt E}  k'\biggl(-\frac {z+\f}  {\d\sqrt E}\biggr)\cdotp
\eeno

At this point, we introduce the   functions 
$$
K^\surf(\zeta) \eqdefa  \int_\zeta^0 k(\s)d\s\,,\quad  K_0^\bot \eqdefa K^\surf \ \andf \ K_1^\bot(\zeta)\eqdefa \int_\zeta^0
\s k'(\s) d\s\, , 
$$
where the function~$k$ is defined in\refeq {cond_BL_withoutR}. We then define  the two boundary layer terms
\beq 
\label  {curved_BL151}
\begin{aligned}
\wt U_{2}^\surf & \eqdefa  -\begin{pmatrix}
0\\0\\
 \ds\dive_\h\Bigl({\lambda_\f(\rho)} u^\theta_\e (t,\rho) \nabla_\h \rho  \Bigr)K^\surf(\zeta_1)
\end {pmatrix} \andf \\
\wt U_{2}^\bot &\eqdefa \begin{pmatrix}
0\\0\\
\ds\dive_\h\Bigl({\lambda_\f(\rho)} u^\theta_\e (t,\rho) \nabla_\h \rho  \Bigr)K_0^\bot(\zeta_1)-
\frac {\d'}{\d} {\lambda_\f(\rho)} u^\theta_\e(t,\rho) K_1^\bot(\zeta_1)
\end {pmatrix} .
\end {aligned}
\eeq

Notice that, because of Conditions\refeq {cond_BL_withoutR},~${{\wt U}_{2,\rm BL}^\surf}$ and~${{\wt U}_{2,\rm BL}^\bot}$ vanish on the boundary of~$\Om_\f$.
Therefore, we define the approximate solution $U_{\app,\e}$ as
\beq
\label {curved_BL16}
U_{\app,\e} \eqdefa  U_{\app,\e,1} +\e^2 \wt U_{2,\rm BL}^\surf+\e^2\wt U_{2,\rm BL}^\bot\,,
\eeq
with $U_{\app,\e,1}$ given by Formula \eqref{curved_BL14}.
Then, thanks to Relations\refeq   {curved_BL15} and\refeq  {curved_BL151}, we deduce that
\begin{equation} \label{eq:U-in-V_s}
 { U_{\app,\e} }_{|\partial\Om_\f} =0 \ \andf \  \dive U_{\app,\e}  =0\,.
\end{equation}
We also define
\beq
\label {Papp}
 {P_{\app,\e} \eqdefa \wt P_{0,\rm int, \e}  +  \e {\wt P}_{1,{\rm BL}}^{\bot}\,,}
\end{equation}
where, recalling~\eqref{new defp1eps},
$$
\wt P_{0,\rm int,\e}  {(t,\rho,z)}  \eqdefa -\int^\rho_0 u_\e^\theta(t,\s)\,d\s
$$
and, recalling~\eqref{P1bot}, 
$$
  {{\wt P}_{1,{\rm BL}}^{\bot} (\rho,z)}
= -\frac {\sqrt {2\b}}  2\frac {1+\f '^2} {\d^\frac53}  \f ' 
u_{  {\e}}^\theta (t,\rho)    e^{- \frac {z+\f } {\d \sqrt E} }  \Big(
 \sin \bigl ( \frac {z+\f } {\d \sqrt E} \bigr)+ \cos \bigl ( \frac {z+\f } {\d \sqrt E} \bigr)\bigr)\, .
$$

This concludes the construction of the approximate  solutions. 
The next sections are devoted to the proof of Propositions \ref{prop_U_app} and \ref{structure_NLterm}.


\section{Proof of the linear approximation result}
\label {estim_error_terms}

 In this section we prove Proposition~\ref{prop_U_app},  which consists in two types of results: on the one hand the convergence result~\eqref{prop_U_app_eq2} and the a priori bounds~\eqref{bdd gradients} and~\eqref{small BL}, and  on the other hand the fact that~$U_{\app,\e} $ satisfies approximately the linear Stokes-Coriolis equation, namely  there is a  family of functions~$(p_\e)_{0< \e\leq \e_0} $ in~$L^2_{\rm loc}(\R^+;L^2(\Omega_\f ))$ such that
$$
  \cL_\e \bigr(U_{\app,\e}  \bigr) \eqdefa \partial_tU_{\app,\e} - \e \beta \Delta U_{\app,\e} + \frac{1}{\e}   e_\zrm\wedge U_{\app,\e} 
 $$
satisfies
\beq\label{Linear equation}
\lim_{\e \to 0}\Big\| \cL_\e \bigr(U_{\app,\e}  \bigr)  + \frac 1 \e \nabla p_\e\Big\|_{L^1(\R^+; L^2(\Omega_\f ))} = 0 \, .
\eeq  
Let us start by proving the convergence of the approximate solution. 
  By construction, the components of $U_{\app,\e}$ are smooth functions over $\overline{\Om_\f}$.
Owing to the conditions in \eqref{eq:U-in-V_s}, we deduce that $\big(U_{\app,\e}\big)_{0<\e\leq\e_0}$ is a family of elements of $C^1(\R^+;\mc V_\s)$.
\subsection{Proof of the convergence result \rm(\ref{prop_U_app_eq2}) }\label{sec:convergence 2.1}
To start with, let us prove~\eqref{prop_U_app_eq2},  {that is}
$$
\lim_{\e\to 0} \bigl\| U_{\app,\e}   - \overline u\bigr\|_{L^\infty(\R^+;L^2(\Om_\f))} = 0 \, .
$$
We  recall that
$$
U_{\app,\e} =  \sum_{j=0}^2 \e^j \wt U_{j,\rmint,\e} +    \sum_{j\in \{0,a,1,2\} } \e^j\bigl(\wt U_{j,\rm BL}^{\surf}+\wt U_{j,\rm BL}^\bot \bigr)
  \, ,
$$
and~$\overline u$ is defined in~\eqref{wellprepared initial data} and \eqref{form of the limit}: $\oline u= \big(\oline u_\h(t,x_\h),0\big)$, where
$\oline u_\h = u^\theta(t,\rho) \nabla_\h^\perp\rho$ and $u^\theta$ given by \eqref{curved_BL12}, namely
\[ 
u^\theta (t,\rho) =  e^{-t\lambda_\f(\rho) }\ u_{0}^\theta(\rho)\,, \qquad \with 
\lambda_\f(\rho) =  \frac {\sqrt {2\b}}   {\f(\rho)} \,\frac {1+\sqrt[4]{1+\f '^2(\rho)}  } {2}\, \cdotp
\] 

We start by estimating the difference between the first term $\wt U_{0,\rmint,\e}$ and the asymptotic velocity profile $\oline u$.
Recall   that, by \eqref{def:u^theta_e} and~\eqref{newdefU0int2}, one has
\[ 
 \wt U_{0,\rmint,\e} (t,x_\h)= u^\theta_\e(t,\rho(x_\h))\ \nabla_\h^\perp\rho(x_\h) = 
 \Bigg(1 - \chi\Big(\frac{\phi(\rho(x_\h))}{\e^{1-a}}\Big)\Bigg)\ u^\theta(t,\rho(x_\h)) \ \nabla_\h^\perp\rho(x_\h)\,.
\] 
with $u^\theta$ defined just above.
Now we notice that,  by definition of~$\lambda_\f$ and thanks to~\eqref{Defin_Domain_0}, one has
 $$
 {\lambda_\f(\rho) =  \frac {\sqrt {2\b}}   {\f(\rho)} \,\frac {1+\sqrt[4]{1+\f '^2(\rho)}  } {2}} \geq
\frac{\sqrt{2\beta}}{\phi(\rho)} \geq \frac {\sqrt {2\b}} H \,\eqdefa \lam\,,
$$
so, for any non-negative time $t$, we can bound
\beno
\bigl \| \wt U_{0,\rmint,\e}(t)  - \overline u(t)\bigr\|^2_{L^2(\Om_\f)}  
& \leq &  e^{-2\lam t} 
\int_{\cO} \Bigl(\int_{-\f(\rho)}^0 dz\Bigr)   \chi^2 \Bigl(\frac {\f(\rho(x_\h))}   {\e^{1-a}}\Bigr) |u_0^\theta(\rho(x_\h))|^2 dx_\h \\
& \leq  & e^{-2\lam t } 
\int_{\cO} \f(\rho(x_\h))  \chi^2 \Bigl(\frac {\f(\rho(x_\h))}   {\e^{1-a}}\Bigr) |u_0^\theta(\rho(x_\h))|^2 dx_\h\\
& \leq  &  \e^{1-a}e^{-2\lam t} 
\int_{\cO} \frac {\f(\rho(x_\h))}{\e^{1-a}}  \chi^2 \Bigl(\frac {\f(\rho(x_\h))}   {\e^{1-a}}\Bigr) |u_0^\theta(\rho(x_\h))|^2 dx_\h  \,.
\eeno
We infer that,  for any non negative~$t$,   
\beq
\label {prop_U_app_demoeq1}
\bigl \|\wt U_{0,\rmint,\e}(t)  - \overline u(t)\bigr\|_{L^2(\Om_\f)}  
 \leq C  \e^{\frac {1-a} 2} e^{-\lam t} \|u^\theta_0 \circ \rho\|_{L^2  } \, ,
\eeq
whence  we immediately deduce  that 
\beq
\label {prop_U_app_demoeq1b}
\bigl \|\wt U_{0,\rmint,\e} - \overline u \bigr\|_{L^\infty(\R^+; L^2(\Om_\f))}  
 \leq C  \e^{\frac {1-a} 2}\|u_0  \|_{ {L^2(\cO) }}\, ,
\eeq
which converges to zero since~$a$ is less than~$1$.

\medbreak
Now let us turn to the other terms defining~$U_{\app,\e} $.  The estimates of the boundary layer terms on the surface are    classical, we refer for instance to Section~7.1 of~\cite{cdggbook}. We  shall perform all the estimates here for completeness, and because of the presence of the cut-off at the shore, which entails some additional difficulties. We shall rely   on the following lemma.
We notice  indeed that all the terms in the expansion are   functions of~$ u^\theta_\e$ or~$u^\theta_\e/\f$, and  of their horizontal derivatives. The following lemma provides estimates of general expressions that appear in the definitions of those terms, and we will be using it many times in the following. Its proof is postponed to   Appendix~\ref{appendix1}.
 \begin{lemme}
\label {deriv_U_0int}
{\sl
Let  $m_0$ be a given integer and let~$Q = Q(x_\h,\cdot)$ be a polynomial of degree~$m_0$ in its second variable, which writes under the form
$$
Q(x_\h,X) = \sum_{j=0}^{m_0} Q_j\big(\rho(x_\h)\big)X^j \, , 
$$ 
where the coefficients~$ Q_j(\rho)$ are smooth functions of~$\rho$, bounded as well as all their derivatives. Let us define
$$
  F_\e(t,x_\h) \eqdefa  e^{-t{\lambda_\f (\rho(x_\h) )}} \biggl(1 - \chi\Big(\frac{\phi\big(\rho(x_\h)\big)}{\e^{1-a}}\Big)\biggr)Q\Bigl( {x_\h}, \frac 1{ \f\big(\rho(x_\h)\big)}\Bigr) \, .
$$
Then for any integer~$k$,  a constant~$C_k$ exists such that, for any~$\e$ small enough, the following estimate holds true:
given any multi-index~$\alpha\in\N^2$, with $|\alpha|=k$, one has
\[
\bigl | \partial_{x_\h}^\alpha   F_{\e}(t,x_\h)  \bigr| \leq C_k e^{-\frac\lam2 t }  \e^{-(k+m_0)(1-a)} \supetage {\g\leq |\al|}{0\leq j\leq m_0} |\partial_{\rho}^{\g}Q_j(\rho)|\,,
\]
where~$\lam \eqdefa  {\sqrt {2\b}} /H$.
}
\end {lemme}} 

Our goal is to   apply Lemma~\ref{deriv_U_0int} to estimate   the remaining terms entering the definition of~$U_{\app,\e} $. 
Notice that all these terms are written in the  basis~$\big(\nabla_\h \rho(x_\h),\nabla_\h^\perp\rho(x_\h),e_{\rm z}\big)$. Since~$\rho$ is smooth and bounded as well as all its  derivatives, we can ignore the contribution of the basis vectors in the estimates (although they are not   functions of~$\rho$ only), and therefore apply Lemma~\ref{deriv_U_0int} to the components of those terms.

For the sake of completeness, let us start by considering the term~$ \wt U_{0,\rmint,\e} $. 
As recalled above, there holds
\beq\label{defU0rmint1}
 \wt U_{0,\rmint,\e} (t,x_\h)=  
 \Bigg(1 - \chi\Big(\frac{\phi(\rho(x_\h))}{\e^{1-a}}\Big)\Bigg)\ e^{-t{\lambda_\f(\rho(x_\h))} }\ u_{0}^\theta(\rho(x_\h))\ \nabla_\h^\perp\rho(x_\h)\,.
\eeq
Let us apply Lemma~\ref{deriv_U_0int}: we notice that~$ m_0=0$
and each component of~$ \wt U_{0,\rmint,\e} (t,x_\h)$  in the basis~$\big(\nabla_\h \rho(x_\h),\nabla_\h^\perp\rho(x_\h),e_{\rm z}\big)$ is of the required form,
with $$Q_0( {\rho})= \ u_{0}^\theta(\rho)\, .$$  It follows that
\beq
\label {22demoeq3ab}
\begin {aligned}
\sup_{|\alpha| = k} \bigl\| \partial_{x_\h}^\alpha\wt U_{0,\rmint,\e}(t,\cdot)\bigr\|_{L^2(\Om_\f)} 
  &\leq C e^{-\frac\lam2 t }  \e^{ -k(1-a)} \|u_0\|_{H^k(\cO)} \, , \\
\sup_{|\alpha| = k} \bigl\| \partial_{x_\h}^\alpha \wt U_{0,\rmint,\e} (t,\cdot)\bigr\|_{L^\infty(\Om_\f)} 
  &\leq C e^{-\frac\lam2 t }   \e^{ -k(1-a)}\|u_0\|_{W^{k,\infty}(\cO)}\, .
\end {aligned}
\eeq
Observe that, when $k>0$, the right-hand side of the previous estimates becomes unbounded when $\e$ approaches $0$. In particular,
when proving~\eqref{bdd gradients} in the next subsection, we will need to improve the above $L^\infty$ bound for $k=1$: as we will see, the assumption
on the structure of the initial datum, i.e. the fact that $u_0/\phi$ belongs to $L^\infty(\mc O)$, plays a crucial role.

Similarly, recalling their definition in~\eqref{flat_BL2_cut-off},~$\wt U_{0,BL}^\surf$ and~$\wt U_{0,BL}^\bot$ correspond to~$ m_0=0$, and~$Q_0$ is of the form
$$
 \chi \biggl( -\frac z {\e^{1-a}}\biggr)   \Big(e^{\frac z {\sqrt E}}  +  e^{- \frac {z+\f } {\d \sqrt E} }\Big)u_{0}^\theta(\rho)\ \nabla_\h^\perp\rho(x_\h)
\, ,
$$
multiplied by  oscillating functions of~$z /{\sqrt E}$ and~$  {(z+\f) } /{\d \sqrt E}$  at the surface and at the bottom respectively. At the surface these oscillating functions are bounded as well as  all their horizontal derivatives. In contrast, at the bottom  horizontal derivatives produce factors of the order~$1/\sqrt E$. 
Thanks to the exponential decay in~$z$ and bounding all the terms by the worst contribution (which produces a factor~$\e^{-1}$ each time  a horizontal derivative acts on an oscillating term), we infer that
\beq
\label {22demoeq3b}
\begin {aligned}
\sup_{|\alpha| = k}\Big( \bigl\| \partial_{x_\h}^\alpha\wt U_{0,BL}^\surf (t,\cdot)\bigr\|_{L^2(\Om_\f)} 
+  \bigl\| \partial_{x_\h}^\alpha\wt U_{0,BL}^\bot (t,\cdot)\bigr\|_{L^2(\Om_\f)} \Big)
  &\leq C  e^{-\frac\lam2 t } \e^{\frac 12-k} \|u_0\|_{H^k(\cO)} \, ,\\
\sup_{|\alpha| = k} \Big(\| \partial_{x_\h}^\alpha\wt U_{0,BL}^\surf (t,\cdot)\bigr\|_{L^\infty(\Om_\f)} +  \bigl\|\partial_{x_\h}^\alpha\wt U_{0,BL}^\bot (t,\cdot)\bigr\|_{L^\infty(\Om_\f)} \Big)
  &\leq C e^{-\frac\lam2 t }  \e^{-k}\|u_0\|_{W^{k,\infty}(\cO)}\, .
\end {aligned}
\eeq
Similarly, by~\eqref{eq:def Uasurfbot}, we deduce
\beq
\label {22demoeq3c}
\begin {aligned}
\sup_{|\alpha| = k} \Big(\bigl\| \partial_{x_\h}^\alpha\wt U_{a,BL}^\surf (t,\cdot)\bigr\|_{L^2(\Om_\f)} 
+   \bigl\| \partial_{x_\h}^\alpha\wt U_{a,BL}^\bot (t,\cdot)\bigr\|_{L^2(\Om_\f)} \Big)
  &\leq C e^{-\frac\lam2 t }  \e^{ {\frac 12}-k(2-a)}\|u_0\|_{H^k(\cO)}\, ,\\
\sup_{|\alpha| = k} \Big(\bigl\| \partial_{x_\h}^\alpha \wt U_{a,BL}^\surf (t,\cdot)\bigr\|_{L^\infty(\Om_\f)} +  \bigl\|\partial_{x_\h}^\alpha\wt U_{a,BL}^\bot (t,\cdot)\bigr\|_{L^\infty(\Om_\f)} 
  \Big)&\leq C e^{-\frac\lam2 t } \e^{-k(2-a)}\|u_0\|_{W^{k,\infty}(\cO)}\, .
\end {aligned}
\eeq
The term~$\wt U_{1,\rmint,\e}$, defined in~\eqref{curved_BL10}, is slightly more delicate, since the third component involves taking derivatives of~$\nabla_\h \rho(x_\h)$. But, as remarked above, these are harmless and can be ignored. So,
we find that~$m_0=2$ and
\beq
\label {22demoeq2}
\begin {aligned}
\sup_{|\alpha| = k} \bigl\| \partial_{x_\h}^\alpha \wt U_{1,\rmint,\e} (t,\cdot)\|_{ L^2(\Om_\f) } \leq Ce^{-\lam t}   \e^{ {-(k+2)(1-a)}} \|u_0\|_{H^{1+k}(\cO)} \,,\\
 \sup_{|\alpha| = k} \bigl\| \partial_{x_\h}^\alpha\wt U_{1,\rmint,\e} (t,\cdot)\|_{ L^\infty(\Om_\f )} \leq Ce^{-\lam t}  \e^{ {-(k+2)(1-a)}} \|u_0\|_{W^{1+k,\infty}(\cO)}\, .
\end {aligned}
\eeq
The same holds  for the boundary layers defined in~\eqref{curved_BL11}, so
\beq
\label {22demoeq4}
\begin {aligned}
\sup_{|\alpha| = k}\Big( \bigl\| \partial_{x_\h}^\alpha \wt U_{1,BL}^\surf (t,\cdot)\bigr\|_{L^2(\Om_\f)} + \bigl\|\partial_{x_\h}^\alpha \wt U_{1,BL}^\bot (t,\cdot)\bigr\|_{L^2(\Om_\f)} \Big)
  &\leq C e^{-\lam t}\e^{ \frac 12-(k+2)(1-a)}\|u_0\|_{H^{1+k}(\cO)}  \, , \\
\sup_{|\alpha| = k}\Big( \bigl\| \partial_{x_\h}^\alpha\wt U_{1,BL}^\surf (t,\cdot)\bigr\|_{L^\infty(\Om_\f)} + \bigl\|\partial_{x_\h}^\alpha\wt U_{1,BL}^\bot (t,\cdot)\bigr\|_{L^\infty(\Om_\f)} 
  &\leq C e^{-\lam t}\e^{-(k+2)(1-a)}\|u_0\|_{W^{1+k,\infty}(\cO)}\, .
\end {aligned}
\eeq 
Finally, we have, according to~\eqref{curved_BL151},
\beq
\label {22demoeq5}
\begin {aligned}
\sup_{|\alpha| = k}\Big( \bigl\| \partial_{x_\h}^\alpha \wt U_{2,BL}^\surf (t,\cdot)\bigr\|_{L^2(\Om_\f)} + \bigl\| \partial_{x_\h}^\alpha \wt U_{2,BL}^\bot (t,\cdot)\bigr\|_{L^2(\Om_\f)} 
  \Big)&\leq C e^{-\frac\lam2 t }   \e^{ \frac 12-(k+2)(1-a)} \|u_0\|_{H^{2+k}(\cO)} \, , \\
\sup_{|\alpha| = k}\Big( \bigl\| \partial_{x_\h}^\alpha\wt U_{2,BL}^\surf (t,\cdot)\bigr\|_{L^\infty(\Om_\f)} + \bigl\|\partial_{x_\h}^\alpha \wt U_{2,BL}^\bot (t,\cdot)\bigr\|_{L^\infty(\Om_\f)} 
 \Big) &\leq C \e^{-(k+2)(1-a)}\|u_0\|_{W^{2+k,\infty}(\cO)}\, .
\end {aligned}
\eeq 

The result~\eqref{prop_U_app_eq2} follows directly by putting together Estimate~\eqref{prop_U_app_demoeq1b} and the bounds
in~\eqref{22demoeq3b}--\eqref{22demoeq5}.

\subsection{{Proof of the bounds {\rm\eqref{bdd gradients}} and {\rm{\eqref{small BL}}}}}\label{sss:conv 2.2}

Now, let us prove the bound~\eqref{bdd gradients}.
We set
\begin{equation} \label{defUBLeps}
U_{{\rm BL},\e}\eqdefa   \sum_{j\in \{0,a,1,2\} } \e^j\bigl(\wt U_{j,\rm BL}^{\surf}+\wt U_{j,\rm BL}^\bot \bigr)
\, .
\end{equation}
Proving~\eqref{bdd gradients} boils down to proving that the two families~$\bigl(\nabla\wt U_{0,\rmint,\e}\bigr)_\e$ and~$\bigl(\e\nabla\wt U_{1,\rmint,\e}\bigr)_\e$ are  bounded in the space~$L^1(\R^+;L^\infty(\Om_\f))$. 
We start by considering the latter term, which is easier to bound.
From Definition~\eqref{curved_BL10}, we see that~$\nabla\wt U_{1,\rmint,\e}$ is linear in~$z$; so, thanks to~\eqref{22demoeq2}
and to Lemma~\ref{deriv_U_0int}, we find
\begin{equation} \label{est:D-U_1-int}
 \e\left\|\nabla \wt U_{1,\rmint,\e} \right\|_{L^1(\R^+;L^\infty(\Om_\phi))}\,\leq\,C\ \e^{1-3(1-a)} \|u_0\|_{W^{1,\infty}(\cO)}\,,
 \end{equation}
 which is bounded as soon as~$a \geq 2/3$.

We now switch to the estimate of $\nabla\wt U_{0,\rmint,\e}$. As already remarked we cannot rely on~\eqref{22demoeq3ab}, which does not provide a uniform bound
in $\e$. Instead, starting from Formula~\eqref{defU0rmint1}, we explicitly compute
\[
 \nabla_\h\wt U_{0,\rmint,\e}\,=\,\big(A_1 + A_2 + A_3 \big)\,\nabla_\h\rho\otimes\nabla_\h^\perp\rho\,+\,A_4\,\nabla_\h\nabla^\perp_\h\rho\,,
\]
where we have defined
\begin{align*}
 A_1 &\eqdefa \frac{1}{\e^{1-a}} \ \chi'\left(\frac{\phi(\rho)}{\e^{1-a}}\right) \ \phi'(\rho) \ e^{-t\lambda_\f(\rho)} u^\theta_0(\rho)\,, \\
 A_2 &\eqdefa -t  \left(1-\chi\left(\frac{\phi(\rho)}{\e^{1-a}}\right)\right) \ \ e^{-t\lambda_\f(\rho)} \ 
 \left(\lambda_\f\right)'(\rho) \ u^\theta_0(\rho)\,, \\
 A_3 &\eqdefa \left(1-\chi\left(\frac{\phi(\rho)}{\e^{1-a}}\right)\right) \ \ e^{-t\lambda_\f(\rho)}\ \left(u^\theta_0\right)'(\rho) \qquad \mbox{ and} \\
 A_4 &\eqdefa \left(1-\chi\left(\frac{\phi(\rho)}{\e^{1-a}}\right)\right) \ \ e^{-t\lambda_\f(\rho)}\ \left(u^\theta_0\right)(\rho)\,.
\end{align*}

Observe that, as done in~\eqref{22demoeq3ab} for $k=0$, one has
\[
 \left\|A_4\right\|_{L^\infty(\Om_\f)}\,\leq\,C\,e^{-\frac{\lambda}{2}t}\,\left\|u_0\right\|_{L^\infty(\mc O)}\,.
\]
Arguing analogously, one also gets
\[
 \left\|A_3\right\|_{L^\infty(\Om_\f)}\,\leq\,C\,e^{-\frac{\lambda}{2}t}\,\left\|u_0\right\|_{W^{1,\infty}(\mc O)}\,.
\]

Next, consider the term $A_1$, whose expression looks singular, at a first sight, when $\e$ goes to $0$. The key remark is that 
\[
\frac{\phi(\rho)}{\e^{1-a}}\,\leq\,1\qquad \mbox{ on the support of }\ \chi'\,.
\]
Therefore, we can write
\begin{align*}
A_1\, =\, \frac{\phi(\rho)}{\e^{1-a}} \ \chi'\left(\frac{\phi(\rho)}{\e^{1-a}}\right) \ \phi'(\rho) \ e^{-t\lambda_\f(\rho)}\
 \frac{u^\theta_0(\rho)}{\phi(\rho)}\,\virgp
\end{align*}
which implies the bound
\[
 \left\|A_1\right\|_{L^\infty(\Om_\f)}\,\leq\,C\,e^{-\frac{\lambda}{2}t}\,\left\|\frac{u_0}{\phi}\right\|_{L^\infty(\mc O)}\,.
\]

Finally, let us focus on the term $A_2$. By recalling the definition of $\lambda_\f$ given in~\eqref{form of the limit}
and explicitly computing its derivative, we see that the term $A_2$ can be written under the form
\begin{align*}
A_2 =  e^{-t\lambda_\f(\rho)} \left(f_1(\rho)\ \frac{t}{\phi(\rho)} + f_2(\rho) \ \frac{t}{\big(\phi(\rho)\big)^2}\right) \ u^\theta_0(\rho)\,,
\end{align*}
where the functions $f_1$ and $f_2$ only depend on $\rho$ and are bounded.
Let us deal only with the term involving $f_2$, the other one actually being simpler. We can decompose it as
\begin{align*}
 e^{-t\lambda_\f(\rho)} \ f_2(\rho) \ \frac{t}{\big(\phi(\rho)\big)^2} \ u^\theta_0(\rho) &= 
e^{-t\frac{\lambda_\f(\rho)}{2}}\ e^{-t\frac{\phi(\rho)\,\lambda_\f(\rho)}{2\phi(\rho)}} \ \frac{t}{\phi(\rho)} \ f_2(\rho) \ \frac{u^\theta_0(\rho)}{\phi(\rho)}\,\cdotp
\end{align*}
Remarking that the function $\phi\,\lambda_\f$ belongs to $L^\infty(\mc O)$, with
\[
\phi(\rho)\,\lambda_\f(\rho)\,\geq\,\sqrt{2\,\beta}\,,
\]
and using the boundedness of the function $\alpha\mapsto e^{-c \alpha}\alpha$ over $\R^+$ when $c>0$, we deduce
\begin{align*}
\left| e^{-t\lambda_\f(\rho)} \ f_2(\rho) \ \frac{t}{\big(\phi(\rho)\big)^2} \ u^\theta_0(\rho)\right|\,\leq\,C\,\left|\frac{u_0(\rho)}{\phi(\rho)}\right|
\,e^{-\frac{\lambda}{2}t}\,.
\end{align*}
This implies the following uniform bound for $A_2$:
\[
 \left\|A_2\right\|_{L^\infty(\Om_\f)}\,\leq\,C\,e^{-\frac{\lambda}{2}t}\,\left\|\frac{u_0}{\phi}\right\|_{L^\infty(\mc O)}\,.
\]

We now collect the bounds for $A_1\ldots A_4$ and use the fact that the function $\rho$ is smooth and bounded with all its derivatives.
After noticing that $\partial_\zrm \wt U_{0,\rmint,\e} \equiv 0$, we finally gather
\beq\label{boundnablaU0}
\big\|
\nabla \wt U_{0,\rmint,\e} 
\big\|_{L^1(\R^+;L^\infty(\Om_\phi))} \leq C\,\left(\left\|u_0\right\|_{W^{1,\infty}(\mc O)}\,+\,\left\|\frac{u_0}{\phi}\right\|_{L^\infty(\mc O)}\right)\,.
\eeq

Together with the bound in~\eqref{est:D-U_1-int}, Estimate~\eqref{boundnablaU0} concludes the proof of the property claimed in~\eqref{bdd gradients}.

\medskip

Finally let us prove the bound~\rm\eqref{small BL}. Recall that we need to prove that
$$
\bigl\|  \sqrt {d_\f} U_{{\rm BL},\e} \bigr\|_{L^\infty(\R^+; L^\infty_\h L^2_\v(\Om_\f))}  
\leq C\e\|u_0\|_{ L^\infty(\Omega_\f)}+C\e^2\|u_0\|_{ {W^{1,\infty}}(\Omega_\f)}\,,
$$
 where~$d_\f$ is defined by
$
d_\f(x_\h, z) \eqdefa \min\{ -z, \f(\rho(x_\h))+z\}$
and~$U_{{\rm BL},\e} $ is defined in~\eqref{defUBLeps}.
Let us start by considering~$\wt U_{0,\rm BL}^\surf+\wt U_{0,\rm BL}^\bot$,
the components of which are defined
  in~\eqref {flat_BL2_cut-off} and can be bounded by
$$e^{- \lambda t}\Bigl(e^{\frac z {\sqrt E}} + e^{-\frac {z+\f} {\sqrt E}}\Bigr) |u^\theta_0(\rho)|\,.$$
Then, an immediate calculation gives
$$
\bigl\|  \sqrt {d_\f} (  \wt U_{0,\rm BL}^\surf+\wt U_{0,\rm BL}^\bot) \bigr\|_{L^\infty(\R^+; L^\infty_\h L^2_\v(\Om_\f))}  
\leq C\e\|u_0\|_{L^{\infty}(\cO)}\,.
$$
Similarly,   Lemma~\ref{deriv_U_0int} provides
$$
\bigl\|\e  \sqrt {d_\f} (  \wt U_{1,\rm BL}^\surf+\wt U_{1,\rm BL}^\bot) \bigr\|_{L^\infty(\R^+; L^\infty_\h L^2_\v(\Om_\f))}  
\leq C\e^2\|u_0\|_{W^{1,\infty}(\cO)}\,.
$$
The estimate for $ \wt U_{2,\rm BL}^\surf+\wt U_{2,\rm BL}^\bot$ is similar, so~\eqref{small BL} is proved.

\subsection{The linear equation} \label{sec:linear}

To conclude the proof of Proposition \ref{prop_U_app}, we now prove~\eqref{Linear equation}.  By construction (see in particular the computations of Section~\ref{interior0} and~\ref{s:Dir-1}), there holds
    \[
\partial_t\wt U_{0,\rmint,\e} + \frac{1}{\e} e_\zrm\wedge\Big(\wt U_{0,\rmint,\e} + \e \wt U_{1,\rmint,\e}\Big) + \frac{1}{\e}\nabla \wt P_{0,\rm int,\e} = 0\,.
\]
In addition, estimates~\eqref{22demoeq3ab} and~\eqref{22demoeq2}  imply that
$$
\begin{aligned}
\Big\| \e\beta \Delta\Big( \wt U_{0,\rmint,\e} + \e \wt U_{1,\rmint,\e}\Big)\Big\|_{L^1(\R^+; L^2(\Omega_\f ))}
& \leq C  \e^{1 + 2(a-1)} \|u_0\|_{H^2(\cO)} +  \e^{2 + 4(a-1)} \|u_0\|_{H^3(\cO)}  \\
& \leq C \e^{2a-1} \|u_0\|_{H^3(\cO)}\,,
\end{aligned}
$$
since~$2a-1\leq 4a-2$.
Note that~$ \wt P_{0,\rm int,\e} $ is defined in~\eqref{new defp1eps} by
$$
\wt P_{0,\rm int,\e} (t,\rho,z) \eqdefa -\int^\rho_0 u_\e^\theta(t,\s)\,d\s
$$
and   belongs to~$L^2_{\rm loc}(\R^+;L^2(\Omega_\f ))$ as required.
Finally, because of Definition \eqref{def:u^theta_e} and Formula~\eqref{curved_BL12},
a factor $\big(\phi(\rho)\big)^{-1}$ appears when taking
the time derivative of $u^\theta_\e$, so we deduce from Lemma~\ref{deriv_U_0int} that
$$\Big\|\partial_t\Big(\e \wt U_{1,\rmint,\e}\Big)\Big\|_{L^1(\R^+; L^2(\Omega_\f ))} \leq C \e^{1-3(1-a)} \|u_0\|_{H^1(\cO)}\,.
$$
Assuming that~$a>2/3$, we find
 \beq
\label {prop_U_app_demoeq8}
\lim_{\e \to 0}\Big \|  \cL_\e \bigr(\wt U_{0,\rmint,\e} +\e  \wt U_{1,\rmint,\e}\bigr) + \frac{1}{\e}\nabla \wt P_{0,\rm int,\e} \Big\|_{L^1(\R^+;L^2(\Om_\f))}= 0 \, .
\eeq

Now let us turn to the boundary layer terms. By construction, the terms~$\wt U^\surf_{0,\rm BL}$ and~$\wt U^\bot_{0,\rm BL}$ cancel the  diffusion term with the rotation term, as seen in~\eqref{BLsurf}. So, one just needs to consider the time   derivative 
which acts, as in the case of~$ \wt U_{1,\rmint,\e}$ above, on~$u^\theta_\e$,
{making a factor $1/\f$ appear}. Using Lemma\refer {deriv_U_0int}, we get 
$$
\Big \|  \cL_\e \bigr(\wt U^\surf_{0,\rm BL}+\wt U^\bot_{0,\rm BL}\bigr) \Big\|_{L^1(\R^+;L^2(\Om_\f))}
\leq C \e^{\frac12-(1-a)}  \|u_0\|_{H^2(\cO)}\,,
$$
 hence
\beq
\label {prop_U_app_demoeq9}
\lim_{\e \to 0}\bigl \|  \cL_\e \bigr(\wt U^\surf_{0,\rm BL}+\wt U^\bot_{0,\rm BL}\bigr) \bigr\|_{L^1(\R^+;L^2(\Om_\f))}=0 \, .
\eeq 
Concerning the boundary layers of order $a$ and higher, they were not inserted in~$\cL_\e$ before: they were introduced to recover either the divergence free condition, or   the boundary conditions. One must therefore compute the action of each of the three parts of~$\cL_\e$. We shall not write the details here, as actually,
 {for $j\in\{a,1,2\}$}, the terms $\e^j \bigl(\wt U^\surf_{j,\rm BL}+\wt U^\bot_{j,\rm BL}\bigr) $ satisfy the same bounds
as~$\wt U^\surf_{0,\rm BL}+\wt U^\bot_{0,\rm BL}$.  So, in the end we find
\beq\label{est:L_BL}
\lim_{\e \to 0}\Big \|  \cL_\e\sum_{j\in\{0,a,1,2\}} \bigr(\wt U^\surf_{j,\rm BL}+\wt U^\bot_{j,\rm BL}\bigr) \Big\|_{L^1(\R^+;L^2(\Om_\f))} = 0 \, . 
\eeq

 \section {Proof of Proposition\refer {structure_NLterm}.}
\label {s:error-nonlin}

In order to understand the structure of the non linear term, let us decompose it in the following way. We denote 
$$
\wt U_{0,BL}\eqdefa \wt U^\surf _{0,BL}+\wt U^\bot_{0,BL}.
$$
 Let us write that 
\beq
\label {structure_NLterm_demoeq1}
\begin {aligned}
U_{\app,\e}\cdot\nabla U_{\app,\e} &= \wt U_{0,\rmint,\e} \cdot\nabla \wt U_{0,\rmint,\e}+ \wt U_{0,BL} \cdot\nabla \wt U_{0,BL} +\sum_{j=1}^3 \cQ_j\with\\
\cQ_1&  \eqdefa \wt U_{0,\rmint,\e} \cdot\nabla \bigl(U_{\app,\e}-\wt U_{0,\rmint,\e}\bigr)\,,\\
\cQ_2&  \eqdefa \bigl(U_{\app,\e}- \wt U_{0,\rmint,\e}\bigr)\cdot\nabla \bigl(U_{\app,\e}- \wt U_{0,BL}\bigr)\andf\\
\cQ_3&  \eqdefa \bigl(U_{\app,\e}- \wt U_{0,\rmint,\e}- \wt U_{0,BL} \bigr) \cdot\nabla \wt U_{0,BL}\,.
\end {aligned}
\eeq
With similar computations to those   leading to Formula\refeq {deriv_nabla_rho}, we get the  following relations:
\beq
\label {derive_repere_mobile}
\begin {aligned}
&\nabla_\h^\perp \rho\cdot \nabla_\h (\nabla_\h^\perp \rho ) = -\D_\h\rho \nabla_\h\rho\,,\ \nabla_\h^\perp \rho\cdot \nabla_\h (\nabla_\h \rho ) = \D_\h\rho \nabla^\perp_\h\rho
\andf\\
&\qquad\qquad\qquad\quad \nabla_\h \rho\cdot\nabla_\h  \nabla_\h\rho = \nabla_\h \rho\cdot\nabla_\h  \nabla^\perp_\h\rho =0\, .
\end  {aligned}
\eeq
They imply that
$$
\wt U_{0,\rmint,\e} \cdot\nabla \wt U_{0,\rmint,\e} = -\bigl(u_\e^\theta(t,\rho)\bigr)^2 \D_\h\rho\nabla_\h\rho.
$$
As the family of functions~$\bigl(u_\e^\theta(t,\rho)\bigr)_\e $ is bounded  in~$L^1(\R^+;L^\infty(\Om_\f))$ and tends to~$u^\theta(t,\rho)$ in~$L^\infty(\R^+;L^2(\Om_\f))$, we infer that 
\beq
\label {structure_NLterm_demoeq2}
\lim_{\e\rightarrow0 } \bigl\| \wt U_{0,\rmint,\e} \cdot\nabla \wt U_{0,\rmint,\e}+\bigl(u^\theta(t,\rho)\bigr)^2 \D_\h\rho\nabla_\h\rho\bigr\|_{L^1(\R^+;L^2(\Om_\f))} =0.
\eeq

Now let us treat the term~$\wt U_{0,BL} \cdot\nabla \wt U_{0,BL}$. As the support of the boundary layer on the surface and the support of the boundary layer at the bottom are disjoint, we get
\beq
\label {structure_NLterm_demoeq3}
\wt U_{0,BL} \cdot\nabla \wt U_{0,BL}= \wt U^\surf_{0,BL} \cdot\nabla \wt U^\surf_{0,BL}+\wt U^\bot_{0,BL} \cdot\nabla \wt U^\bot_{0,BL}
\eeq
Let us first state the following lemma, which covers the case of the surface also (it is enough to apply it with $\f\equiv 0$). It claims that the non-linear term~$\wt U_{0,BL} \cdot\nabla \wt U_{0,BL}$ does not create terms of  higher order. Its proof is postponed to  Appendix~\ref{appendix 2}.
\begin{lemme}
\label {l:non-linear}
{\sl
Let us consider five functions~$(M^\rho ,M ^\theta,N^\rho,N^\theta,N^\zrm)$ on~$[\rho_0,\infty[\times \R^-\times \R^+$. Let us define
$$
M _{\rm BL}\eqdefa M^\rho _{\rm BL} \nabla_\h\rho +M ^\theta_{\rm BL} \nabla^\perp_\h\rho-\f' M^\rho _{\rm BL} e_\zrm\andf 
 N_{\rm BL}\eqdefa N^\rho_{\rm BL} \nabla_\h\rho +N^\theta_{\rm BL} \nabla^\perp_\h\rho+  {N^\zrm_{\rm BL} e_\zrm}\,,
$$
with,  according to\refeq  {defin_BLtilde},
$$
 a_{\rm BL} (t,x_\h,z)\eqdefa a \biggl(\rho(x_\h), -\frac {z+\f(\rho(x_\h)) } {\d\sqrt E}, \frac {z+\f(\rho(x_\h)) } {\e^{1-a}} \biggr)
$$
Then we have
$$
\longformule{
M _{\rm BL}\cdot\nabla  N_{\rm BL} 
= \biggl(M^\rho  \Bigl(\partial_\rho N^{\rho}-\frac {\d'} \d \zeta_1\partial_{\zeta_1} N^{\rho}\Bigr) 
-M ^{\theta} N^{\theta}\D_\h\rho\bigg)_{\rm BL}  \nabla_\h\rho }
{  {}
+\biggl(M^\rho  \Bigl(\partial_\rho N^{\theta}-\frac {\d'} \d \zeta_1\partial_{\zeta_1} N^{\theta}\Bigr) 
- {M ^{\theta} N^{\rho}\D_\h\rho}\bigg)_{\rm BL}
 \nabla^\perp_\h\rho
+\biggl(M^\rho  \Bigl(\partial_\rho N^{\zrm}-\frac {\d'} \d \zeta_1\partial_{\zeta_1} N^{\zrm}\Bigr) 
\bigg)_{\rm BL} e_\zrm\,.
}
$$
}
\end {lemme}

\begin {proof} [Continuation of the proof of Proposition\refer {structure_NLterm}]

By definition\refeq {flat_BL2_cut-off}, applying the above lemma with the boundary layer at order~$0$ ensures that 
$$
\bigl | \wt U^\surf_{0,BL} \cdot\nabla \wt U^\surf_{0,BL}(t,x_\h,z) \bigr |+
\bigl | \wt U^\bot_{0,BL} \cdot\nabla \wt U^\bot_{0,BL}(t,x_\h,z) \bigr | \leq C \bigl |u^\theta_\e(t,\rho)\partial_\rho u^\theta_\e(t,\rho)\bigr| \Bigl(e^{\frac z {\sqrt E}} + e^{-\frac {z+\f}  {\sqrt E}}\Bigr)\,.
$$
Let us recall the definition of~$u^\theta_\e$, which is 
$$
u^\theta_\e(t,\rho) \eqdefa \left(1 - \chi\Big(\frac{\phi(\rho)}{\e^{1-a}}\Big)\right) e^{-t {\lambda_\f(\rho) }} \ u_0^\theta (\rho)\, .
$$
Then, we have
$$
\longformule{
\frac 12 e^{2t{\lambda_\f(\rho)} } \partial_\rho \bigl(u^\theta_\e(t,\rho)\bigr)^2  = \frac{\phi'}{\e^{1-a}}\chi' \Big(\frac{\phi}{\e^{1-a}}\Bigr)\left(1 - \chi\Big(\frac{\phi}{\e^{1-a}}\Big)\right) \bigl(u_0^\theta \bigr)^2}
{ {}
+\left(1 - \chi\Big(\frac{\phi}{\e^{1-a}}\Big)\right)^2 \biggl( u_0^\theta \partial_\rho u_0^\theta
-\bigl(u_0^\theta\bigr)^2 \biggl( \frac t \f \partial_\rho ( {\f \lambda_\f(\rho)}) - \frac {t\f'}{\f} {\lambda_\f(\rho)} \biggr)\biggr)\,\cdotp
}
$$
As~$xe^{-\lam x}$ is bounded, we infer that 
$$
\frac 12 \partial_\rho \bigl(u^\theta_\e(t,\rho)\bigr)^2 \leq C e^{-\lam t} \e^{a-1}\bigl(u_0^\theta \bigr)^2 +
\bigl |  u_0^\theta \partial_\rho u_0^\theta\bigr|\, .
$$
Thus,  we deduce the following bound:
$$
\longformule{
\bigl | \wt U^\surf_{0,BL} \cdot\nabla \wt U^\surf_{0,BL}(t,x_\h,z) \bigr |+
\bigl | \wt U^\bot_{0,BL} \cdot\nabla \wt U^\bot_{0,BL}(t,x_\h,z) \bigr | 
}
{ {}
\leq  C e^{-\lam t} \e^{a-1}  \Bigl(e^{\frac z {\sqrt E}} + e^{-\frac {z+\f}  {\sqrt E}}\Bigr) \bigl(\bigl(u_0^\theta \bigr)^2 +
\bigl |  u_0^\theta \partial_\rho u_0^\theta\bigr | \bigr)\,.
}
$$
Then, by integration we get
\beq
\label {structure_NLterm_demoeq4}
\bigl \| \wt U^\surf_{0,BL} \cdot\nabla \wt U^\surf_{0,BL}\bigr\|_{L^1(\R^+; L^2(\Om_\f))} \leq C \e^{a-\frac{1}{2}}
\|u_0\|_{L^\infty(\cO)}  \|u_0\|_{H^1(\cO)} \,.
\eeq

Now let us treat the  {terms~$\cQ_1$, $\mc Q_2$ and $\mc Q_3$} of\refeq {structure_NLterm_demoeq1}. The easiest term is~$\cQ_2$. Indeed, using estimates {\refeq {22demoeq3b}}--\eqref {22demoeq5} and fixing the parameter $a$ so that $1-a$ is small enough, we get
$$
\bigl\| U_{\app,\e}- \wt U_{0,\rmint,\e}\bigr\|_{L^\infty(\R^+; L^2(\Om_\f))} \leq C 
 {\|u_0\|_{H^1(\cO)}} \e^b\,,
$$
for a suitable exponent $b>0$ depending on the fixed value of $a$.
On the other hand, using estimates \refeq {22demoeq3c}--\eqref{22demoeq5} and~\eqref{boundnablaU0}, we obtain
$$
\bigl\| \nabla \bigl(U_{\app,\e}- \wt U_{0,BL}\bigr)\bigr\|_{L^1(\R^+; L^\infty(\Om_\f))} \leq C \left( \|u_0\|_{W^{3,\infty}(\cO)} \,+\,
\left\|\frac{u_0}{\phi}\right\|_{L^\infty(\mc O)}\right) \,.
$$
We immediately infer that
\beq
\label {structure_NLterm_demoeq5}
\|\cQ_2\|_{L^1(\R^+; L^2(\Om_\f))} \leq C\e^b  \|u_0\|_{H^1(\cO)} \left( \|u_0\|_{W^{3,\infty}(\cO)} \,+\,
\left\|\frac{u_0}{\phi}\right\|_{L^\infty(\mc O)}\right)\,.
\eeq


In order to estimate~$\cQ_1$, let us observe that, for any vector field~$V$ on the form
$$
V(t,x_\h,z) = V^\rho (t,\rho(x_\h),z) \nabla_\h \rho +V^\theta (t,\rho(x_\h),z) \nabla^\perp_\h \rho+V^\zrm(t,x_\h,z) e_\zrm\,,
$$
thanks to \eqref{derive_repere_mobile} and the definition \eqref{defU0rmint1} of $\wt U_{0,\rmint,\e}$ we have
\beq
\label {2.1nl_demoeq2}
\begin {aligned}
& \bigl(\wt U_{0,\rmint,\e} \cdot\nabla V\bigr) (t,x_\h,z) \\
& = \left(1\,-\,\chi\Big(\frac{\phi(\rho(x_\h))}{\e^{1-a}}\Big)\right) e^{-t{\lambda_\f(\rho(x_\h))}} u^\theta_{0}(\rho(x_\h)) \\
&\qquad \times
\Bigl(V^\rho (t,\rho(x_\h),z)\vect B_\rho(t,x_\h,z) + V^\theta (t,\rho(x_\h),z)
\vect B_\theta(t,x_\h,z)+\nabla_\h^\perp\rho\cdot\nabla_\h V^\zrm (t,x_\h,z) e_\zrm\Bigr)\,,
\end {aligned}
\eeq
where~$\vect B_\rho$ and~$\vect B_\theta$ are smooth vector fields. {We remark that no derivatives of the radial and azimuthal components, $V^\rho$ and $V^\theta$
respectively, appear on the right-hand side, but only the horizontal derivates of the vertical component $V^\zrm$.}
Let us apply this formula in the case when~$V= \wt U^\surf_{0,BL}+ \wt U^\bot_{0,BL}$. Using Definition\refeq  {flat_BL2_cut-off} of the boundary layers at order~$0$, 
{we further remark that, in this case, $V^\zrm$ depends on $x_\h$ only through $\rho$. Thus one has $\nabla_\h^\perp\rho\cdot\nabla_\h V^\zrm = 0$, and}
we get that
$$
 \bigl | \wt U_{0,\rmint,\e} \cdot\nabla (\wt U_{0,\rm BL}^\surf+  \wt U_{0,\rm BL}^\bot\bigr) (t,x_\h,z)\bigr| \leq 
C e^{-2\lam t} |u_0(x_\h)|^2 \bigl( e^{\frac z {\sqrt E}} +e^{-\frac {z+\f}{\d\sqrt E}}\bigr).
$$
By integration, we infer that 
\beq
\label {2.1nl_demoeq3}
\bigl \| \wt U_{0,\rmint,\e} \cdot\nabla (\wt U_{0,\rm BL}^\surf+ \wt U_{0,\rm BL}^\bot)\bigr\|_{L^1(\R^+;L^2(\Om_\f))} 
\leq C  {\e^{\frac12}} \|u_0\|_{L^4(\cO)}^2 \,.
\eeq
Let us study the term~$\e  \wt U_{0,\rmint,\e} \cdot\nabla \wt U_{1,\rmint,\e}$. Using again\refeq  {2.1nl_demoeq2}, by Definition\refeq  {curved_BL10} of~$\wt U_{1,\rmint,\e}$,  {we deduce} that 
$$
\longformule{
\e  \bigl(\wt U_{0,\rmint,\e} \cdot\nabla \wt U_{1,\rmint,\e}\bigr) (t,x_\h,z) = 
\e {\lambda_\f(\rho)} e^{-2t{\lambda_\f(\rho)}} \bigl(u^\theta_{0,\e} \bigr)^2(\rho) \vect B_\rho(t,x_\h,z)
}
{ {}
+\frac {\e {\sqrt 2}} 2 e^{-t{\lambda_\f(\rho)}} u^\theta_{0,\e} (\rho) \nabla_\h^\perp\rho \cdot  \nabla_\h \dive_\h 
\biggl(\Bigl ( 1-  z {\lambda_\f(\rho)} \Bigr)u^\theta_\e(t,\rho)  \nabla_\h \rho  \biggr)  e_\zrm \,.
}
$$
As the norm of~$\nabla_\h\rho$ is equal to~$1$, we infer that, for a function~$f$  on the interval~$[\rho_0,\infty[$,
$$
\dive_\h \big(f(\rho) \nabla_\h \rho\big) = \partial_\rho f (\rho) +f(\rho)\D_\h\rho\,.
$$
We then get the relation
$$
\dive_\h 
\biggl(\Bigl ( 1-  z {\lambda_\f(\rho)} \Bigr)u^\theta_\e(t,\rho)  \nabla_\h \rho  \biggr) = 
\partial_\rho \biggl(\Bigl ( 1-  z{\lambda_\f(\rho)} \Bigr)u^\theta_\e(t,\rho) \biggr) 
+\Bigl ( 1-  z {\lambda_\f(\rho)} \Bigr)u^\theta_\e(t,\rho) \D_\h \rho  \, .
$$
Thus, we have
$$
\e  \bigl(\wt U_{0,\rmint,\e} \cdot\nabla \wt U_{1,\rmint,\e}\bigr) (t,x_\h,z) = 
\frac {\e {\sqrt 2}} 2 e^{-2t{\lambda_\f(\rho)}} \big(u^\theta_{0,\e}\bigr)^2 (\rho)
\Bigl ( 1-  z{\lambda_\f(\rho)} \Bigr) \nabla_\h^\perp \rho\cdot\nabla_\h\D_\h \rho\,.
$$
Using that~$z$ belongs to the interval~$[-\f(\rho),0]$, we infer that
$$
\e \bigl|  \bigl(\wt U_{0,\rmint,\e} \cdot\nabla \wt U_{1,\rmint,\e}\bigr) (t,x_\h,z)\bigr| \leq C \e e^{-2\lam t} 
\bigl(u_0(x_\h)\bigr)^2 \, .
$$
 Therefore, by integration, we finally deduce that 
$$
\e \bigl\| \wt U_{0,\rmint,\e} \cdot\nabla \wt U_{1,\rmint,\e}\bigr\|_{L^1(\R^+;L^2(\Om_\f))}  \leq C \e \|u_0\|_{{L^4(\mc O)}}^2 \,.
$$
We omit the proof of the estimates  of the other terms, which lead to
\begin{equation}
\begin{aligned}
\label {2.1nl_demoeq4}
 \| \cQ_1 \|_{L^1(\R^+;L^2(\Om_\f))} &\leq C \e^{\frac 12 } 
 {\left(\left\|u_0\right\|^2_{L^4(\mc O)} + \left\|u_0\right\|_{L^2(\mc O)} \left\|u_0\right\|_{W^{2,\infty}(\mc O)}\right)} \\
&  {\leq C \e^{\frac12} \left\|u_0\right\|_{L^2(\mc O)} \left\|u_0\right\|_{W^{2,\infty}(\mc O)}\,.}
\end{aligned}
 \end{equation}

Now let us estimate  the term~$\cQ_3$. By Definition\refeq  {flat_BL2_cut-off} of  the boundary layers at order~$0$ and Estimate\refeq {22demoeq3b}, we~have
\beq
\label {2.1nl_demoeq5}
\bigl \| \nabla_\h \wt U_{0,\rm BL}(t,\cdot) \bigr \|_{L^2(\Om_\f)} \leq C e^{-\frac \lam 2 t} \e^{-\frac 12} 
\left\|u_0\right\|_{H^1(\mc O)}\,.
\eeq
By derivation with respect to the vertical variable~$z$, we get 
$$
\bigl | \partial_z  \wt U_{0,\rm BL}(t,x_\h,z) \bigr |\leq \frac C \e  e^{-\lam t} |u_0(x_\h)| \bigl (e^{\frac z {\sqrt E}} + e^{-\frac {z+\f} {\sqrt E}}\bigr).
$$
By integration, we infer that
$$
\bigl \| \partial_z  \wt U_{0,\rm BL}(t,\cdot) \bigr \|_{L^2(\Om_\f)} \leq C \e^{-\frac 12}  e^{-\lam t} \|u_0\| _{L^2(\cO)}\,.
$$
Together with\refeq {2.1nl_demoeq5}, this gives
\beq
\label {2.1nl_demoeq6}
\bigl \| \nabla  \wt U_{0,\rm BL}(t,\cdot) \bigr \|_{L^1(\R^+;L^2(\Om_\f))} \leq C \e^{-\frac 12} 
 {\left\|u_0\right\|_{H^1(\mc O)}\,.}
\eeq
On the other hand, using \refeq {22demoeq3c}--\eqref{22demoeq5}, we claim that
$$
\bigl \|U_{\app,\e}- \wt U_{0,\rmint,\e}- \wt U_{0,BL} \bigr\|_{L^\infty(\R^+\times\Om_\f)}\leq \e^{2a-1} 
 {\left\|u_0\right\|_{W^{2,\infty}(\mc O)}\,.}
$$
 Putting this bound together with Inequality\refeq {2.1nl_demoeq6}  yields
$$
\|\cQ_3\|_{L^1(\R^+;L^2(\Om_\f)} \leq C\e^{2a-\frac 32}  {\left\|u_0\right\|_{H^1(\mc O)} \left\|u_0\right\|_{W^{2,\infty}(\mc O)}\,.}
$$
With the estimates\refeq {structure_NLterm_demoeq2},\refeq  {structure_NLterm_demoeq4},\refeq {structure_NLterm_demoeq5},\refeq {2.1nl_demoeq4}   and Formula\refeq  {structure_NLterm_demoeq1}, this concludes the proof of Proposition\refer   {structure_NLterm}.
\end{proof}


\appendix
\section{Proof of Lemma~\ref {deriv_U_0int}} \label{appendix1}

Our aim is to estimate spatial derivatives of
$$
  F_\e(t,x_\h) \eqdefa  e^{-t{\lambda_\f(\rho)}} \biggl(1 - \chi\Big(\frac{\phi(\rho)}{\e^{1-a}}\Big)\biggr)Q\Bigl(x_\h, \frac 1 \f\Bigr) \, ,$$
where
$$
Q(x_\h,X) = \sum_{j=0}^{m_0} Q_j\big(\rho(x_\h)\big)X^j \, .
$$ 
The proof of the lemma is based on a direct differentiation and the Leibniz rule, together with the use of the localization property
$\phi(\rho)\gtrsim \e^{1-a}$, which holds true on the support of $  F_\e$.

As $\rho$ is smooth and bounded as well as all its derivatives, and since all the functions appearing in the formula depend on $x_\h$ only through $\rho$,
we consider $\rho$ as a variable and perform differentiation only with respect to $\rho$. We can therefore write~$F_\e$ under the form
\begin{equation}\label{defFeps}
\begin{aligned}
F_\e(t,x_\h) &=   \mc F_\e\big(t,\rho,\frac 1{\f(\rho)}\big) =  e^{-t{\lambda_\f(\rho)}}  v(\rho)\widetilde {\mc F_\e}\big( \rho,\frac 1{\f(\rho)}\big)\, , \with  \\
 \widetilde {\mc F_\e}\big( \rho,\omega_1) &\eqdefa
 \sum_{j=0}^{m_0} Q_j (\rho  )\omega_1^j \quad \mbox{and} \quad  v(\rho) \eqdefa 1 - \chi\Big(\frac{\phi(\rho)}{\e^{1-a}}\Big) 
 \end{aligned}
\end{equation}
and estimate~$\partial_\rho^k \Big(\mc F_\e\big(t,\rho,\frac 1{\f(\rho)}\big) \Big)  $.
Then, the proof of the lemma is a direct consequence of the following formula:
\beq
\label {deriv_U_0int_demo_eq1}
\partial_\rho^k \Big(\mc F_\e\big(t,\rho,\frac 1{\f(\rho)}\big) \Big)  =
 e^{-t{\lambda_\f(\rho)}}  \sum_{j=0}^k   {\mc F_\e^{k,j}} \Bigl( \rho, \frac 1 {\f(\rho)}, \frac t {\f(\rho)}\Bigr)  \partial_\rho^{k-j} v(\rho)
\eeq
where each~$\mc F_\e^{k,j} = \mc F_\e^{k,j}(\rho,\om_1,\om_2)$ is a polynomial with respect to the second and third variables,
respectively of degrees at most~$j+m_0$ 
and~$ j$. The coefficients of the polynomials are smooth functions of~$\rho$ and  are bounded by  derivatives of order~1 to~$k$ of~$Q_j$.

\medbreak

Let us prove Formula \eqref{deriv_U_0int_demo_eq1} by induction.
Before starting the argument, we observe that the function
$\lambda_\f(\rho)$, defined in \eqref{form of the limit}, contains a factor $\frac{1}{\phi}$. Therefore, in order to get the precise expression
on the right-hand side of \eqref{deriv_U_0int_demo_eq1}, in the computations below we will use the following trick:
\[
\partial_\rho\big(\lambda_\f(\rho)\big)\,=\,\partial_\rho\left(\frac{\phi(\rho)}{\phi(\rho)}\lambda_\f(\rho)\right)\,=\,
\frac{1}{\phi(\rho)}\,\partial_\rho\big(\phi(\rho)\,\lambda_\f(\rho)\big)\,-\,\frac{\phi'(\rho)}{\phi(\rho)}\,\lambda_\f(\rho)\,,
\]
where we remark that $\phi(\rho)\,\lambda_\f(\rho)$ is a function of $\rho$ which contains no negative powers of $\phi(\rho)$.

With these considerations in mind, let us prove Formula \eqref{deriv_U_0int_demo_eq1} for $k=1$.
By the Leibniz formula and the chain rule, we~have
$$
\longformule{
e^{t{\lambda_\f(\rho)}}\partial_\rho \Big(  {\mc F_\e}\big(t,\rho,\frac 1{\f(\rho)}\big) \Big) = 
(\partial_\rho  \widetilde {\mc F_\e})\Bigl( \rho,\frac 1 \f\Bigr) v(\rho) - \frac {\f'} {\f^2}(\partial_{\omega_1} \widetilde {\mc F_\e})\Bigl( \rho,\frac 1 \f\Bigr) v(\rho)
}
{ {}+ \widetilde {\mc F_\e}\Bigl(\rho,\frac 1 \f\Bigr)\biggl(\Bigl(-\partial_\rho\big( {\f \lambda_\f(\rho)}\big)\frac t \f
+\f' {\lambda_\f(\rho)} \frac t \f \Bigr) v(\rho) + \partial_\rho v (\rho)  \biggr) \, .
}
$$
Then, after defining 
\begin{equation*}
\begin{aligned} 
 \mc F_\e^{1,0}( \rho,\om_1,\om_2) &\eqdefa  (\partial_\rho  \widetilde {\mc F_\e})( \rho,\om_1) \andf  \\
\mc F_\e^{1,1}( \rho,\om_1,\om_2)&\eqdefa   - \phi' \om_1^2(\partial_{\omega_1} \widetilde {\mc F_\e})( \rho,\om_1)+ \widetilde {\mc F_\e}( \rho,\om_1)  \bigl( -\partial_\rho\big( {\f \lambda_\f(\rho)}\big)\om_2 +\f' {\f \lambda_\f(\rho)} \om_1\om_2\bigr)
 \,  , 
\end{aligned}
\end {equation*}
we get Assertion\refeq  {deriv_U_0int_demo_eq1} for~$k=1$.
Next, let us assume\refeq {deriv_U_0int_demo_eq1} for some~$k\geq0$. The Leibniz formula and the chain rule imply that 
\begin{align*}
&e^{t{\lambda_\f(\rho)}}\partial_\rho^{k+1} \Big(\mc F_\e\big(t,\rho,\frac 1{\f(\rho)}\big) \Big) \\
&\quad = \sum_{j=0}^k\mc F_\e^{k,j} \Bigl( \rho, \frac 1 \f, \frac t \f\Bigr)\partial_\rho^{k-j+1}v(\rho)
\\
&\qquad+ \sum_{j=0}^k\left(-\partial_\rho\big( {\f \lambda_\f(\rho)}\big)\frac t \f
+\f'  { \lambda_\f(\rho)} \frac t \f \right)\mc F_\e^{k,j} \Bigl( \rho, \frac 1 \f, \frac t \f\Bigr) \partial_\rho^{k-j}v(\rho) \\
&\qquad+ \sum_{j=0}^k\Bigg(\partial_\rho\mc F_\e^{k,j} \Bigl( \rho, \frac 1 \f, \frac t \f\Bigr)
 { 
 -\frac{ \phi' }{\phi^2} \partial_{\omega_1}\mc F_\e^{k,j} \Bigl( \rho, \frac 1 \f, \frac t \f\Bigr)- \phi' \frac{t}{\phi} \frac{1}{\phi} \partial_{\omega_2}\mc F_\e^{k,j} \Bigl( \rho, \frac 1 \f, \frac t \f\Bigr)}\Bigg)\partial_\rho^{k-j}v(\rho) \, .
\end{align*}
Changing~$j$ into~$j+1$ in  each sum ensures\refeq {deriv_U_0int_demo_eq1} for~$k+1$, with~$\mc F_\e^{k+1,0}=\partial_\rho\mc  F_\e^{k,0}$ and, for~$j$ between~$1$ and~$k+1$,
\begin{align*}
\mc  F_\e^{k+1,j} (\rho,\om_1,\om_2) &= \mc F_\e^{k,j}(\rho,\om_1,\om_2) + \bigl( {\f \lambda_\f(\rho)} \f'\om_1\om_2 - \partial_\rho\big( {\f \lambda_\f(\rho)}\big) \om_2\bigr) \mc F_\e^{k,j-1} (r,\om_1,\om_2) \\
&  -\f'\bigl(\om_1\om_2\partial_{\om_2}
+(\om_1)^2\partial_{\om_1}\bigr)\mc F_\e^{k,j-1}(\rho,\om_1,\om_2)\, .
\end{align*}
This proves\refeq {deriv_U_0int_demo_eq1}  for any~$k\in\N$.

\medbreak

In order to prove the lemma, we use  Formula\refeq {deriv_U_0int_demo_eq1},  recalling that~$ \widetilde {\mc F_\e}$ is defined in~\eqref{defFeps}. We get 
$$
 \begin{aligned}
\partial_\rho^k \Big(\mc F_\e\big(t,\rho,\frac 1{\f(\rho)}\big) \Big)  &=e^{-t{\lambda_\f(\rho)}} 
  \sum_{j=0}^k   {\mc F_\e^{k,j}} \Bigl( \rho, \frac 1 {\f(\rho)}, \frac t {\f(\rho)}\Bigr) \partial_\rho^{k-j }v(\rho)\\
 & =e^{-t{\lambda_\f(\rho)}}  \sum_{j=0}^k \mc F_\e^{k,j} \Bigl(\rho,\frac t \f,\frac 1 \f\Bigr)    \partial_\rho^{k-j }\Bigl (1-\chi \Bigl(\frac {\f} {\e^{1-a}}  \Bigr) \Bigr) \end{aligned}
$$
But~$\ds \biggl | \partial_\rho^\ell \Bigl (1-\chi \Bigl(\frac {\f} {\e^{1-a}}  \Bigr) \Bigr)\biggr | \leq C_k \e^{-\ell(1-a)}$
 and  $\phi\gtrsim \e^{1-a}$ on the support of $\mc F_\e$. Moreover~$ {\mc F_\e^{k,j}} $ is bounded by derivatives of order~1 to~$k$ of~$Q_j$. It follows that
\begin{align*}
 \biggl | \partial_\rho^k \Big(\mc F_\e\big(t,\rho,\frac 1{\f(\rho)}\big) \Big)  \biggr|    \leq C_k e^{-t\frac {\sqrt {2\b}} {2H}}    {\e^{-(k+m_0)(1-a)}}    \supetage {\g\leq |\al|}{0\leq j\leq m_0} |\partial_{\rho}^{\g}Q_j(\rho)|\,.
\end{align*}
This estimate proves the lemma.
 \qed

\section{Proof of Lemma \ref{l:non-linear}}\label{appendix 2}

The Leibniz formula implies that 
\beq
\label {l:non-linear_demoeq1}
\begin{aligned}
M _{\rm BL}\cdot\nabla  N_{\rm BL} &=  \bigl(M _{\rm BL}\cdot\nabla N_{\rm BL}^\rho \bigr) \nabla_\h\rho + \bigl(M _{\rm BL}\cdot\nabla N_{\rm BL}^\theta \bigr) \nabla_\h^\perp\rho + \bigl(M _{\rm BL}\cdot\nabla  N_{\rm BL}^\zrm \bigr)e_\zrm \\
 &\qquad\qquad\qquad\qquad\qquad\qquad\quad{}
+  N_{\rm BL}^\rho M _{\rm BL}\cdot\nabla  \nabla_\h\rho +  N_{\rm BL}^\theta M _{\rm BL}\cdot\nabla \nabla_\h^\perp\rho\,.
\end {aligned}
\eeq
Using formulas\refeq  {derive_repere_mobile}, we infer that 
\beq
\label {l:non-linear_demoeq2}
 N_{\rm BL}^\rho M _{\rm BL}\cdot\nabla  \nabla_\h\rho +  N_{\rm BL}^\theta M _{\rm BL}\cdot\nabla \nabla_\h^\perp\rho = -\big(M ^{\theta} N^{\theta}\D_\h\rho\big)_{\rm BL}  \nabla_\h\rho 
+\big(  {M ^{\theta} N^{\rho}\D_\h\rho}\big)_{\rm BL}\nabla^\perp_\h\rho.
 \eeq
As for any function~$a$ on~$[\rho_0,\infty[\times \R^-\times \R^+$, $a_{\rm BL}$ is a function of~ $\rho(x_\h)$, we have
$$
\nabla_\h^\perp \rho \cdot \nabla_\h a_{\rm BL}=0.
$$
 Thus we infer~that 
$$
M _{\rm BL}\cdot\nabla a_{\rm BL} =  M^\rho _{\rm BL} \bigl(\nabla_\h\rho\cdot \nabla_\h
-\f'\partial_z \bigr)a_{\rm BL}.
$$
Then, using\refeq  {derrBL1} and  that~$\bigl(\nabla_\h\rho\cdot \nabla_\h
-\f'(\rho)\partial_z)\bigr) \bigl(z+\f(\rho)\bigr)=0$,  we infer that 
$$
M _{\rm BL}\cdot\nabla a_{\rm BL} = M^\rho _{\rm BL} \Bigl(\partial_\rho a -\frac {\d'} \d \zeta_1\partial_{\zeta_1} a\Bigr) _{\rm BL}
$$
Together with\refeq  {l:non-linear_demoeq2}, plugging this formula into\refeq  {l:non-linear_demoeq1} ensures the result.
\qed

\medbreak
\medbreak
\medbreak

\end{document}